\documentclass[a4paper]{article}
\usepackage[top=2.5cm,bottom=2.5cm,left=2.5cm,right=2.5cm]{geometry}
\usepackage{amsmath,amsthm,amssymb,mathrsfs}
\usepackage{enumerate}
\usepackage{graphicx}

\newtheorem{dfn}{Definition}[section]
\newtheorem{thm}[dfn]{Theorem}

\newtheorem{lem}[dfn]{Lemma}

\newtheorem{cor}[dfn]{Corollary}

\newtheorem{rem}[dfn]{Remark}
\newtheorem{prop}[dfn]{Proposition}\makeatletter

         \@addtoreset{equation}{section}
\makeatother

\newcommand{\Es}{{\rm Es}}
\newcommand{\LE}{{\rm LE}}
\usepackage{color,bm}

\newcommand{\old}[1]{{}} 

\newcommand{\cC}{{\cal B}}
\newcommand{\bfeta}{{\bm \eta}}
\newcommand{\bfzeta}{{\bm \zeta}}
\newcommand{\bfetag}{{\bm \eta}^{\rm g}}
\newcommand{\bfiota}{{\bm \iota}}
\newcommand{\bfup}{{\bm \upsilon}}
\newcommand{\bfgam}{{\bm \gamma}}
\newcommand{\bfP}{{\bf P}}
\newcommand{\bfQ}{{\bf Q}}
\newcommand{\bfM}{{\bf M}}
\newcommand{\bfN}{{\bf N}}
\newcommand{\bbfP}{{\overline{\bf P}}}
\newcommand{\bbfQ}{{\overline{\bf Q}}}
\newcommand{\bbfM}{{\overline{\bf M}}}

\newcommand{\ovg}{\overline{\gamma}}
\newcommand{\wtg}{\widetilde{\gamma}}

\newcommand{\thistheoremname}{}
\newtheorem{genericthm}[dfn]{\thistheoremname}

\begin{document}
\title{\bf {\large One-point function estimates for loop-erased random walk in three dimensions}}
\author{Xinyi Li and Daisuke Shiraishi}
\date{}

\maketitle
\begin{abstract}
In this work, we consider loop-erased random walk (LERW) in three dimensions and  give an asymptotic estimate on the one-point function for LERW and the non-intersection probability of LERW and simple random walk in three dimensions for dyadic scales. These estimates will be crucial to the characterization of the convergence of LERW to its scaling limit in natural parametrization. As a step in the proof, we also obtain a coupling of two pairs of LERW and SRW  with different starting points conditioned to avoid each other.  \end{abstract}

\section{Introduction}\label{sec:intro}
\subsection{Introduction and main results}
Loop-erased random walk (LERW) is a random simple path obtained by erasing all loops chronologically from a simple random walk path, which was originally introduced by Greg Lawler (\cite{Loop}). Since his introduction of LERW, it has been studied extensively both in mathematics and physics literature. In two dimensions, it is proved that it has a conformally invariant scaling limit, which is charaterized by Schramm-Loewner evolution (SLE) (see \cite{Sch} and \cite{LSW}). LERW also has a strong connection with other models in statistical physics, e.g. the uniform spanning tree (UST) which arises in statistical physics in conjunction with the Potts model (see \cite{Pem} and \cite{Wil} for the relation between LERW and UST). In this paper, we consider the one-point function for LERW in three dimensions, i.e., we study the probability that LERW in $\mathbb{Z}^{3}$ hits a given point and obtain an asymptotic bound with error estimate for dyadic scales.

LERW in $\mathbb{Z}^{d}$ enjoys a Gaussian behavior if $d$ is large. In fact, it is known that the scaling limit of LERW is Brownian motion (see Theorem 7.7.6 of \cite{Lawb}) for $d\geq 4$. Furthermore, the probability of LERW hitting a given point $x \in \mathbb{Z}^{d}$ (we write $p_{d}^{x}$ for this hitting probability) is of order $|x|^{2-d}$ for $d \ge 5$ and  $|x|^{-2} ( \log |x| )^{-{1}/{3} }$ for $d = 4$  assuming that LERW starts from the origin (see Section 11.5 of \cite{LawLim} for $d \ge 5$ and \cite{Law4} for $d = 4$).

On the other hand, if $d$ is small, the situation changes dramatically. In two dimensions, LERW converges to $\text{SLE}_{2}$ when the lattice spacing tends to 0 (see \cite{Sch} and \cite{LSW}). Furthermore, it is established by Rick Kenyon (\cite{Ken}) that $p_{2}^{x} \approx |x|^{-{3}/{4}}$ (the notation $\approx$ means that the logarithm of both sides are asymptotic as $|x| \to \infty$, see also \cite{Mas} for estimates on $p_{x}^{2}$). Recently, using SLE techniques, it is proved in \cite{BLV} that $p_{2}^{x} \sim c   |x|^{-{3}/{4}}$ for some constant $c$ where the notation $\sim$ means that the both sides are asymptotic. 

In contrast to other dimensions, relatively little is known for LERW in three dimensions. One crucial reason for this is that we have no nice tool like SLE to describe the LERW scaling limit (the existence of the scaling limit is proved in \cite{Koz} though). In \cite{Lawbound}, it is shown that 
\begin{equation}
c |x|^{-2 + \epsilon} \le p^{x}_{3} \le C |x|^{- \frac{4}{3}}
\end{equation}
for some $c, C, \epsilon > 0$. The existence of the critical exponent for $p^{x}_{3}$ is established in \cite{S}. Namely, it is proved that 
\begin{equation}\label{mukasi}
\mbox{there exists $\alpha \in [\frac{1}{3}, 1)$ such that }p^{x}_{3} \approx |x|^{-1-\alpha}.
\end{equation}
This allows us to show that the dimension of LERW or its scaling limit is equal to $2- \alpha$ (see \cite{S} and \cite{S2}). Numerical experiments and field-theoretical prediction suggest that $2 - \alpha = 1.62 \pm 0.01$ (see \cite{Gut}, \cite{Wil2} and \cite{WF}).

The main purpose of this paper is to improve \eqref{mukasi}. To state the main results precisely, let us introduce some notation here. Take a point $x \neq 0$ from $\mathbb{D} = \{ x \in \mathbb{R}^{3} \ | \ |x| <1 \}$ a unit open ball in $\mathbb{R}^{3}$ centered at the origin. We consider the simple random walk $S$ on $\mathbb{Z}^{3}$ started at the origin and write $T$ for the first time that $S$ exits from a ball of radius $2^{n}$ centered at $0$. Let $x_{n}$ be the one of the nearest point from $2^{n} x$ among $\mathbb{Z}^{3}$. Finally, we set 
\begin{equation}
a_{n,x} = P \Big( x_{n} \in \text{LE} (S[0, T] ) \Big)
\end{equation}
for the probability that LERW hits $x_{n}$ where $\text{LE} (\lambda )$ stands for the loop-erasure of a path $\lambda$ (see Section \ref{sec:2.2} for its precise definition). 

Now we can state the main theorem of this paper.

\begin{thm}\label{ONE.POINT}
There exist universal constants $c > 0$, $\delta > 0$ and a constant $c_{x} > 0$ depending only on $x  \in \mathbb{D} \setminus \{ 0 \}$ such that for all $n\in\mathbb{Z}^+$ and $x  \in \mathbb{D} \setminus \{ 0 \}$, 
\begin{equation}\label{one-goal}
a_{n, x} = c_{x} 2^{-(1 + \alpha ) n } \Big\{  1 + O \Big( d_{x}^{-c} 2^{- \delta  n} \Big) \Big\}  \ (\text{as } n \to \infty),
\end{equation}
where $d_{x} = \min \{ |x| , 1 - |x| \}$ and $\alpha$ is the exponent in \eqref{mukasi}. Moreover, the constant $c_{x}$ satisfies
\begin{align}
& a_{1}  |x|^{-1-\alpha } \le c_{x} \le a_{2}  |x|^{-1-\alpha } \  \ \ \ \  \   \ \ \ \ \  \ \  \Big(\text{if } 0 < |x| \le  \frac{1}{2} \Big)\\
& a_{1} (1-|x|)^{1-\alpha } \le c_{x} \le a_{2} (1-|x|)^{1-\alpha } \ \  \Big(\text{if }  \frac{1}{2} \le |x| < 1 \Big),
\end{align}
where $a_{1}, a_{2} > 0$ are universal constants.
\end{thm}

In order to prove Theorem \ref{ONE.POINT}, it turns out that we need to estimate the following non-intersection probability of simple random walk and LERW. Let $S^{1}$ and $S^{2}$ be independent simple random walks on $\mathbb{Z}^{3}$ started at the origin. We write $T^{i}_{n}$ for the first time that $S^{i}$ exits from a ball of radius $n$. We are interested in 
\begin{equation}\label{esn}
\text{Es} (n) := P \Big( \text{LE} ( S^{1} [0, T^{1}_{n} ] ) \cap S^{2} [1, T^{2}_{n} ] = \emptyset \Big)
\end{equation}
the probability that LERW $\text{LE} ( S^{1} [0, T^{1}_{n} ] )$ and simple random walk $S^{2} [1, T^{2}_{n} ] $ do not intersect (we denote this non-intersection event by $A^{n}$).  In this paper, we will show the following theorem.

\begin{thm}\label{ESCAPE.1}
There exist $c > 0$ and $\delta > 0$ such that for all $n\in\mathbb{Z}^+$,
\begin{equation}\label{GOAL.1}
\Es (2^{n} )= c 2^{- \alpha n } \big( 1 + O(2^{-\delta n} ) \big),
\end{equation}
where $\alpha$ is the exponent in \eqref{mukasi}. 
\end{thm}
This theorem immediately implies a lot of up-to-constants estimates for LERW. We summarize them in the following corollary but postpone its proof till the end of Section \ref{sec:3.2}. Write 
\begin{equation}
M_{n} = \text{len} \big( \text{LE} ( S^{1} [0, T^{1}_{n} ] ) \big)
\end{equation}
for the number of lattice steps for LERW  and let $\Es(\cdot,\cdot)$ be another escape probability defined in \eqref{escape-1-3}.
\begin{cor}\label{COR}
It follows that for $ n\geq m\geq 1$, 
\begin{align}
\Es (n) \asymp n^{-\alpha };\quad \Es (m, n) \asymp \big( \frac{m}{n} \big)^{\alpha};\quad \label{esaymp} E (M_{n} ) \asymp n^{2-\alpha };\quad \frac{M_{n}}{n^{2-\alpha}} \textrm{ is tight}.
\end{align}
\end{cor}

Before finishing this subsection, it may be worth mentioning one of the motivations of this work. It turns out that our results in this work help us characterize how 3D LERW converges to its scaling limit. In particular, it is a key ingredient for giving a natural time-parametrization of the scaling limit of 3D LERW. Some progress towards in this direction will be made in \cite{Natural}.

\subsection{Some words about the proof}\label{sec:words}
Let us here explain a sketch of proofs for the main results. We recall that $S$ is the simple random walk on $\mathbb{Z}^{3}$ started at the origin and that $T$ stands for the first time that $S$ exits from $B(n)$ a ball of radius $n$. We write $\gamma = \text{LE} (S[0, T] )$ for the loop-erasure of the simple random walk path. Take a point $x \in B(n)$. In order for $\gamma$ to hit the point $x$, the following two conditions are required: 

\vspace{2mm}

(i) $x \in S[0, T]$;     \ \ \ \ (ii)  $\text{LE} \big( S[0, \sigma_{x} ] \big) \cap S[\sigma_{x} + 1, T] = \emptyset$ 

\vspace{2mm}

\hspace{-5.5mm}where $\sigma_{x}$ stands for the last time (up to $T$) that $S$ hits $x$. Considering the time reversal of $\text{LE} \big( S[0, \sigma_{x} ] \big) $ and translating the path, we can relate the probability of the second condition (ii) to the non-intersection probability $\text{Es} (n)$ defined as in \eqref{esn}. In fact, it is known that the probability that $\gamma$ hits $x$ is comparable to $n^{-1} \text{Es} (n)$ if $x$ is not too close to the origin and the boundary of $B(n)$ (see \cite{S} for this). Thus, loosely speaking, the proof of Theorem \ref{ONE.POINT} boils down to that of Theorem \ref{ESCAPE.1}. 

We will now explain how to prove Theorem \ref{ESCAPE.1}. Since the existence of the scaling limit of LERW (we denote the scaling limit by ${\cal K}$) is already proved by Gady Kozma in \cite{Koz}, in order to estimate on $\Es (2^{n})$, it is natural to compare it with the non-intersection probability of ${\cal K}$ and a Brownian motion both started at the origin. However,  this approach unfortunately does not work without modification because there is still a non-negligible ``gap" between the simple random walk and Brownian motion as well as LERW and ${\cal K}$. The idea to deal with this issue is that somehow we separate the starting points of LERW and simple random walk wide enough so that the gap becomes negligible.

Let us here be more precise on how to rigorize the idea of separating starting points. Note that for clarity of presentation we may not use the same notation as Sections \ref{sec:noninter} through \ref{sec:length}.

\begin{itemize}
\item Let $b_{n} = \text{Es} (2^{n} )/ \text{Es} (2^{n-1} )$. It suffices to show that 
\begin{equation}\label{BOILBOIL}
b_{n} = c \Big( 1 + O \big( 2^{- \delta n} \big) \Big)
\end{equation}
for some $c > 0$ and $\delta > 0$. Note that it is proved in \cite{S} that $c_{1} \le b_{n} \le c_{2}$ for some constants $c_{1}, c_{2} > 0$.

\item For two sequences $\{ f_{n} \}$ and $\{ g_{n} \}$, we write $f_{n} \simeq g_{n}$ if $f_{n} = g_{n} \big\{ 1 + O \big( 2^{- \delta n} \big) \big\}$ for some constant $\delta > 0$.

\item Take $q \in (0, 1)$. For a path $\lambda$ and integer $k \ge 0$, we denote the first time that $\lambda$ exits from $B \big( 2^{(1-kq) n } \big) $ by $t_{k, q}$.  We set $t = t_{0,q}$.

\item We write $\gamma_{n} =\text{LE} ( S^{1} [0, T^{1}_{2^{n}} ] )$ and write $\lambda_{n} = S^{2} [1, T^{2}_{2^{n}} ]$ where $T^{i}_{m}$ stands for the first time that $S^{i}$ exits from $B(m)$. Notice that $\lambda_{n} (0) \neq 0$. Since it is proved in \cite{Mas} that the distribution of $\gamma_{n} [0, t_{1, q} ]$ is sufficiently close to that of $\gamma_{n-1} [0, t_{1, q} ]$, we have $P (F) \simeq P (F')$ and (note that $A_{n} \subset F$ and $A_{n-1} \subset F'$)
\begin{equation}\label{sk1}
b_{n}  \simeq \frac{P\big( A_{n} \ \big| \ F \big)}{P\big( A_{n-1} \ \big| \ F' \big)}
\end{equation}
where $A_{n}:=A^{2^n}$ is the event considered in $\text{Es} (2^{n})$ (see \eqref{esn}), $F$ is the event that $\gamma_{n} [0, t_{1, q} ]$ does not intersect with $\lambda_{n} [0, t_{1, q} ]$ and $F'$ is the event that $\gamma_{n-1} [0, t_{1, q} ]$ does not intersect with $\lambda_{n-1} [0, t_{1, q} ]$. See Lemma \ref{1st-lem} for more details.

\item To write $P\big( A_{n} \ \big| \ F \big)$ explicitly, we introduce the following function $g (\gamma, \lambda)$. Suppose that we have a pair of two paths $(\gamma, \lambda)$ such that they are lying in $B \big( 2^{(1-q)n} \big)$ (denote the set of such pairs by $\Gamma$). Let $X$ be a random walk started at the endpoint of $\gamma$ and conditioned that $X [1,t]$ does not intersect with $\gamma$. Also let $Y$ be the simple random walk started at the endpoint of $\lambda$. Then the function $g$ is defined by 
\begin{equation}\label{sk2}
g( \gamma , \lambda ) = P \Big( \big( \text{LE} ( X [0, t ] ) \cup \gamma \big) \cap \big( Y [0, t] \cup \lambda \big) = \emptyset \Big),
\end{equation}
for $(\gamma, \lambda ) \in \Gamma$. Note that the starting point of $\gamma$ does not necessarily coincide with that of $\lambda$ for $(\gamma, \lambda ) \in \Gamma$.

\item It follows from the domain Markov property of $\gamma_{n}$ and the strong Markov property of $\lambda_{n}$ that 
\begin{equation}\label{sk3}
P\big( A_{n} \ \big| \ F \big) = \sum_{(\gamma, \lambda ) \in \Gamma} g (\gamma , \lambda ) P \Big( \big( \gamma_{n} [0, t_{1, q} ] , \lambda_{n} [0, t_{1, q} ] \big) = (\gamma, \lambda ) \ \Big| \ F \Big).
\end{equation}
See Section \ref{sec:3.3} for rigorous arguments for this bullet and the last one.
\item The first key observation is that the function $g$ depends only on the end part of $(\gamma, \lambda)$ in the following sense. Take two elements $(\gamma, \lambda)$ and $(\gamma', \lambda')$ of $\Gamma$. Suppose that $\big( \gamma [0, t_{2, q} ], \lambda [0, t_{2, q} ]  \big)  = \big( \gamma' [0, t_{2, q} ], \lambda' [0, t_{2, q} ] \big) $. Then we have
\begin{equation}\label{sk4}
g( \gamma, \lambda ) \simeq g( \gamma', \lambda' ).
\end{equation}
Namely, $g$ depends only on $\big( \gamma [t_{2, q}, t_{1,q} ],  \lambda [t_{2, q}, t_{1,q} ] \big)$ the end part of  $(\gamma, \lambda)$. See Prop.\ \ref{prop1} for more details.

\item Why does \eqref{sk4} hold? To see this, assume that the non-intersection event considered in \eqref{sk2} occurs. This event forces $X$ and $Y$ not to return to an inner ball $B \big( 2^{(1-2q)n} \big)$ with high probability. Therefore, the initial part of $(\gamma, \lambda)$ is not important for computing $g (\gamma, \lambda)$.

\item This observation allows us to write 
\begin{equation}\label{sk5}
P\big( A_{n} \ \big| \ F \big) \simeq  E \Big\{ g \big( \gamma_{n} [t_{2, q}, t_{1, q} ] , \lambda_{n} [t_{2,q}, t_{1, q} ] \big)  \ \Big| \ F \Big\}.
\end{equation}
Namely, in order to deal with $P\big( A_{n} \ \big| \ F \big)$, we only have to control the (conditional) distribution of  the end part $\big( \gamma_{n} [t_{2, q}, t_{1, q} ] , \lambda_{n} [t_{2,q}, t_{1, q} ] \big)$ conditioned on the event $F$. With this in mind, for $(\gamma, \lambda) \in \Gamma$, let 
\begin{equation}\label{sk6}
\mu ( \gamma, \lambda ) = P \Big( \big( \gamma_{n} [t_{2, q}, t_{1, q} ] , \lambda_{n} [t_{2,q}, t_{1, q} ] \big) = (\gamma, \lambda ) \ \Big| \ F \Big) 
\end{equation}
be the probability measure on $\Gamma$ induced by $\big( \gamma_{n} [t_{2, q}, t_{1, q} ] , \lambda_{n} [t_{2,q}, t_{1, q} ] \big)$ conditioned on the event $F$.

\item Let $v = (2^{(1-3 q)n }, 0, 0 )$ and $w = -v$ be two poles of  $ B \big( 2^{(1-3 q)n} \big)$. To change the starting points of $S^{1}$ and $S^{2}$, we write $\overline{S}^{1}$ and $\overline{S}^{2}$ for the simple random walks started at $v$ and $w$. Let $\overline{T}^{i}_{r}$ be the first time that $\overline{S}^{i}$ exits from $B (r)$. We write $\overline{\gamma}_{n} = \text{LE} (\overline{S}^{1} [0, \overline{T}^{1}_{2^{n}} ])$, $\overline{\gamma}_{n-1} = \text{LE} (\overline{S}^{1} [0, \overline{T}^{1}_{2^{n-1}} ])$, $\overline{\lambda}_{n} = \overline{S}^{2} [0, \overline{T}^{2}_{2^{n}} ]$ and  $\overline{\lambda}_{n-1} = \overline{S}^{2} [0, \overline{T}^{2}_{2^{n-1}} ]$.

\item The events $\overline{A}_{n}$, $\overline{A}_{n-1}$, $\overline{F}$ and $\overline{F}'$ are defined by replacing $\gamma_{n}$, $\gamma_{n-1}$, $\lambda_{n}$ and $\lambda_{n-1}$ by $\overline{\gamma}_{n} $,  $\overline{\gamma}_{n-1} $,  $\overline{\lambda}_{n} $ and  $\overline{\lambda}_{n-1}$ respectively in the definition of $A_{n}$, $A_{n-1}$, $F$ and $F'$ defined in the fourth item. For example, $\overline{A}_{n}$ is the event that $\overline{\gamma}_{n}$ does not intersect with $\overline{\lambda}_{n}$, and $\overline{F}'$ is the event that $\overline{\gamma}_{n-1} [0, t_{1,q}] $ does not intersect with $\overline{\lambda}_{n-1} [0, t_{1,q}]$, etc. Then, by the same reason for the equation \eqref{sk1}, we have $P (\overline{F}  ) \simeq P ( \overline{F}' )$ and
\begin{equation}\label{sk7}
\frac{P (\overline{A}_{n}) }{ P (\overline{A}_{n-1} )} \simeq \frac{ P \big( \overline{A}_{n} \ \big| \ \overline{F} \big) }{P \big( \overline{A}_{n-1} \ \big| \ \overline{F}' \big)},
\end{equation}
since the distribution of $\overline{\gamma}_{n} [0, t_{1,q}]$ is close to that of $\overline{\gamma}_{n-1} [0, t_{1,q}]$ by \cite{Mas}.

\item The same ideas used to show \eqref{sk5} gives that  
\begin{equation}\label{sk8}
P \big( \overline{A}_{n} \ \big| \ \overline{F} \big) \simeq  E \Big\{ g \big( \overline{\gamma}_{n} [t_{2, q}, t_{1, q} ] , \overline{\lambda}_{n} [t_{2,q}, t_{1, q} ] \big)  \ \Big| \ \overline{F} \Big\}.
\end{equation}
As in the equation \eqref{sk6}, we define the probability measure $\nu$ by 
\begin{equation}\label{sk9}
\nu ( \gamma, \lambda ) = P \Big( \big( \overline{\gamma}_{n} [t_{2, q}, t_{1, q} ] , \overline{\lambda}_{n} [t_{2,q}, t_{1, q} ] \big) = (\gamma, \lambda ) \ \Big| \ \overline{F} \Big). 
\end{equation}

\item Here is the second key observation. It follows from some coupling technique that the total variation distance between $\mu$ and $\nu$ is small enough so that $P\big( A_{n} \ \big| \ F \big)  \simeq P \big( \overline{A}_{n} \ \big| \ \overline{F} \big)$. See bullets later in this subsection for a brief explanation why this coupling works. Rigourous arguments are wrapped up in Section \ref{sec:coupling}  with the help of a recent work \cite{Lawrecent}. This observation also gives that  $P\big( A_{n-1} \ \big| \ F' \big)  \simeq P \big( \overline{A}_{n-1} \ \big| \ \overline{F}' \big)$. Therefore, we have 
\begin{equation}\label{sk10}
b_{n} \simeq \frac{P (\overline{A}_{n}) }{ P (\overline{A}_{n-1} )}.
\end{equation}
Namely, the original starting points ($=$ the origin) are replaced by the two different poles of $B \big( 2^{(1-3q)n} \big)$.

\item This replacement of the starting points can be carried out for $b_{n-1}$, which enables us to compare $b_{n}$ and $b_{n-1}$ via multiscale analysis established in \cite{Koz}. In fact, taking $q > 0$ sufficiently small, we see that $b_{n} = b_{n-1} \big\{ 1 + O \big ( 2^{-\delta n} \big) \big\}$ for some $\delta > 0$. This gives \eqref{BOILBOIL}.
\end{itemize}

\newpage
\begin{figure}
\begin{center}
\begin{tabular}{cc}
\includegraphics[scale=0.5]{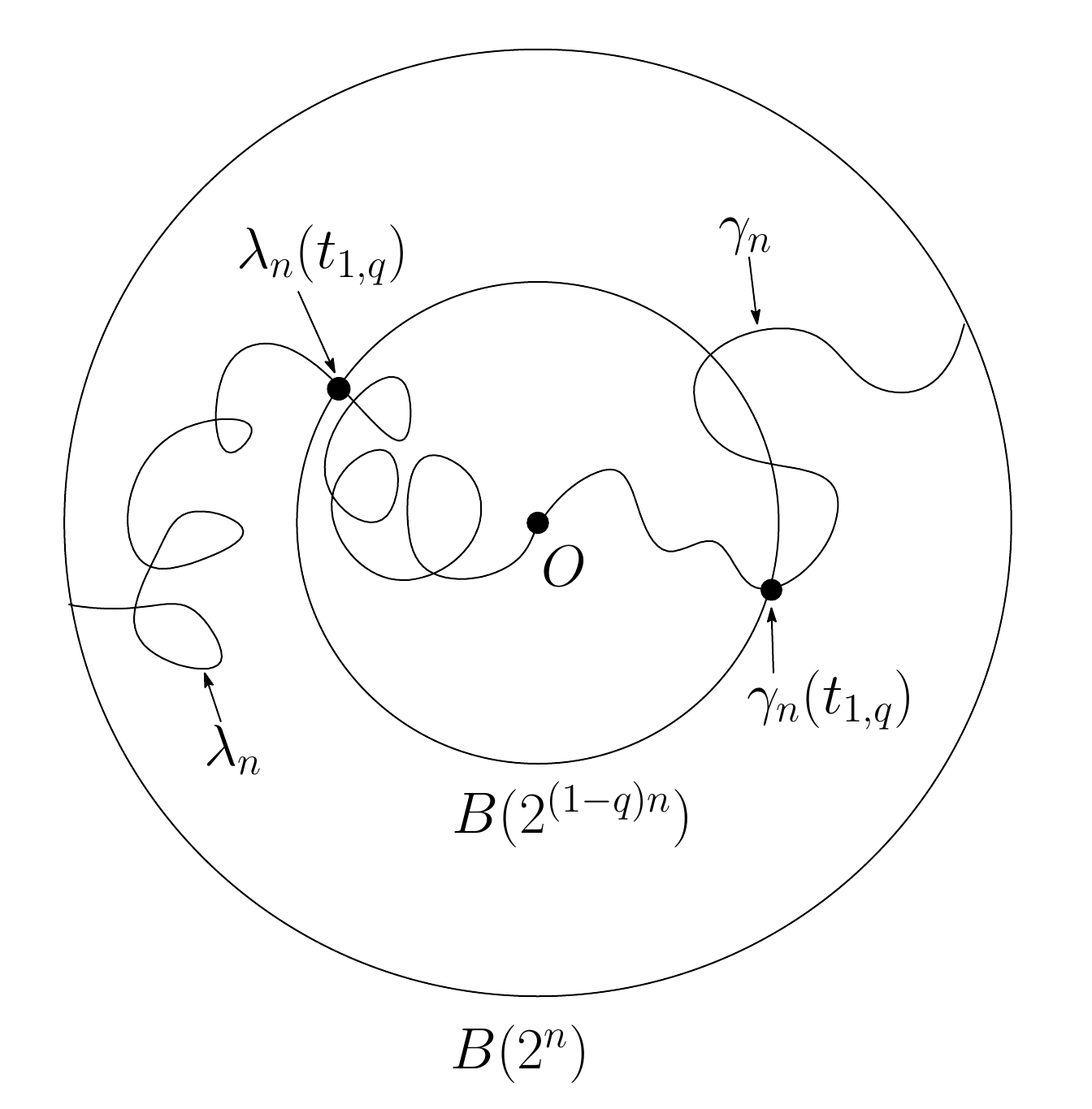} & \includegraphics[scale=0.46]{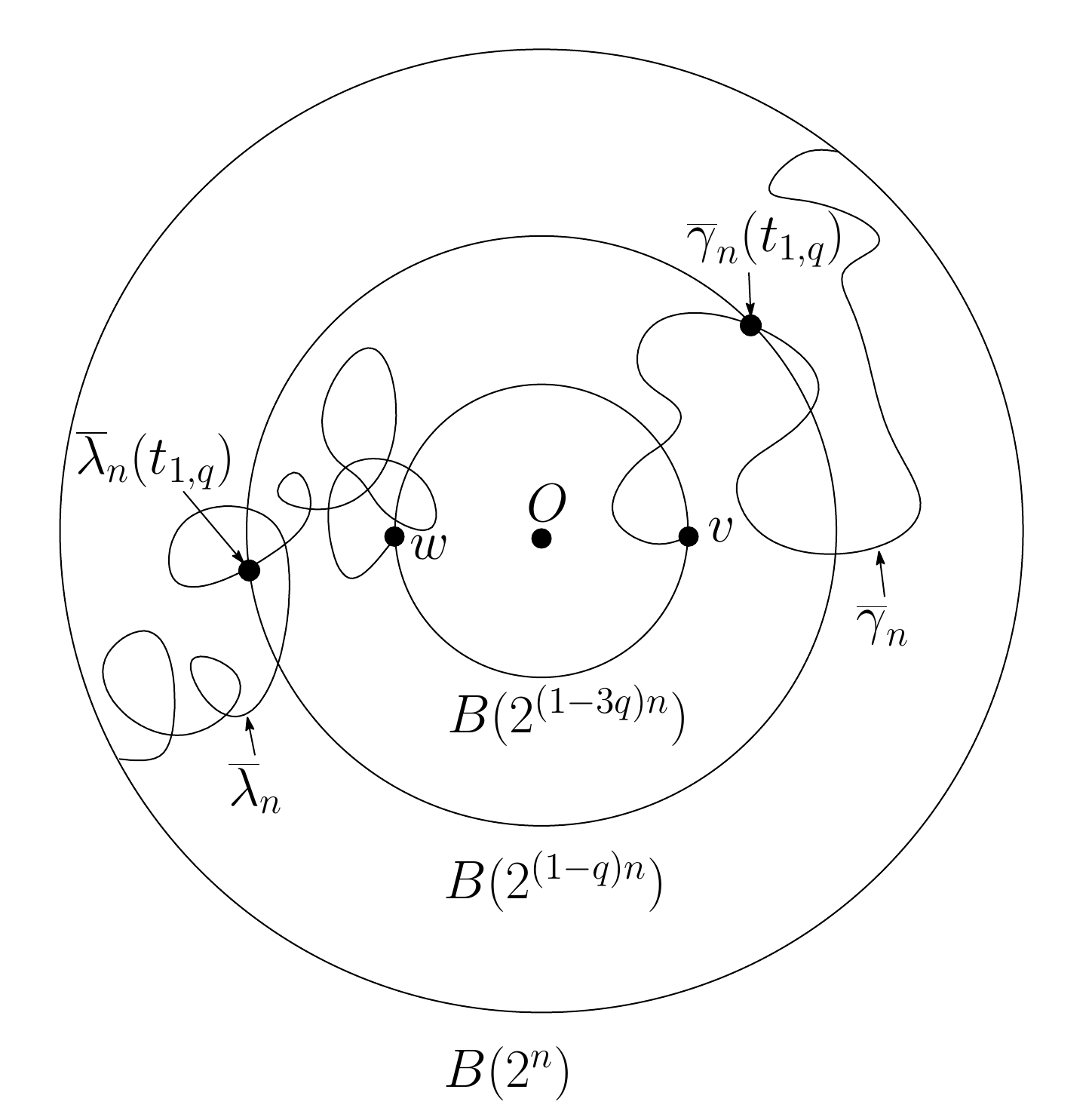}
\end{tabular}
\end{center}
\caption{Illustrations for events $F$ and $\overline{F}$.}
\end{figure}
Let us also give a brief explanation on the coupling in the second key observation above.
\begin{itemize}
\item  It is well known that LERW is not Markovian per se but can still be regarded as a Markov process if one records all of its history. In fact, consider an infinite LERW (abbreviated as ILERW later on) $\eta$ and write $\eta_n$ for $\eta[0,T_{2^n}]$, the part of the path stopped at first exiting $B(2^n)$, then one can construct a formal Markov process $(\eta_0,\eta_1,\eta_2,\ldots)$ with corresponding transition probability at each ``step''. Letting $\eta_{m,n}=\eta[T_{2^m},T_{2^n}]$ for $m\leq n$. It is also well known that LERW enjoys a weak asymptotic independence: the correlation of $\eta_k$ and $\eta_{m,n}$ decays like $O(2^{k-m})$ (or $\eta_{m,\infty}$ for ILERW). Hence, it is possible to couple two ILERW's $\eta^x$ and $\eta^y$ started from $x,y\in B(2^k)$, such that $\eta^x_{n+k,\infty}=\eta^y_{n+k,\infty}$ with probability $1-O(2^{-\beta n})$. Key observations that leads to this coupling are:
\begin{itemize}
\item[i)] 3D LERW rarely ``back-trackes'' very far. In fact, the actual configuration in $B(2^k)$ barely matters for the distribution of the path after reaching $\partial B(2^{k+m})$, if $m$ is large. Hence, it is possible to find an $m$ such that if two LERW's are coupled for $m$ steps, then the probability of getting decoupled ever is bounded by, say, $1/2$.

\item[ii)] At each step, there is a uniform positive probability for $\eta^x$ and $\eta^y$ to be coupled for $m$ steps from the next step. More precisely, for any realizations of $\eta_n^x$ and $\eta_n^y$, there exists $c>0$, such that with probability greater than $c$, $\eta_{n+1,n+m+1}^x=\eta_{n+1,n+m+1}^y$.
\item[iii)] The exponential convergence rate follows from a combination of i) and ii), by bundling every $(m+1)$ steps as a giant step.
\end{itemize} 

\item In the same spirit, it is possible to couple a pair of ILERW and SRW started from a pair of (not necessarily distinct) points inside $B(2^k)$ and conditioned to not intersect until first exit of $B(2^{2n+k})$, with another such pair, such that their paths agree from first exiting $B(2^{n+k})$ onward with probability $1-O(2^{-\beta n})$. In this case, observation i) above is still easily verifiable and observation ii) follows thanks to an auxiliary result generally known as the ``separation lemma''. For the form that satisfies our setup, see Theorem 6.1.5 of \cite{S} or Claim 3.4 of \cite{SS}.

\item As a prototype of such coupling, although with a slightly different setup, has already been proved by Greg Lawler in \cite{Lawrecent}, we will not reinvent the wheel here; instead, we are going to show that it is possible to obtain the coupling described above through ``tilting'' the coupling in \cite{Lawrecent}, as they are in fact intimately related. In \cite{Lawrecent}, Lawler considered the law of a pair of ILERW's $(\eta^1,\eta^2)$ both started from the origin and tilted its law by 
\begin{equation}\label{eq:tiltintro}
1_{\eta^1_n\cap \eta^2_n=\{0\} } \exp (-L_n(\eta^1_n,\eta^2_n)),
\end{equation}
where $L_n(\eta^1_n,\eta^2_n)$ stands for the loop term of loops in $B(2^n)$ that touch both $\eta^1_n$ and $\eta^2_n$. Then, it is shown that it is possible to couple $(\eta^1,\eta^2)$ with another pair of ILERW $(\overline{\eta}^1,\overline{\eta}^2)$ started from different initial configurations inside $B(2^k)$, and tilted similarly, such that $ \eta^i_{(n+k)/2,n}=\overline{\eta}^i_{(n+k)/2,n}$, $i=1,2$ with probability greater than $1-O(2^{-\beta (n-k)})$ for some $\beta>0$. 
\item In fact, if one decompose the SRW in our setup as LERW and loops from an independent loop soup, then the conditioning that ILERW and SRW do not intersect can be interpreted as an ILERW and a LERW do not intersect plus loops that touches both paths do not appear, which is in a way very similar to the tilting of \eqref{eq:tiltintro}, despite a few stitches in the definition for it is not trivial to deal with the replacement of ILERW by LERW. In Section \ref{sec:coupling}, we are going to deal with this issue and then add back loops to obtain the coupling we need.
\end{itemize}
\begin{figure}[htb]
\begin{center}
\includegraphics[scale=0.4]{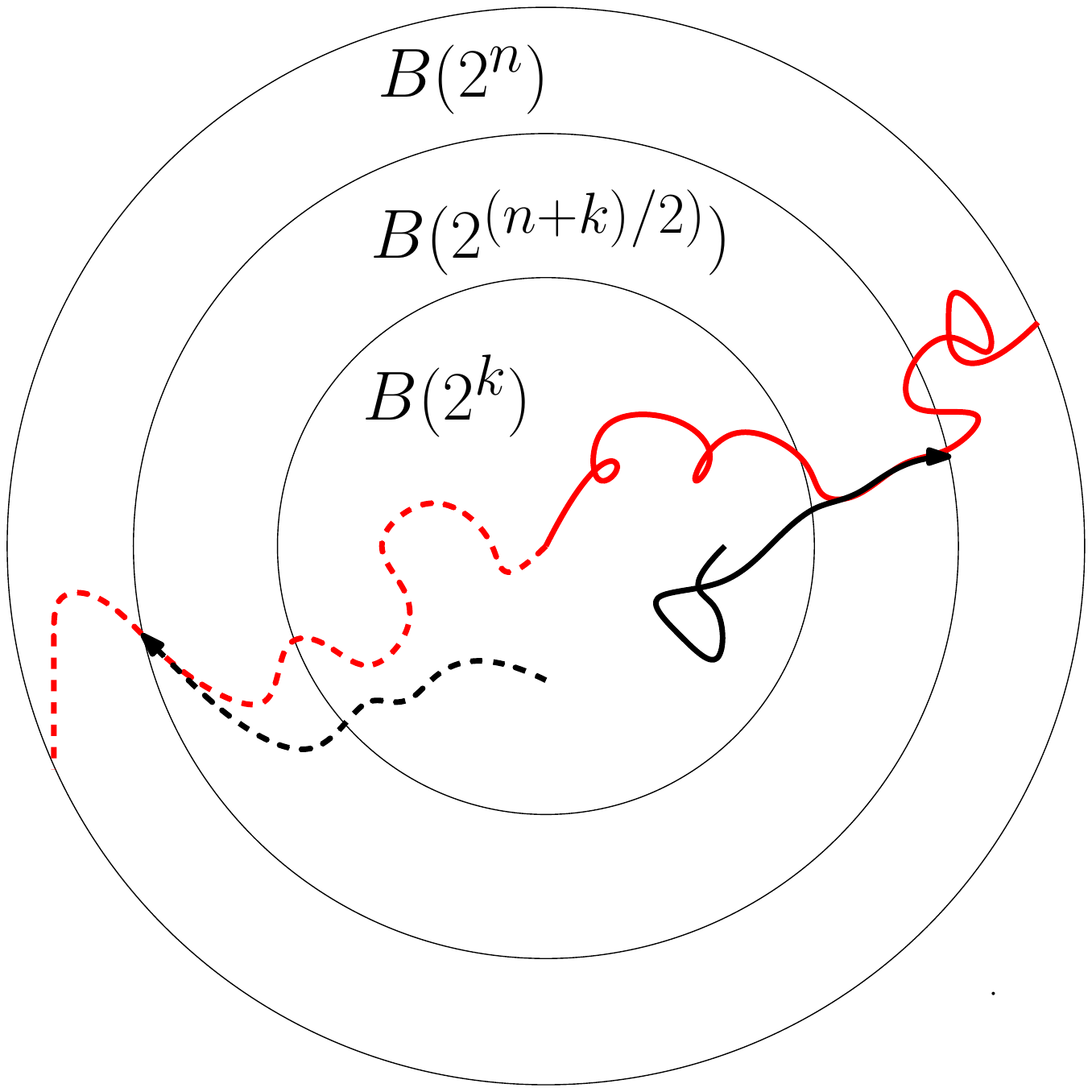}
\caption{A schematic sketch of the coupling. Dashed and continuous curves represent LERW's and SRW's respectively. Although it is impossible to couple the beginning parts of walks, we can find a coupling such that both LERW's and SRW's agree after first exiting of $B(2^{(k+n)/2})$ with high probability.}
\end{center}
\end{figure}
Finally, let us explain the structure of this paper. In Section \ref{sec:notations}, we introduce notations and discuss some basic properties of LERW. We prove Theorem \ref{ESCAPE.1} and Corollary \ref{COR} in Section \ref{sec:noninter}, assuming coupling results from the next section. Section \ref{sec:coupling} is dedicated to the discussion and proof of various couplings of two pairs of LERW and SRW conditioned to avoid each other, crucial to the proof of both main theorems. Finally, we give a proof of Theorem \ref{ONE.POINT} in Section \ref{sec:length}. As it resembles a lot the proof of Theorem \ref{ESCAPE.1}, we will be less pedagogical in the presentation.

\medskip

\noindent {\bf Acknowledgements:} The authors are grateful to Greg Lawler for numerous helpful and inspiring discussions. XL wishes to thank Kyoto University for its warm hospitality during his visit, when part of this work was conceived. The authors would also thank the generous support from the National Science Foundation via the grant DMS-1806979 for the conference ``Random Conformal Geometry and Related Fields'', where part of this work was accomplished. 

\section{Notations, conventions and a short introduction to LERW}\label{sec:notations}

In this section, we introduce some notations and conventions that we are going to use throughout the paper in Section \ref{kigou}. Then, we give a very short introduction on various properties that will be used in this paper in Sections \ref{sec:2.2} to \ref{sec:2.5}.

\subsection{Notations and conventions}\label{kigou}
In this subsection, we will give some definitions which will be used throughout the paper. In the text, we also use ``$:=$'' to denote definition.

We call $\lambda = [ \lambda (0), \lambda (1), \cdots , \lambda (m) ] $, a sequence of points in $\mathbb{Z}^3$, a path, if $|\lambda (j-1) - \lambda (j) | =1 $ for all $j=1,\ldots, m$. Let ${\rm len}(\lambda)=m$ denotes the length of $\lambda$. We call $\lambda$ a self-avoiding path (SAP) if $\lambda (i) \neq \lambda (j)$ for all $i \neq j$. For two paths $\lambda_1 [0,m_1]$ and $\lambda_2[0,m_2]$, with $\lambda_1(m_1)=
\lambda_2(0)$, we write
$$
\lambda_1 \oplus \lambda_2 := [\lambda_1(0),\lambda_1(1),\ldots,\lambda_1(m_1),\lambda_2(1),\ldots,\lambda_2(m_2)]
$$
for the concatenation of $\lambda_1$ and $\lambda_2$.

We write $| \cdot |$ for the Euclid distance in $\mathbb{R}^{3}$. For $n \ge 0$ and $z \in \mathbb{Z}^{3}$, we write 
$$
B (z, n) := \big\{  x \in \mathbb{Z}^{3} \ \big| \ |x-z| < n \big\}.
$$
If $z = 0$, we write $B(n)$ and $\cC(n)=B(2^n)$ for short. We write $\mathbb{D} = \{ x \in \mathbb{R}^{d}\, \big|\, |x| < 1 \}$ and $\overline{\mathbb{D}}$ for its closure. 

 For any path $\eta$, we write $T_{x, r}(\eta)$ for the first time that $\eta$ hits $\partial  B (x, r)$ the outer boundary of $B (x, r)$. We write  $T_{r}(\eta)$ for the case that $x=0$. Let $T^{r}(\eta) = T_{2^{r}}(\eta)$. We will drop the dependence on $\eta$ in the notation whenever there is no confusion.
 
For a subset $A \subset \mathbb{Z}^{d}$, we let $\partial A = \{ x \notin A  |  \text{ there exists } y \in A \text{ such that } |x-y| =1 \}$ 
We write $\overline{A} := A \cup \partial A$. Given a subset $A \subset \mathbb{Z}^{d}$ and $r > 0$, we write $r A := \{ ry \ | \ y \in A \}$.
 
Throughout the paper, we will use various letters, e.g.\ $S$, $S^{1}$, $S^{2}$, $R^{1}$, $R^{2}$, etc., to represent simple random walks on $\mathbb{Z}^{3}$ and will use ad-hoc notations for its probability law. Unless otherwise indicated, we use $E_\bullet$ with the same sub- and superscripts for the corresponding expectation of a probability measure $P_\bullet$.   For the probability law and the expectation of $S$ started at $z$, we use $P^{z}$ and $E^{z}$ respectively.

For a subset $A \subset \mathbb{Z}^{3}$ and $x, y \in A$, we write $$G_{A} (x, y) = E^{x} \Big( \sum_{j= 0}^{\tau - 1} {\bf 1} \{ S(j) = y \} \Big),$$ where $\tau = \inf \{ t \ \big| \ S(t) \in \partial A \}$, for Green's function in $A$.

We use $c, C', \dotsb$ to denote arbitrary positive constants which may change from line to line and use $c$ with subscripts, i.e., $c_1,c_2,\ldots$ to denote constants that stay fixed. If a constant is to depend on some other quantity, this will be made explicit. For example, if $C$ depends on $\delta$, we write $C_{\delta}$. 
  
 \subsection{Loop-erased random walk, reversibility and domain Markov property}\label{sec:2.2}
In this subsection, we will give the definition of the loop-erased random walk (LERW) and review some known facts about it, especially the time reversibility and the domain Markov property. As we are working in the case of $d=3$, we will only state things for $\mathbb{Z}^3$. 

We begin with the definition of the chronological loop-erasure of a path.
\begin{dfn}\label{procedure}
Given a path $\lambda = [\lambda (0), \lambda (1), \cdots , \lambda (m) ] \subset \mathbb{Z}^{3}$, we define its loop-erasure $\LE ( \lambda )$ as follows. Let 
\begin{equation}\label{procedure-1}
s_{0} := \max \{ t \big| \lambda (t) = \lambda (0) \},
\end{equation}
and for $i \ge 1$, let 
\begin{equation}\label{procedure-2}
s_{i} := \max \{ t \big|\lambda (t) = \lambda (s_{i-1} +1) \}.
\end{equation}
We write $n = \min \{ i \big| s_{i} = m \}$. Then we define $\LE ( \lambda ) $ by 
\begin{equation}\label{procedure-3}
\LE ( \lambda ) = [ \lambda ( s_{0} ), \lambda ( s_{1} ), \cdots , \lambda ( s_{n} ) ].
\end{equation}
If $\lambda = [\lambda (0), \lambda (1), \cdots  ] \subset \mathbb{Z}^{3}$ is an infinite path such that for each $n$, $\big|\{ k \ge n\ | \ \lambda (k) = \lambda (n) \}\big|<\infty$, then we can define $\LE ( \lambda )$ similarly.
\end{dfn}

In general, we will use the term loop-erased random walk, or LERW, loosely to refer to the (random) SAP obtained by loop-erasing from some finite SRW. However, to make things precise, we have to specify the stopping time of this SRW. A very common scenario is the following: let $D$ be a finite subset of $\mathbb{Z}^3$ and let $S^x$ be the SRW started from $x$ stopped at the first exit of $D$, when we call $\LE(S^x)$ the LERW from $x$ stopped at exiting $D$. In contrast, if $S^x$ is an infinite SRW started from $x\in\mathbb{Z}^3$, then $\LE(S^x)$ will be referred to as an {\it infinite LERW} or ILERW started from $x$.

``LERW stopped at exiting $B(n)$'' and ILERW are different stochastic objects. Moreover, the law of the former and that of the latter truncated at first exiting $B(n)$ differ greatly, especially at the ending parts. However, if we only look at beginning parts, they still look pretty similar. The following quantative lemma is excerpted from \cite{Mas}.
\begin{lem}[Corollary 4.5 of \cite{Mas}]\label{lem:LEWILEW}
Given $0\in D\subset\mathbb{Z}^3$, let $\lambda^\circ$ be an ILERW started at the origin, and $\lambda$ a LERW started from 0 stopped at exiting $D$. Let $P^\circ$ and $P$ be their respective laws. Moreover, suppose $n\geq 2$ and $l\geq 0$ satisfying $B(nl)\subset D$. Truncate $\lambda^\circ$ and $\lambda$ at first exit of $B(l)$ and denote by $\lambda_l^\circ$ and $\lambda_l$ respectively.
Then, for all $\omega\in\Gamma_l$,
$$
P^\circ[\lambda^\circ_l=\omega]=\big(1+O\big(n^{-1}\big)\big)P[\lambda_l=\omega].
$$
\end{lem}

\medskip

For a path $\lambda[0,m] \subset \mathbb{Z}^{d}$, we define its time reversal $\lambda^{R}$ by $\lambda^{\rm R} := [\lambda (m), \lambda (m-1), \cdots , \lambda (0) ]$. Note that in general, $\LE ( \lambda ) \neq (\LE ( \lambda^{\rm R} ) )^{\rm R}$. However, as next lemma shows, the time reversal of LERW has same distribution to the original LERW. Let $\Lambda_{m}$ be the set of paths of length $m$ started at the origin.

\begin{lem}[Lemma 7.2.1 of \cite{Lawb}]\label{reversal} 
For each $m \ge 0$, there exists a bijection $T^{m} : \Lambda_{m} \to \Lambda_{m}$ such that for each $\lambda \in \Lambda_{m}$, we have
\begin{equation}\label{rev-1}
\LE ( \lambda ) = (\LE ( (T^{m} \lambda )^{\rm R} ) )^{\rm R}.
\end{equation}
Moreover, we know that that $\lambda$ and $T^{m} \lambda $ visit the same edges in the same directions with the same multiplicities.
\end{lem}

\medskip

Note that LERW is not a Markov process. However it satisfies the domain Markov property in the following sense. 

\begin{lem}[Proposition 7.3.1 of {\rm \cite{Lawb}}]\label{domain} 
Let $D $ be a finite subset of $\mathbb{Z}^{3}$. Suppose that $\lambda_{i}$ ($i =1, 2$) are simple paths of length $m_{i}$ with $\lambda_{1}, \lambda_2[0,1,\ldots,m_2-1]\subset D$, $\lambda_{1} ( m_{1} ) = \lambda_{2} (0) $ and $\lambda_2(m)\in\partial D$. Let $\lambda$ be the concatenation of $\lambda_1$ and $\lambda_2$ and suppose also that $\lambda$ is a SAP. Let $Y$ be a random walk  started at $\lambda_1(m_1)=\lambda_{2} (0)$ conditioned on $Y [1, T_{D}(D) ] \cap \lambda_{1} = \emptyset$ and denote the law of $Y$ by $\overline{P}$. 
Then we have
\begin{equation}\label{domain-1}
P^{\lambda_{1} (0) } \Big( \LE ( S[0, T_{D} ] ) = \lambda\ \big| \ \LE ( S[0, T_{D} ] ) [0, m_{1} ] = \lambda_{1} \Big) = \overline{P} \Big( \LE ( Y [0, T^{Y}_{D} ] ) = \lambda_{2} \Big).
\end{equation}
\end{lem}

\subsection{Random walk loop soup and LERW}
In this subsection, we give another description of LERW through random walk loop soup measure. We refer readers to Sections 4 and 5 of \cite{Lawlernotes} for detailed discussions in this direction.

Let $\lambda$ be a loop-erased random walk from $x\in D \subseteq\mathbb{Z}^3$ stopped at exiting $D$. For generality of notation we do not require $D$ to be a finite set. For instance, if $D=\mathbb{Z}^3$, then $\lambda$ is actually an ILERW.  Let $\eta$ be a SAP in $\overline{D}$ of length $n$ such that $\eta[0,n-1]\subset D$ and write $\tau=\eta(n)$, then
$$
P\big[\lambda[0,n]=\eta\big]= 6^{-n} F_{\eta}(D) {\rm Esc}_{\eta, D}\big(\tau\big),$$ 
where 
$$
{\rm Esc}_{\eta,D}\big(\tau\big):= P^{\tau}\big[ S[1,2,\ldots]\mbox{ hits $\partial D$ before hitting $\eta$}\big]
$$
denotes the escape probability for SRW, and 
$$
F_{\eta}(D):=\prod_{j=0}^{n} G_{A_j}(\eta(j),\eta(j))\mbox{ where $A_j = D \big\backslash \eta[0,j-1]$, for $j=0,1,\ldots,n$}.
$$
Note that if $\tau\in\partial D$, then ${\rm Esc}_{\eta,D}(\tau)=1$.
Let $m$ denote the (unrooted) random walk loop measure defined in Section 5 of \cite{Lawlernotes}.  Then,
\begin{equation}\label{eq:Fetadef}
 F_{\eta}(D)= \exp\left\{\sum_{l\subseteq D,\; l \cap \eta \neq \emptyset}  m(l)\right\}.
\end{equation}
This description may seem mysterious to readers unfamiliar with the subject, but what it actually does is nothing more than weighting each SAP by the total weight of all SRW paths whose chronological loop-erasure gives this SAP. 

Conversely, starting from a LERW path, we are also able to ``add back'' loops from a loop soup and obtain a SRW. More precisely, letting $\lambda$ be the LERW as above, and let $\cal L$ be an independent Poissonian loop soup with intensity $m$.  Then we can add back loops from $\cal L$ to $\lambda$ through the following procedure.

\begin{prop}[Proposition 5.9 of \cite{Lawlernotes}, See also Prop.\ 4.3, ibid.]\label{prop:addloops}
Given $x\in D\subseteq\mathbb{Z}^3$, let $\lambda$ be a LERW from $x$ stopped at exiting $D$. Insert loops into $\lambda$ in the following way:
\begin{itemize}
\item Repeat for each $j=0,1,\ldots,n-1$:
\begin{itemize}
\item Choose all loops from $\cal L$ in $D\backslash \lambda[0,j-1]$ that touches $\lambda(j)$.
\item For each such loop, choose a representative rooted at $\lambda(j)$ (if there are several representatives then choose uniformly among all possibilities).
\item Concatenate all these loops in the order they appear in the soup and call the concatenated loop $l_j$.
\end{itemize}
\item Insert $l_j$'s into $\lambda$ in the following order (note that $l_j$ starts and stops at $\lambda(j)$): $$l_0\oplus\lambda[0,1]\oplus l_1\oplus\lambda[1,2]\oplus l_2\oplus\cdots\oplus l_{n-1}\oplus\lambda[n-1,n],$$ and call the new path $\gamma$.
\end{itemize}
 Then, $\gamma$ has the law of the SRW started at $x$ stopped at the first exit of $D$.
\end{prop}

\subsection{Escape probability and scaling limit}\label{sec:2.5}

As we discussed in Section \ref{sec:words}, the probability that a LERW and an independent simple random walk do not intersect up to exiting a large ball, which is referred to as escape probability, is a key object in the paper. 

\begin{dfn}\label{escape}
Let $0<m < n$. Let $S^{1}$ and $S^{2}$ be independent SRW's on $\mathbb{Z}^{3}$ started at the origin, and write $P$ for their joint distribution. We define escape probabilities $\Es (n)$ and $\Es (m, n)$ as follows: let
\begin{equation}\label{escape-1-1}
\Es (n) := P\Big( \LE(S^{1} [0, T^n(S^1) ])\cap S^{2} [1, T^n(S^2) ]= \emptyset \Big),
\end{equation}
and let
\begin{equation}\label{escape-1-3}
\Es (m, n) :=  P \Big( \LE(S^{1} [0, T^n(S^1) ])[s,u]\cap S^{2} [1, T^n(S^2) ]= \emptyset \Big),
\end{equation}
where  $u=T^n\big(\LE(S^{1} [0, \tau^{1}_{n} ])\big)$ and $s = \sup \{ t \le u \ | \ \lambda ( t) \in \partial B (0, m ) \}$. More precisely, we first consider the loop erasure of a random walk up to exiting $B (n)$, then we only look at the loop erasure after the last visit to $B (m)$. $\Es (m, n)$ is the probability that this part of the loop erasure does not intersect an independent simple random walk up to the first exiting of $B (n)$.
\end{dfn}
As the most accurate asymptotics of $\Es(n)$ and $\Es(m,n)$ are given in Corollary \ref{COR}, we will not talk about existing weaker estimates in the form of \eqref{mukasi}, but only state a fact which will be used later. It is showed in Lemma 7.2.2 of \cite{S} that 
\begin{equation}\label{Eslogasym}
\lim_{n \to \infty }\frac{\log\Es\big( 2^{(1-q)n }, 2^{n} \big) }{ \log 2^{- \alpha q n} }  = 1.
\end{equation}
where $\alpha$ is the same as in \eqref{mukasi}. 

\medskip

Finally, we review some known facts about the scaling limit of LERW in three dimensions, whose existence was first proved in \cite{Koz}. We refer to \cite{SS} for properties of this limit.
Let $S$ be a simple random walk started at the origin on $\mathbb{Z}^{3}$. Remind the definition of $\mathbb{D}$ in Section \ref{kigou}. Write
\begin{equation}\label{rescale}
\text{LEW}_{n} =\frac{\LE ( S[0, \tau_{n}] )}{n}.
\end{equation}
 We write ${\cal H} (\overline{\mathbb{D}} )$ for the metric space of the set of compact subsets in $\overline{\mathbb{D}}$ with the Hausdorff distance $d_{\text{H}}$. Thinking of $\text{LEW}_{n}$ as random elements of ${\cal H} (\overline{\mathbb{D}} )$, let $P^{(n)}$ be the probability measure on ${\cal H} (\overline{\mathbb{D}} )$ induced by $\text{LEW}_{n}$. Then \cite{Koz} shows that $P^{(2^{j})}$ is a Cauchy sequence with respect to the weak convergence topology, and therefore $P^{(2^{j})}$ converges weakly to some limit probability measure $\nu$. We write ${\cal K}$ for the random compact subset associated with $\nu$ and call ${\cal K}$ the scaling limit of LERW in three dimensions. It is also shown in \cite{Koz} that ${\cal K}$ is invariant under rotations and dilations.

\section{Non-intersection probability}\label{sec:noninter}
This section is dedicated to the proof of Theorem \ref{ESCAPE.1} and is organized in a hierarchical structure. We lay out the structure of the whole proof in Section \ref{sec:3.1} assuming three key intermediate results, namely Lemma \ref{1st-lem}, Propositions \ref{CONDCOMP} and \ref{difscale}. These results are proved in Sections \ref{sec:3.2}, \ref{sec:3.7} and  \ref{sec:difscale} respectively. Sections \ref{sec:3.3}-\ref{sec:3.6} contain intermediate results for Proposition \ref{CONDCOMP} which requires the coupling from Section \ref{sec:coupling}. Section \ref{sec:3.1} also contains the proof of Corollary \ref{COR}.
\subsection{Notations and the proof of Theorem \ref{ESCAPE.1}}\label{sec:3.1}
We start with introducing notations for various walks and paths we are going to discuss in this section. Then we state without proof a few key propositions that compare the non-intersection probabilities under different setups. After that, we give a proof of Theorem \ref{ESCAPE.1} assuming these intermediate results. At last, we give the proof of Corollary \ref{COR}.


Let $n\in\mathbb{Z}^+$. Let $S^{1}$ and $S^{2}$ be independent SRW's on $\mathbb{Z}^{3}$. Write
\begin{equation}\label{gamma.1}
\gamma^{n}_{x} = \LE( S^{1} [0, T^{n} ] ) = \LE( S^{1} [0, T_{2^{n}}] )
\end{equation}
for the loop-erasure of $S^{1}$ up to $T^n$ assuming that $S^{1} (0) = x$.  Using the notation above, $\gamma^{n}_{x} [T^{k}, T^{l} ]$ stands for $\gamma^{n}_{x} [t, u]$ where $t $ (resp. $u$) denotes the first time that $\gamma^{n}_{x}$ hits the boundary of $\cC_k$ (resp. $\cC_l$). Write $\gamma^{n}$ for $\gamma^{n}_{0}$. 

\vspace{3mm}

Let 
\begin{equation}\label{lambda.1}
\lambda_{x} = \Big( \lambda_{x} (k) \Big)_{k \ge 0} =  \Big( S^{2} (k) \Big)_{k \ge 0}
\end{equation}
be $S^{2}$ assuming that $S^{2} (0) = x$. Write $\lambda$ for $\lambda_{0}$.

As introduced in Section \ref{sec:intro}, we are interested in the event 
\begin{equation}\label{An}
A_{n} := \Big\{ \gamma^{n} [0, T^{n}] \cap \lambda  [1, T^{n} ] = \emptyset  \Big\},
\end{equation}
and the quantity 
$$a_{n} := \Es( 2^n)=P ( A_n ),$$
for $\alpha$ as in \eqref{mukasi}. Write 
\begin{equation}\label{ratio-a}
b_{n} = \frac{a_{n}}{a_{n-1}}.
\end{equation}


We fix some $q \in (0,1/10)$ whose explicit value will be specified in Prop.\ \ref{difscale}. We also assume that we always take large $n$ such that $n\geq 30/q$. Now, let 
\begin{equation}\label{Anq}
A_{n, q} := \Big\{ \gamma^{n} [0, T^{(1-q)n}] \cap \lambda  [1, T^{(1-q)n} ] = \emptyset \Big\}\mbox{ and }
A^{-}_{n, q} := \Big\{ \gamma^{n-1} [0, T^{(1-q)n}] \cap \lambda  [1, T^{(1-q)n} ] = \emptyset \Big\}
\end{equation}
(note that $A^{-}_{n, q}$ is different from $A_{n-1, q}$). Then we have 
\begin{equation}
a_{n} = P \Big( A_{n} \ \Big|  \  A_{n, q}  \Big) P \Big(  A_{n,q} \Big)\mbox{ and } a_{n-1} = P \Big( A_{n-1} \ \Big|  \  A^{-}_{n, q}  \Big) P \Big(  A^{-}_{n,q} \Big).
\end{equation}

The following lemma shows that the probability of $A_{n,q}$ is very close to that of $A^{-}_{n,q}$, allowing us to relate $b_n$ to the ratio of two conditional probabilities.
\begin{lem}\label{1st-lem}
It follows that for all $n$ and $q \in (0,1)$
\begin{equation}\label{ratio-1}
b_{n} = \frac{P \Big( A_{n} \ \Big|  \  A_{n, q}  \Big) }{P \Big( A_{n-1} \ \Big|  \  A^{-}_{n, q}  \Big)} \big( 1 + O (2^{- q n} ) \big).
\end{equation}
\end{lem}
We postpone its proof to Section \ref{sec:3.2}.

\vspace{0.3cm}

As explained in Section \ref{sec:intro}, we will relate quantities such as $A_n$ and $A_{n,q}$ to non-intersection probabilities of SRW and LERW started at a mesoscopic distance. To this end, we introduce the following notations.

Let
\begin{equation}
x_1 :=\big(- 2^{(1-5q)n}, 0, 0 \big), \ y_1:= - x_1
\end{equation}
be two poles of $B \big( 2^{(1-5 q) n} \big)$. 
Let 
\begin{equation}\label{star}
\gamma_{1}  = \gamma^{n}_{x_1 }, \ \lambda_{1}  = \lambda_{y_1} [0, T^{n}]
\end{equation}
be a LERW started from $x_1$ stopped at exiting $T^{n}$ and SRW started from $y_1$ stopped at exiting $T^{n}$.  

\vspace{2mm}

Set 
\begin{align}\label{bnq}
B_{n} = \Big\{ \gamma_{1}  [0, T^{n}] \cap \lambda_{1}  [0, T^{n} ] = \emptyset \Big\}\mbox{, and }B_{n, q} = \Big\{ \gamma_{1}  [0, T^{(1-q)n}] \cap \lambda_{1}  [0, T^{(1-q)n} ] = \emptyset \Big\},
\end{align}
which are the analog of $A_{n}$ and $A_{n,q}$ defined in \eqref{An} and \eqref{Anq}. Write 
\begin{equation*}
\pi \big( \gamma_{1}  [0, T^{(1-q)n}],  \lambda_{1}  [0, T^{(1-q)n} ]  \big) = \big( \gamma_{1}  [T^{(1-3q)n}, T^{(1-q)n}],  \lambda_{1}  [T^{(1-3q)n}, T^{(1-q)n} ]  \big).
\end{equation*}
We also write $\gamma'_{1} := \LE(S^{1} [0, T^{n-1} ] ) $ assuming that $S^{1} (0) = x_{1}$ and define 
\begin{align*}
B^{-}_{n} := \Big\{ \gamma'_{1} [0, T^{n-1}] \cap \lambda_{1} [0, T^{n-1}] = \emptyset \Big\}\mbox{ and } B^{-}_{n, q} := \Big\{ \gamma'_{1} [0, T^{(1-q)n}] \cap \lambda_{1} [0, T^{(1-q)n} ] = \emptyset \Big\}.
\end{align*}

We claim that conditional probabilities that appear in \eqref{ratio-1} can be replaced with a small error by corresponding conditional probabilities for $\gamma_1$ and $\lambda_1$.
\begin{prop}\label{CONDCOMP} There exists $\delta>0$, such that
\begin{equation}\label{couple-4}
P \Big( A_{n} \ \Big|  \  A_{n, q}  \Big) = \big( 1 + O (2^{- \delta q n} ) \big) P \Big( B_{n} \ \Big|  \  B_{n, q}  \Big)
\end{equation}
and
\begin{equation}\label{couple-4p}
P \Big( A_{n-1} \ \Big|  \  A^{-}_{n, q}  \Big) = \big( 1 + O (2^{- \delta q n} ) \big) P \Big( B^{-}_{n} \ \Big|  \  B^{-}_{n, q}  \Big).
\end{equation}
\end{prop}
We will postpone the proof to Section \ref{sec:3.7} and dedicate Sections \ref{sec:3.3} - \ref{sec:3.6} to preparatory works. Also, we note that this proposition relies on the coupling result from Section \ref{sec:coupling}.  As a corollary, we have:
\begin{cor}
For some universal constant $\delta > 0$
\begin{equation}\label{CHANGE}
b_{n} =  \frac{P \Big( B_{n} \ \Big|  \  B_{n, q}  \Big) }{P \Big( B^{-}_{n} \ \Big|  \  B^{-}_{n, q}  \Big)} \big( 1 + O (2^{- \delta q n} ) \big).
\end{equation}
\end{cor}
We now introduce quantities that correspond to the scale $2^{n-1}$. Let $$x_{2} := \big( -2^{(1-5q)n -1},0,0 \big) = \frac{x_{1}}{2}, \ y_{2} := -x_{2} = \frac{y_{1}}{2}\mbox{ and }\gamma_{2} := \LE(S^{1} [0, T^{n-1}] )\mbox{ , }\lambda_{2} := S^{2}$$ assuming that $S^{1} (0) = x_{2}$ and  $S^{2} (0) = y_{2}$.
We then define
\begin{align*}
C_{n} := \Big\{ \gamma_{2} [0, T^{n-1} ] \cap \lambda_{2} [0, T^{n-1} ] = \emptyset \Big\} \mbox{ and }C_{n,q} := \Big\{ \gamma_{2} [0, T^{(1-q)n-1} ] \cap \lambda_{2} [0, T^{(1-q)n-1} ] = \emptyset \Big\} 
\end{align*}
and similarly, let $\gamma'_{2} := \LE(S^{1} [0, T^{n-2} ] ) $ assuming that $S^{1} (0) = x_{2}$
and 
\begin{align*}
C^{-}_{n} := \Big\{ \gamma'_{2} [0, T^{n-2}] \cap \lambda_{2} [0, T^{n-2}] = \emptyset \Big\}\mbox{ and } C^{-}_{n, q} := \Big\{ \gamma'_{2} [0, T^{(1-q)n-1}] \cap \lambda_{2} [0, T^{(1-q)n-1} ] = \emptyset \Big\}.
\end{align*}
Similar to \eqref{CHANGE} (note that the only difference between $C_\bullet$ and $B_\bullet$ is the scale), we also have
\begin{equation}\label{CHANGE-2}
b_{n-1} =  \frac{P \Big( C_{n} \ \Big|  \  C_{n, q}  \Big) }{P \Big( C^{-}_{n} \ \Big|  \  C^{-}_{n, q}  \Big)} \big( 1 + O (2^{-  q n} ) \big).
\end{equation}

The following proposition states that the probability of $B_n$ and $B_{n}^-$ are actually close to that of $C_n$ and $C_{n}^-$. We postpone its proof to Section \ref{sec:difscale}.
\begin{prop}\label{difscale}
There exist universal constants $c_{1} > 0$ and $q_{1} > 0$  such that for all $n \ge 1$  and $q \in (0, q_{1} )$,
\begin{align}
&P \Big( B_{n}  \Big) = P \Big( C_{n}  \Big)  \big( 1 + O (2^{- c_{1} q n} ) \big) \mbox{ and }\label{Goal1-1} \\
&P \Big( B^{-}_{n}  \Big) = P \Big( C^{-}_{n}  \Big)  \big( 1 + O (2^{- c_{1} q n} ) \big). \label{Goal2-1}
\end{align}
\end{prop}

We are now ready to prove Theorem \ref{ESCAPE.1}.

\begin{proof}[Proof of Thm.\ \ref{ESCAPE.1} assuming Lemma \ref{1st-lem}, Propositions \ref{CONDCOMP} and \ref{difscale}]
To prove \eqref{GOAL.1}, it suffices to show that there exists universal constants $c_{1}, N> 0$  such that for all $n \ge N$, 
\begin{equation}\label{12mainclaim}
b_{n} = b_{n-1} \big( 1 + O (2^{- c_{1} n} ) \big).
\end{equation}
Recall \eqref{CHANGE} and \eqref{CHANGE-2}. By Proposition 4.4 of \cite{Mas} as in the proof of Lemma \ref{1st-lem}, we have
\begin{align}
&\frac{P \Big( B_{n} \ \Big|  \  B_{n, q}  \Big) }{P \Big( B^{-}_{n} \ \Big|  \  B^{-}_{n, q}  \Big)} = \frac{P \Big( B_{n}  \Big) }{P \Big( B^{-}_{n}   \Big)} \big( 1 + O (2^{- q n} ) \big), \mbox{ and }  \frac{P \Big( C_{n} \ \Big|  \  C_{n, q}  \Big) }{P \Big( C^{-}_{n} \ \Big|  \  C^{-}_{n, q}  \Big)} = \frac{P \Big( C_{n}  \Big) }{P \Big( C^{-}_{n}   \Big)} \big( 1 + O (2^{- q n} ) \big). 
\end{align}
Hence, it follows that there exists a universal constant $\delta > 0$ such that for all $n$ and $q \in (0,1)$
\begin{equation}\label{wb-1}
b_{n} = \frac{P \Big( B_{n}  \Big) }{P \Big( B^{-}_{n}   \Big)} \bigg( 1 + O (2^{-\delta q n} ) \bigg)  
\mbox{ and }b_{n-1} = \frac{P \Big( C_{n}  \Big) }{P \Big( C^{-}_{n}   \Big)} \bigg( 1 + O (2^{- \delta q n} ) \bigg). 
\end{equation}
The claim \eqref{12mainclaim} hence follows by Proposition \ref{difscale} with appropriately chosen $q$ and $N=30/q$ (see above \ref{Anq}). This finishes the proof of Theorem \ref{ESCAPE.1}.
\end{proof}

 \begin{proof}[Proof of Corollary \ref{COR}]
The  first statement of \eqref{esaymp} follows from Proposition 6.2.1 of \cite{S} and \eqref{GOAL.1}, and the second follows from the following fact proved in Proposition 6.2.2 and 6.2.4 of \cite{S}: for $m \le n$,
 \begin{equation*}
 \Es (m) \Es (m,n) \asymp \Es (n).
 \end{equation*}
It follows from Theorem 8.1.4 and Proposition 8.1.5 of \cite{S} that 
 \begin{equation*}
 E( M_{n} ) \asymp n^{2}\Es(n) \asymp n^{2-\alpha},
 \end{equation*}
 which gives the third statement. Finally, exponential tail bounds on $M_{n}$ as in Theorem 8.1.6 and Theorem 8.2.6 of \cite{S} ensure the tightness of ${M_{n}}/{n^{2-\alpha}}$. 
 \end{proof} 
Before ending this subsection, we introduce some path spaces which will be used in the following sections. We write  $\Gamma $ 
for the set of paths satisfying 
\begin{itemize}
\item[$(i)$] $\eta$ is a SAP;

\item[$(ii)$] $\eta (0) =0$, $\eta \big( \text{len} (\eta) \big) \in  \partial \cC_{(1-q) n } $ and $\eta \big[ 0, \text{len} (\eta ) -1 \big] \subset \cC_{(1-q) n } $. 

\end{itemize}
We also write  $\Lambda $ 
for the set of paths satisfying (ii) above only.

\subsection{Proof of Lemma \ref{1st-lem}}\label{sec:3.2}

\begin{proof}[Proof of Lemma \ref{1st-lem}]
Since 
\begin{equation*}
b_{n} = \frac{ P \Big( A_{n} \ \Big|  \  A_{n, q}  \Big) P \Big(  A_{n,q} \Big)}{ P \Big( A_{n-1} \ \Big|  \  A^{-}_{n, q}  \Big) P \Big(  A^{-}_{n,q} \Big)},
\end{equation*}
it suffices to show that 
\begin{equation}
\frac{P \Big(  A_{n,q} \Big)}{P \Big(  A^{-}_{n,q} \Big)} = \big( 1 + O (2^{-q n} ) \big).
\end{equation}

For a path $\eta$, write 
\begin{equation}
f(\eta) = P \Big( \lambda [1, T^{(1-q)n} ] \cap \eta = \emptyset \Big)
\end{equation}
for the probability that $S^{2}$ up to $T^{(1-q)n}$ and $\eta$ do not intersect. Using the function $f$, we see that 
\begin{equation}\label{lem1-1}
 P \Big(  A_{n,q} \Big) = \sum_{\eta \in \Gamma} f (\eta ) P \Big( \gamma^{n} [0, T^{(1-q)n}] = \eta \Big).
 \end{equation}
%
Similarly, we have 
\begin{equation}\label{lem1-2}
 P \Big(  A^{-}_{n,q} \Big) = \sum_{\eta \in \Gamma} f (\eta ) P \Big( \gamma^{n-1} [0, T^{(1-q)n}] = \eta \Big).
 \end{equation}
However, by Corollary 4.5 of \cite{Mas}, for any $\eta \in \Gamma$, we have
\begin{equation}
 P \Big( \gamma^{n} [0, T^{(1-q)n}] = \eta \Big) = \big( 1 + O (2^{-q n} ) \big)  P \Big( \gamma^{n-1} [0, T^{(1-q)n}] = \eta \Big).
 \end{equation}
This finishes the proof of this lemma.
 \end{proof}

\subsection{Decomposition of paths and weak independence}\label{sec:3.3}
In this subsection, we decompose the paths at their first exit of $\cC_{(1-q)n}$ and state without proof some preliminary results for Proposition \ref{CONDCOMP}. 

Remind the definition of $\Lambda$ and $\Gamma$ at the end of Section \ref{sec:3.1}. Define 
\begin{equation}\label{cnq}
{\cal C}  = \Big\{ (\eta^{1}, \eta^{2} ) \in \Gamma \times \Lambda \ \Big| \ \eta^{1} [0, T^{(1-q) n}] \cap \eta^{2} [1, T^{(1-q) n} ] = \emptyset \Big\}.
\end{equation}
Take $(\eta^{1}, \eta^{2} ) \in {\cal C} $. Let $w^{i}$ be the endpoint of $\eta^{i}$ lying on $\partial \cC_{(1-q) n } $. 
Write 
\begin{align}\label{X}
&R^{1}, R^{2} \text{ for two independednt simple random walks started at } w^{1}, w^{2} \text{ and } \notag \\
&X \text{ for } R^{1} \text{ conditioned that } R^{1} [1, T^{n} ] \cap \eta^{1} = \emptyset.
\end{align}
By the domain Markov property of LERW (see Lemma \ref{domain}), conditioned on $\gamma^{n} [0, T^{(1-q)n}] = \eta^{1}$, the (conditional) distribution of $\gamma^{n} [T^{(1-q)n}, T^{n} ]$ is same as the law of $\LE( X [0, T^{n} ] )$. With this in mind, for  $(\eta^{1}, \eta^{2} ) \in {\cal C} $, let
\begin{equation}\label{g}
g (\eta^{1}, \eta^{2} ) = P \Big( \big( \LE( X [0, T^{n} ] ) \cup \eta^{1} [0, {\rm len} (\eta^{1} ) ] \big) \cap \big( R^{2} [0, T^{n} ] \cup  \eta^{2} [1, {\rm len} (\eta^{2} ) ] \big) = \emptyset \Big)
\end{equation}
be the probability that the loop-erasure of $X$ and $R^{2}$ do not intersect. We are now able to re-write $P ( A_{n} )$ in terms of $g$:
$$
 P \Big( A_{n} \Big)= \sum_{(\eta^{1}, \eta^{2} ) \in {\cal C} } g (\eta^{1}, \eta^{2} )  P \Big(  \big( \gamma^{n} [0, T^{(1-q)n}], \lambda  [0, T^{(1-q)n} ] \big) = (\eta^{1}, \eta^{2})  \Big).
 $$
Hence
\begin{equation}\label{henkei}
P \Big( A_{n} \ \Big|  \  A_{n, q}  \Big) = 
 \sum_{(\eta^{1}, \eta^{2} ) \in {\cal C} } g (\eta^{1}, \eta^{2} ) \mu_{n, q} (\eta^{1}, \eta^{2} ),
\end{equation}
where 
\begin{equation}\label{mu}
\mu_{n, q} (\eta^{1}, \eta^{2} ) = P \Big(  \big( \gamma^{n} [0, T^{(1-q)n}], \lambda  [0, T^{(1-q)n} ] \big) = (\eta^{1}, \eta^{2}) \ \Big| \ A_{n,q} \Big)
\end{equation}
stands for the conditional distribution on $A_{n,q}$.

The next proposition measures the magnitude of $g(\eta^1,\eta^2)$ in terms of a function $h$ of $(\eta^1,\eta^2)$ we are going to define below and $\Es(\cdot,\cdot)$ in Definition \ref{escape}. We postpone its proof till Section \ref{sec:3.6}.
\begin{prop}\label{ghEs}
One has
\begin{equation}\label{eq:ghEs}
g (\eta^{1}, \eta^{2} )  \asymp h  (\eta^{1}, \eta^{2} ) \Es \big( 2^{(1-q) n}, 2^{n} \big),
\end{equation}
where
\begin{align}\label{h}
h  (\eta^{1}, \eta^{2} )&\; = P \Big( \big( \LE( X [0, T^{n} ] ) [0, s] \cup \eta^{1} [0, {\rm len} (\eta^{1} ) ] \big) \notag  \cap \big( R^{2} [0, T^{(1-q)n+1} ] \cup  \eta^{2} [1, {\rm len} (\eta^{2} ) ] \big) = \emptyset \Big) \notag \\
\text{and } s&\; = \inf \Big\{ k \ge 0 \  \Big| \ \LE  \big( X [0, T^{n} ]  \big)  (k) \notin \cC_{(1-q) n+1 }  \Big\}.
\end{align}
\end{prop}
\begin{rem}
We now explain the significance of Prop.\ \ref{ghEs}. The function $h$ measures closeness of $\eta^{1}$ and $\eta^{2}$ in the following sense. Let 
\begin{equation}\label{eq:Ddef}
D (\eta^{1}, \eta^{2} ) = \frac{\min \Big\{ {\rm dist} \big( w^{1}, \eta^{2} \big), {\rm dist} \big( w^{2}, \eta^{1} \big) \Big\}}{2^{(1-q)n + 1 }},
\end{equation}
where $w^{i}$ stands for the endpoint of $\eta^{i}$. It turns out that 
\begin{equation*}
h (\eta^{1}, \eta^{2})  \text{ is small } \Longleftrightarrow D (\eta^{1}, \eta^{2} ) \text{ is small}.
\end{equation*}
However, we note that if $D (\eta^{1}, \eta^{2} ) \le \frac{1}{2}$, then $D (\eta^{1}, \eta^{2} )$ does not depend on the initial part of $(\eta^{1}, \eta^{2})$, i.e.,
\begin{equation*}
D (\eta^{1}, \eta^{2} ) = D \big( \eta^{1} [T^{(1-q)n}, T^{(1-q)n + 1} ],  \eta^{2} [T^{(1-q)n}, T^{(1-q)n + 1} ] \big).
\end{equation*}
This gives an intuitive reason why (loosely speaking) $h (\eta^{1}, \eta^{2})$ does not depend on the initial part of $(\eta^{1}, \eta^{2})$. Therefore, once we show that the dependence of the magnitude of $h (\eta^{1}, \eta^{2})$ is small on the initial part of $(\eta^{1}, \eta^{2} ) $ is negligible, we are able to show the same thing for $g (\eta^{1}, \eta^{2} ) $. This proposition will be one of the ingredients for the proof of Proposition \ref{CONDCOMP}.
\end{rem}

\subsection{Asymptotic independence of $g (\eta^{1}, \eta^{2})$ from initial parts}\label{sec:3.4}
The goal of this subsection is to show that roughly speaking, $g$ does not depend on the ``initial part" of  $ (\eta^{1}, \eta^{2} ) $. In other words, if two pairs $ (\eta^{1}, \eta^{2} ),  (\eta^{3}, \eta^{4} )  \in {\cal C}$ satisfy $\eta^{i} [ T^{(1-2q) n}, T^{(1-q) n} ] = \eta^{i + 2} [ T^{(1-2q) n}, T^{(1-q) n} ]$ for $i =1,2$, then $g (\eta^{1}, \eta^{2})$ is very close to $g (\eta^{3}, \eta^{4})$.

\vspace{4mm}

We recall that the ${\cal C} $, the set of pairs of paths, was defined as in \eqref{cnq}.  Take $(\eta^{1}, \eta^{2} ) \in {\cal C}$. We denote the endpoint of $\eta^{i}$ by $w^{i}$ which lies on $\partial\cC_{(1-q) n } $. We also define a truncating operation on paths by
\begin{align}\label{pi}
&\pi (\eta^{i}) = \pi_{n,q} (\eta^{i} ) = \eta^{i} [T^{(1-3q) n}, T^{(1-q) n} ]; \quad(\eta^{1}, \eta^{2}) = \pi_{n,q} (\eta^{1}, \eta^{2}) = \big( \pi_{n,q} (\eta^{1} ) , \pi_{n,q} (\eta^{2} ) \big).
\end{align}

\vspace{3mm}

We want to consider an analog of $g (\eta^{1}, \eta^{2})$ for $\pi (\eta^{1}, \eta^{2})$. With this in mind, we write (note the difference of $\overline{X}$ defined here and $X$  in \eqref{X})
\begin{align}\label{Xver}
&R^{1}, R^{2} \text{ for two independednt simple random walks started at } w^{1}, w^{2} \text{ and } \notag \\
&\overline{X} \text{ for } R^{1} \text{ conditioned that } R^{1} [1, T^{n} ] \cap \pi ( \eta^{1} ) = \emptyset.
\end{align}
Let 
\begin{equation}\label{gver}
\overline{g} (\eta^{1}, \eta^{2} ) = P \Big( \big( \LE( \overline{X} [0, T^{n} ] ) \cup \pi (\eta^{1} ) \big) \cap \big( R^{2} [0, T^{n} ] \cup  \pi (\eta^{2} ) \big) = \emptyset \Big).
\end{equation}
Note that $\overline{g} (\eta^{1}, \eta^{2} )$ is a function of $\pi (\eta^{1}, \eta^{2}) $ and it does not depend on the initial part of $(\eta^{1}, \eta^{2} )$. 

\vspace{3mm}

We next define an analog of $h( \eta^{1}, \eta^{2} )$ for $\pi (\eta^{1}, \eta^{2})$ (see \eqref{h} for the definition of $h$). To do it, let 
\begin{equation*}
\overline{s} = \inf \Big\{ k \ge 0 \  \Big| \ LE  \big( \overline{X} [0, T^{n} ]  \big)  (k) \notin \cC_{(1-q)n + 1}   \Big\}.
\end{equation*}
We define
\begin{equation}\label{hver}
\overline{h}  (\eta^{1}, \eta^{2} ) = P \Big( \big( \LE( \overline{X} [0, T^{n} ] ) [0, \overline{s}] \cup \pi (\eta^{1})  \big) \cap \big( R^{2} [0, T^{(1-q)n+1} ] \cup \pi ( \eta^{2} ) \big) = \emptyset \Big).
\end{equation}
Again we remark that $\overline{h}  (\eta^{1}, \eta^{2} )$ is a function of $\pi (\eta^{1}, \eta^{2}) $.
An easy modification of the proof of Proposition \ref{ghEs} gives that 
\begin{prop}\label{barghEs}
One has
\begin{equation}\label{modif}
\overline{g} (\eta^{1}, \eta^{2} ) \asymp \overline{h} ( \eta^{1}, \eta^{2} )\Es\big( 2^{(1-q)n }, 2^{n} \big).
\end{equation}
\end{prop}
\vspace{2mm}

The following proposition shows that $\overline{g} (\eta^{1}, \eta^{2} )$ is close enough to $g  (\eta^{1}, \eta^{2} )$ for ``typical" $(\eta^{1}, \eta^{2} )$ in the sense that $\overline{h} ( \eta^{1}, \eta^{2} )$ is not too small.  More precisely, we have 
\begin{prop}\label{prop1}
There exists $C < \infty$ such that for all $n$, $q \in (0,1)$ and $(\eta^{1}, \eta^{2} ) \in {\cal C}$ satisfying 
\begin{equation}\label{cond}
\overline{h}  (\eta^{1}, \eta^{2} ) \ge 2^{ - \frac{q n}{2} },
\end{equation}
it follows that 
\begin{equation}\label{g-gver}
\big| g( \eta^{1}, \eta^{2} ) - \overline{g} (\eta^{1}, \eta^{2} ) \big| \le C 2^{ - \frac{q n}{2} } \overline{g} (\eta^{1}, \eta^{2} ).
\end{equation}
\end{prop}

\begin{proof} We follow the notations introduced at the beginning of Section \ref{sec:3.3}.
Take $(\eta^{1}, \eta^{2} ) \in {\cal C}$ satisfying \eqref{cond} 
and  set 
\begin{align}\label{event}
&H_{1} =  \Big\{ \big( \LE( R^{1} [0, T^{n} ] ) \cup \eta^{1} [0, {\rm len} (\eta^{1} ) ] \big) \cap \big( R^{2} [0, T^{n} ] \cup  \eta^{2} [1, {\rm len} (\eta^{2} ) ] \big) = \emptyset \Big\}; \notag \\
&H_{2} = \Big\{ R^{1} [1, T^{n} ] \cap \eta^{1} = \emptyset \Big\}; \notag \\
&\overline{H}_{1} =  \Big\{ \big( \LE( R^{1} [0, T^{n} ] ) \cup \pi (\eta^{1}) \big) \cap \big( R^{2} [0, T^{n} ] \cup  \pi (\eta^{2})  \big) = \emptyset \Big\}; \notag \\
&\overline{H}_{2} = \Big\{ R^{1} [1, T^{n} ] \cap \pi (\eta^{1}) = \emptyset \Big\}.
\end{align}
Then by definition, we have
\begin{align}\label{ggver}
&g (\eta^{1}, \eta^{2} ) = P \big( H_{1} \ \big| \ H_{2} \big) = \frac{P \big( H_{1}, \ H_{2} \big) }{P \big( H_{2} \big) }; \quad \overline{g} (\eta^{1}, \eta^{2} ) = P \big( \overline{H}_{1} \ \big| \ \overline{H}_{2} \big) = \frac{P \big( \overline{H}_{1}, \ \overline{H}_{2} \big) }{P \big( \overline{H}_{2} \big) }.
 \end{align}
 
We first show that $P \big( H_{2} \big)$ is close to $P \big( \overline{H}_{2} \big)$. It is clear that $P \big( H_{2} \big) \le P \big( \overline{H}_{2} \big)$ since $\pi (\eta^{1} ) \subset \eta^{1}$. On the other hand, we have
\begin{equation*}
P \big( \overline{H}_{2} \big) - P \big( H_{2} \big) \le P \Big( \overline{H}_{2}, \ R^{1} [1, T^{n} ] \cap\cC_{(1-3q)n} \neq \emptyset \Big).
\end{equation*}
 In oder to bound the RHS of the inequality above, set
$$
 \widetilde{B} = B \Big( w^{1}, \frac{2^{(1-q)n}}{3} \Big)\mbox{ and }J = \eta^{1} \cap  \widetilde{B} = \pi ( \eta^{1} ) \cap  \widetilde{B}.
 $$
We also let 
\begin{equation*}
u = \inf \{ k \ | \ R^{1} (k) \in \partial  \widetilde{B} \}.
\end{equation*}
By the strong Markov property and Proposition 1.5.10 of \cite{Lawb},
\begin{equation*}
P \Big( \overline{H}_{2}, \ R^{1} [1, T^{n} ] \cap \cC_{(1-3q)n}  \neq \emptyset \Big) \le c 2^{-2 q n} P \Big( R^{1} [1, u] \cap J = \emptyset \Big).
\end{equation*}
We write $D = \big\{ x \in \partial B \ \big| \ x \notin B \big( \frac{5}{4} \cdot 2^{(1-q)n} \big) \big\}$ for a subset of $\partial  \widetilde{B}$. Then by Proposition 6.1.1 of \cite{S} and Proposition 1.5.10 of \cite{Lawb} again, we see that 
\begin{align*}
&\;P \big( H_{2} \big) \ge P \Big( R^{1} [1, u] \cap J = \emptyset, \ R^{1} (u) \in D, \  R^{1} [u, T^{n} ] \cap B \big( 2^{(1-q)n } \big) = \emptyset \Big) \\
\ge &\;c P \Big( R^{1} [1, u] \cap J = \emptyset, \ R^{1} (u) \in D \Big) \ge c P \Big( R^{1} [1, u] \cap J = \emptyset \Big).
\end{align*}
Therefore, we have
\begin{equation}\label{h2}
P \big( H_{2} \big) = P \big( \overline{H}_{2} \big) \big( 1 + O (2^{-2 q n } ) \big).
\end{equation}

We next compare $P \big( H_{1}, \ H_{2} \big) $ and $P \big( \overline{H}_{1}, \ \overline{H}_{2} \big)$. Note that $H_{1} \subset \overline{H}_{1}$. Thus, 
\begin{equation*}
P \big( H_{1}, \ H_{2} \big) \le P \big( \overline{H}_{1}, \ \overline{H}_{2} \big).
\end{equation*}
Moreover, by the strong Markov property as above, it follows that 
\begin{align}\label{H1}
&\;P \big( \overline{H}_{1}, \ \overline{H}_{2} \big) - P \big( H_{1}, \ H_{2} \big) \notag \\
\le&\; P \Big( \overline{H}_{1}, \ \overline{H}_{2}, \  R^{1} [0, T^{n}] \cap \cC_{(1-3q)n} \neq \emptyset \Big)+ P \Big( \overline{H}_{1}, \ \overline{H}_{2}, \  R^{2} [0, T^{n}] \cap \cC_{(1-3q)n}  \neq \emptyset \Big)  \\
\le&\; c 2^{-2 q n} P \Big( R^{1} [1, u] \cap J = \emptyset \Big) \le c 2^{-2 q n} P \big( H_{2} \big).\notag
\end{align}
By \eqref{modif} and \eqref{cond}, we see that 
\begin{equation*}
\overline{g} (\eta^{1}, \eta^{2} )  \ge c 2^{ - \frac{q n}{2} }\Es\big( 2^{(1-q)n }, 2^{n} \big),
\end{equation*}
Combining this with \eqref{ggver}, we have 
\begin{equation}\label{low}
P \big( \overline{H}_{1}, \ \overline{H}_{2} \big) \ge c 2^{ - \frac{q n}{2} } P \big( \overline{H}_{2} \big) 
\Es \big( 2^{(1-q)n }, 2^{n} \big).
\end{equation}
Therefore, by \eqref{Eslogasym} and the fact that $\alpha < 1$, we have
\begin{equation}\label{eq:3.42}
P \big( \overline{H}_{1}, \ \overline{H}_{2} \big) \ge c 2^{ - \frac{3 q n}{2} } P \big( \overline{H}_{2} \big) .
\end{equation}
Thus, using \eqref{eq:3.42} and \eqref{H1}, we conclude that 
\begin{equation*}
P \big( \overline{H}_{1}, \ \overline{H}_{2} \big) - P \big( H_{1}, \ H_{2} \big) \le c 2^{ - \frac{q n}{2} } P \big( \overline{H}_{1}, \ \overline{H}_{2} \big),
\end{equation*}
which gives
\begin{equation}\label{concl}
P \big( H_{1}, \ H_{2} \big) = P \big( \overline{H}_{1}, \ \overline{H}_{2} \big) \big( 1 + O (2^{ - \frac{q n}{2} } ) \big).
\end{equation}
Finally, using \eqref{ggver}, \eqref{h2} and \eqref{concl}, we have 
\begin{equation*}
g (\eta^{1}, \eta^{2} ) = \frac{P \big( H_{1}, \ H_{2} \big) }{P \big( H_{2} \big) } = \frac{P \big( \overline{H}_{1}, \ \overline{H}_{2} \big) }{P \big( \overline{H}_{2} \big) } \Big( 1 + O (2^{ - \frac{q n}{2} } ) \Big) = \overline{g} (\eta^{1}, \eta^{2} ) \big( 1 + O (2^{ - \frac{q n}{2} } ) \big),
\end{equation*}
which completes the proof.
\end{proof}

\subsection{Comparison of conditional probabilities}\label{sec:3.5}
The goal of this subsection is \eqref{henkei-4} and \eqref{henkei-5}, in which $P ( A_{n} \ |  \  A_{n, q}  )$ and $P ( B_{n} \ |  \  B_{n, q}  )$ are both rewritten (with a small error term) into weighted sums of $\overline{g}(\cdot,\cdot)$ which allows an easy comparison using results from Section \ref{sec:coupling}.

\vspace{3mm}

We recall that $\mu_{n,q}$ was defined as in \eqref{mu} which is a probability measure on ${\cal C}$ obtained by the conditional distribution on $A_{n,q}$. We also recall 
the decomposition of $P ( A_{n} |    A_{n, q}  )$ in \eqref{henkei}. 
The next proposition shows that we can replace $g (\eta^{1}, \eta^{2} ) $ in the RHS of \eqref{henkei} with $\overline{g} (\eta^{1}, \eta^{2} )$ with small enough error terms.

\begin{prop}\label{replace}
One has that 
\begin{equation}\label{henkei-3}
P \Big( A_{n} \ \Big|  \  A_{n, q}  \Big) = \big(1 + O (2^{ - \frac{q n}{2} } ) \big) \sum_{(\eta^{1}, \eta^{2} ) \in {\cal C}} \overline{g} (\eta^{1}, \eta^{2} ) \mu_{n, q} (\eta^{1}, \eta^{2} ). 
\end{equation}
\end{prop}

\begin{proof}
We set 
\begin{equation*}
{\cal C}_{1} = \big\{ (\eta^{1}, \eta^{2} ) \in {\cal C} \ \big| \ (\eta^{1}, \eta^{2} ) \text{ satisfies \eqref{cond}} \big\}
\end{equation*} 
and let ${\cal C}_{2} = {\cal C} \setminus {\cal C}_{1}$. 

By the separation lemma (see Theorem 6.1.5 of \cite{S} or Claim 3.4 of \cite{SS} for the separation lemma), we see that there exists a universal constant $c, c'> 0$ such that for all $n$ and $q \in (0,1)$
\begin{equation*}
\mu_{n,q} \Big( \big\{ (\eta^{1}, \eta^{2} ) \in {\cal C} \ \big| \  \eta^{1} \text{ and } \eta^{2} \text{ are } c \text{-well-separated} \big\} \big) \ge c',
\end{equation*}
where we say $\eta^{1}$ and $\eta^{2}$ are $c$-well-separated if 
\begin{equation}\label{eq:cwell}
\min \Big\{ \text{dist} \Big( \eta^{1} \big( {\rm len} (\eta^{1} ) \big), \eta^{2} \Big),    \text{dist} \Big( \eta^{2} \big( {\rm len} (\eta^{2} ) \big), \eta^{1} \Big)  \ge c 2^{(1-q)n}.
 \end{equation}
If $\eta^{1}$ and $\eta^{2}$ are  $c$-well-separated, then it is easy to see that there exists $c' > 0$
\begin{equation*}
\overline{h} (\eta^{1}, \eta^{2} ) \ge c',
\end{equation*}
which gives 
\begin{equation*}
\overline{g} (\eta^{1}, \eta^{2} ) \ge c \Es \big( 2^{(1-q)n}, 2^{n} \big).
\end{equation*}
Therefore, we have 
\begin{equation}\label{compara-es}
\sum_{(\eta^{1}, \eta^{2} ) \in {\cal C}} \overline{g} (\eta^{1}, \eta^{2} ) \mu_{n, q} (\eta^{1}, \eta^{2} ) \asymp \Es \big( 2^{(1-q)n}, 2^{n} \big).
\end{equation}

Combining this with Proposition \ref{prop1}, we have 
\begin{align*}
P \Big( A_{n} \ \Big|  \  A_{n, q}  \Big)=&\; \sum_{(\eta^{1}, \eta^{2} ) \in {\cal C}_{1}} g (\eta^{1}, \eta^{2} ) \mu_{n, q} (\eta^{1}, \eta^{2} ) + \sum_{(\eta^{1}, \eta^{2} ) \in {\cal C}_{2}} g (\eta^{1}, \eta^{2} ) \mu_{n, q} (\eta^{1}, \eta^{2} ) \\
=&\; \big(1 + O (2^{ - \frac{q n}{2} } ) \big) \sum_{(\eta^{1}, \eta^{2} ) \in {\cal C}_{1}} \overline{g} (\eta^{1}, \eta^{2} ) \mu_{n, q} (\eta^{1}, \eta^{2} ) + O (2^{ - \frac{q n}{2} } ) \sum_{(\eta^{1}, \eta^{2} ) \in {\cal C}} \overline{g} (\eta^{1}, \eta^{2} ) \mu_{n, q} (\eta^{1}, \eta^{2} ) \\
=&\; \big(1 + O (2^{ - \frac{q n}{2} } ) \big) \sum_{(\eta^{1}, \eta^{2} ) \in {\cal C}} \overline{g} (\eta^{1}, \eta^{2} ) \mu_{n, q} (\eta^{1}, \eta^{2} ),
\end{align*}
which gives the proposition.
\end{proof}
The following corollary is a by product of the proof above (see \eqref{compara-es}).
\begin{cor} One has that
\begin{equation}\label{couple-3}
P \Big( A_{n} \ \Big|  \  A_{n, q}  \Big) \asymp \Es \big( 2^{(1-q)n}, 2^{n} \big).
\end{equation}
\end{cor}

\vspace{4mm}

Recall that $\overline{g} (\eta^{1}, \eta^{2} ) $ is a function of $\pi (\eta^{1}, \eta^{2} )$ (see \eqref{gver} for the definition of $\overline{g} (\eta^{1}, \eta^{2} ) $). With this in mind, we define a set of pairs of paths $\overline{ {\cal C}}$ by
\begin{equation}
\overline{ {\cal C}} = \big\{  (\overline{\eta}^{1}, \overline{\eta}^{2} )  \ \big| \   (\overline{\eta}^{1}, \overline{\eta}^{2} ) \text{ satisfies }  (iii),  (iv) \text{ and } (v)  \big\} 
\end{equation}
where
\begin{align}
& (iii) \ \overline{\eta}^{1} \text{ is a SAP and } \overline{\eta}^{2}  \text{ is a path}. \\
& (iv) \ \overline{\eta}^{i} (0) \in \partial \cC_{(1-3q) n } , \ \overline{\eta}^{i} [0, {\rm len} (\overline{\eta}^{i} ) -1 ] \subset \cC_{(1-q) n }  \mbox{ and } \overline{\eta}^{i} \big( {\rm len} (\overline{\eta}^{i} ) \big) \in \partial \cC_{(1-q) n }  \text{ for } i=1,2. \\
& (v) \ \overline{\eta}^{1} [0,  {\rm len} (\overline{\eta}^{1} ) ] \cap \overline{\eta}^{2} [0,  {\rm len} (\overline{\eta}^{2} ) ] = \emptyset.
\end{align}
With little abuse of notation, we can then define $\overline{g} (\overline{\eta}^{1}, \overline{\eta}^{2} )$ for $(\overline{\eta}^{1}, \overline{\eta}^{2} ) \in \overline{ {\cal C}}$ through $\overline{g} ({\eta}^{1},{\eta}^{2} )$ for any $ ({\eta}^{1},{\eta}^{2} )$ such that $\pi({\eta}^{1},{\eta}^{2} )=(\overline{\eta}^{1}, \overline{\eta}^{2} )$.

We next define a probability measure $\overline{ \mu}_{n,q}$ on $\overline{ {\cal C}} $. For $(\overline{\eta}^{1}, \overline{\eta}^{2} ) \in \overline{ {\cal C}}$, define $\overline{ \mu}_{n,q} (\overline{\eta}^{1}, \overline{\eta}^{2} )$ by
\begin{eqnarray}\label{muver}
\overline{ \mu}_{n,q} (\overline{\eta}^{1}, \overline{\eta}^{2} )=\left\{ \begin{array}{ll}
\mu_{n,q}  \Big( F_{n, q}  (\overline{\eta}^{1}, \overline{\eta}^{2} ) \Big)  & \text{ if } F_{n, q}  (\overline{\eta}^{1}, \overline{\eta}^{2} )\neq \emptyset \\
0 & \text{ otherwise}  \\
\end{array} \right.
\end{eqnarray}
where 
\begin{equation}\label{fnq}
F_{n, q}  (\overline{\eta}^{1}, \overline{\eta}^{2} ) = \big\{ (\eta^{1}, \eta^{2} ) \in {\cal C} \ \big| \ \pi (\eta^{1}, \eta^{2} )  = (\overline{\eta}^{1}, \overline{\eta}^{2} ) \big\},
\end{equation}
see \eqref{mu} and \eqref{pi} for $\mu_{n,q}$ and  $\pi (\eta^{1}, \eta^{2})$ respectively.
Note that $\overline{ \mu}_{n,q} $ is a probability measure on $\overline{ {\cal C}}$ which is induced by 
\begin{equation}
\pi  \big( \gamma^{n} [0, T^{(1-q)n}], \lambda  [0, T^{(1-q)n} ] \big) \text{ conditioned on } A_{n,q}.
\end{equation}

\vspace{4mm}

The next corollary rephrases Proposition \ref{replace} in terms of $\overline{ \mu}_{n,q} $.

\begin{cor}\label{cor1}
It follows that 
\begin{equation}\label{henkei-4}
P \Big( A_{n} \ \Big|  \  A_{n, q}  \Big) = \big(1 + O (2^{ - \frac{q n}{2} } ) \big) \sum_{(\overline{\eta}^{1}, \overline{\eta}^{2} ) \in \overline{ {\cal C}}} \overline{g} (\overline{\eta}^{1}, \overline{\eta}^{2} ) \overline{ \mu}_{n,q} (\overline{\eta}^{1}, \overline{\eta}^{2} ).
\end{equation}
\end{cor} 

We now turn to $P ( B_{n}|   B_{n, q} )$. We let $ \overline{ \mu}^{\star}_{n,q}$ be the probability measure on $\overline{ {\cal C}}$ which is induced by 
\begin{equation}
\pi \Big( \gamma_1 [0, T^{(1-q)n}],  \lambda_1 [0, T^{(1-q)n} ]  \Big) \text{ conditioned on } B_{n,q}
\end{equation}
As the next proposition can be proved very similarly, we will omit its proof. 
\begin{prop}\label{bnbnq}
It follows that 
\begin{equation}\label{henkei-5}
P \Big( B_{n} \ \Big|  \  B_{n, q}  \Big) = \big(1 + O (2^{ - \frac{q n}{2} } ) \big) \sum_{(\overline{\eta}^{1}, \overline{\eta}^{2} ) \in \overline{ {\cal C}}} \overline{g} (\overline{\eta}^{1}, \overline{\eta}^{2} ) \overline{ \mu}^{\star}_{n,q} (\overline{\eta}^{1}, \overline{\eta}^{2} ).
\end{equation}
\end{prop}

\subsection{Proof of Propositions \ref{ghEs} and \ref{barghEs} } \label{sec:3.6}

As two propositions are extremely similar, we will only prove Proposition \ref{ghEs}.%
We treat two directions of \eqref{eq:ghEs} in Lemmas \ref{2nd-lem} and \ref{3rd-lem} separately.

We recall the definition of $\Es (n)$ and $\Es (m, n)$ in Definition \ref{escape}. 

\begin{lem}\label{2nd-lem}
There exists $c < \infty$ such that for all $n$, $q$ and $(\eta^{1}, \eta^{2} ) \in {\cal C}$,
\begin{equation}\label{eq:2nd-lem}
g (\eta^{1}, \eta^{2} )  \le c h  (\eta^{1}, \eta^{2} )\Es\big( 2^{(1-q) n}, 2^{n} \big).
\end{equation}
\end{lem}
\begin{proof}
Take $(\eta^{1}, \eta^{2} ) \in {\cal C}$. Recall the definition of $F=F_{n,q}$ in \eqref{fnq}. 
Then we have 
\begin{equation*}
g( \eta^{1}, \eta^{2} ) = \frac{P \Big( A_{n}, F (\eta^{1}, \eta^{2} ) \Big)}{P \Big( F (\eta^{1}, \eta^{2} ) \Big)}.
\end{equation*}
Define 
\begin{equation}\label{t1}
t_{1} =  \max \Big\{ k  \  \Big| \ \gamma^{n}  (k) \in \partial\cC_{(1-q) n +3} \big)  \Big\}.
\end{equation}
Then it follows that 
\begin{align*}
\quad \;P \Big( A_{n}, F (\eta^{1}, \eta^{2} ) \Big)  
& \le P \Big( \gamma^{n} [t_{1}, T^{n} ] \cap \lambda [1, T^{n}] = \emptyset, \;\gamma^{n} [0, T^{(1-q)n + 1}] \cap \lambda [1, T^{(1-q)n + 1}] = \emptyset, \   F (\eta^{1}, \eta^{2} ) \Big) .
\end{align*}

By Proposition 4.6 of \cite{Mas}, since $ \gamma^{n} [t_{1}, T^{n} ]$ and $\gamma^{n} [0, T^{(1-q)n + 1}]$ are ``independent up to constant", we see that 
\begin{align*}
&\;P \Big( \gamma^{n} [t_{1}, T^{n} ] \cap \lambda [1, T^{n}] = \emptyset, \ \gamma^{n} [0, T^{(1-q)n + 1}] \cap \lambda [1, T^{(1-q)n + 1}] = \emptyset, \   F (\eta^{1}, \eta^{2} ) \Big) \\
\asymp &\;E^{2} \Big( Y_{1} Y_{2} {\bf 1} \big\{ \lambda  [0, T^{(1-q)n} ]  = \eta^{2} \big\} \Big),
\end{align*}
where $P^{i}$ stands for the probability law of $S^{i}$ assuming $S^{i} (0) = 0$ and $Y^{i}$ are defined by 
\begin{align*}
&Y^{1} = P^{1} \Big(  \gamma^{n} [t_{1}, T^{n} ] \cap \lambda [1, T^{n}] = \emptyset \Big) = P^{1} \Big(  \gamma^{n} [t_{1}, T^{n} ] \cap \lambda [T^{(1-q)n + 2}, T^{n}] = \emptyset \Big); \\
& Y^{2} = P^{1} \Big( \gamma^{n} [0, T^{(1-q)n + 1}] \cap \lambda [1, T^{(1-q)n + 1}] = \emptyset, \  \gamma^{n} [0, T^{(1-q)n}] = \eta^{1} \Big).
\end{align*}
Note that $Y^{1}$ is a function of $\lambda [T^{(1-q)n + 2}, T^{n}] $ while $Y^{2}  {\bf 1} \big\{ \lambda  [0, T^{(1-q)n} ]  = \eta^{2} \big\}$ is a function of $\lambda [0, T^{(1-q)n + 1} ]$. Therefore by Harnack principle, it follows that  $Y^{1}$ and $Y^{2}  {\bf 1} \big\{ \lambda  [0, T^{(1-q)n} ]  = \eta^{2} \big\}$ are also ``independent up to constant". Thus, we have
\begin{equation*}
 E^{2} \Big( Y_{1} Y_{2} {\bf 1} \big\{ \lambda  [0, T^{(1-q)n} ]  = \eta^{2} \big\} \Big) \asymp  E^{2} \Big( Y_{1} \Big) E^{2} \Big( Y_{2} {\bf 1} \big\{ \lambda  [0, T^{(1-q)n} ]  = \eta^{2} \big\} \Big).
\end{equation*}
Again by Harnack principle, we see that 
\begin{equation*}
E^{2} \Big( Y^{1} \Big) \asymp E^{2} \Big\{  P^{1} \Big(  \gamma^{n} [t_{1}, T^{n} ] \cap \lambda [1, T^{n}] = \emptyset \Big) \Big\} =\Es \big( 2^{(1-q)n + 3} , 2^{n} \big).
\end{equation*} 
By Proposition 6.2.1, 6.2.2 and 6.2.4 of \cite{S}, we see that 
\begin{equation*}
\Es  \big( 2^{(1-q)n + 3} , 2^{n} \big) \asymp\Es \big( 2^{(1-q)n} , 2^{n} \big).
\end{equation*}
On the other hand, 
\begin{align*}
&E^{2} \Big( Y_{2} {\bf 1} \big\{ \lambda  [0, T^{(1-q)n} ]  = \eta^{2} \big\} \Big)  = P \Big( \gamma^{n} [0, T^{(1-q)n + 1}] \cap \lambda [1, T^{(1-q)n + 1}] = \emptyset, \ F(\eta^{1}, \eta^{2} ) \Big).
\end{align*}
Therefore, by domain Markov property of LERW (see Lemma \ref{domain}), we have 
\begin{align*}
&\;g( \eta^{1}, \eta^{2} ) = \frac{P \Big( A_{n}, F (\eta^{1}, \eta^{2} ) \Big)}{P \Big( F (\eta^{1}, \eta^{2} ) \Big)} \le c  \frac{\Es  \big( 2^{(1-q)n} , 2^{n} \big) }{P \Big( F (\eta^{1}, \eta^{2} ) \Big)}P \Big( \gamma^{n} [0, T^{(1-q)n + 1}] \cap \lambda [1, T^{(1-q)n + 1}] = \emptyset, \ F(\eta^{1}, \eta^{2} ) \Big)  \\
=&\; c\,\Es \big( 2^{(1-q)n} , 2^{n} \big) P \Big( \gamma^{n} [0, T^{(1-q)n + 1}] \cap \lambda [1, T^{(1-q)n + 1}] = \emptyset  \ \Big| \ F(\eta^{1}, \eta^{2} ) \Big) = c\Es \big( 2^{(1-q)n} , 2^{n} \big) h(\eta^{1}, \eta^{2} ),
\end{align*}
which completes the proof of \eqref{eq:2nd-lem}.
\end{proof}
The following claim can be proved in a similar way.
\begin{cor}
For all $(\overline{\eta}^{1}, \overline{\eta}^{2} ) \in \overline{ {\cal C}}$, 
\begin{equation}\label{couple-2}
\overline{g} (\overline{\eta}^{1}, \overline{\eta}^{2} ) \le c \Es \big( 2^{(1-q)n}, 2^{n} \big).
\end{equation}
\end{cor}

The next lemma shows the opposite direction.

\begin{lem}\label{3rd-lem}
There exists $c  > 0$ such that for all $n$, $q$ and $(\eta^{1}, \eta^{2} ) \in {\cal C}$,
\begin{equation}\label{eq:3rd-lem}
g (\eta^{1}, \eta^{2} )  \ge c h  (\eta^{1}, \eta^{2} )\Es\big( 2^{(1-q) n}, 2^{n} \big).
\end{equation}
\end{lem}

\begin{proof}
We will follow the proof of Proposition 5.3 of \cite{Mas} and Proposition 6.2.4 of \cite{S}. 

We recall that $t_{1}$ is the last time that $\gamma^{n}$ lies in $\partial B \big( 2^{(1-q)n + 3} \big)$ (see \eqref{t1} for $t_{1}$).  Let (these notations pertain only in this proof) 
\begin{align*}
&\gamma_{1} = \gamma^{n} [0, T^{(1-q)n +1}];\;\gamma_{2} = \gamma^{n} [t_{1}, T^{n} ];\;\widehat{\gamma} = \gamma^{n} [T^{(1-q)n +1}, t_{1}],
\end{align*}
so that $\gamma^{n} = \gamma_{1} \oplus \widehat{\gamma} \oplus \gamma_{2}$. Let $\gamma_{0} = \gamma^{n} [0, T^{(1-q)n}]$. We decompose $\lambda$ into $\lambda_{1} = \lambda [1, T^{(1-q)n + 2}]$ and $\lambda_{2} = \lambda [T^{(1-q)n + 2}, T^{n}]$. Let $\lambda_{0} = \lambda [0, T^{(1-q)n}]$.

For $\kappa \in [0, 1]$, define $\pi (\kappa ) = \{ (x_{1}, x_{2}, x_{3} ) \in \mathbb{R}^{3} \ | \ x_{1} = \kappa \}$. We set $H ( \kappa ) = \{ x \in \mathbb{R}^{3} \ | \ |x| \le 1 \} \cap \pi (\kappa)$ and a cone $O ( \kappa ) = \{ rx \ | \ r \ge 0, \  x \in H ( \kappa ) \}$. We set $O_{i} = O \big( \frac{2 + i}{3+i} \big)$ for $i=1,2, 3$. Note that $O_{3} \subset O_{2} \subset O_{1}$. Define 
\begin{align*}
&W = \Big\{ x \in \mathbb{Z}^{3} \  \Big| \  \frac{3}{4} \cdot 2^{(1-q)n + 1} \le  |x| \le  \frac{5}{4} \cdot 2^{(1-q)n + 3} \Big\} \cap O_{2}\mbox{ and } \\
&W^{\ast} = \Big\{ x \in \mathbb{Z}^{3} \  \Big| \  \frac{3}{4} \cdot 2^{(1-q)n + 1} \le  |x| \le  \frac{5}{4} \cdot 2^{(1-q)n + 3} \Big\} \cap O_{3},
\end{align*}
and set $A = W \cup \cC_{(1-q)n + 1 } $. Let $K_{1}, K_{2}$ be sets of paths defined by 
\begin{align*}
&K_{1} = \Big\{ \eta \ \Big| \ \eta \cap \partial \cC_{(1-q)n + 1} \in O_{3}, \  P \big( \gamma_{1} = \eta \big) > 0 \Big\};
\; K_{2} = \Big\{ \eta \ \Big| \ \eta \cap \partial \cC_{(1-q)n + 3} \in O_{3}, \  P \big( \gamma_{2} = \eta \big) > 0 \Big\}.
\end{align*}
Then we have
\begin{align*}
&\;P \Big( A_{n}, F (\eta^{1}, \eta^{2} ) \Big) = P \Big( \gamma^{n} [0, T^{n} ] \cap \lambda [1, T^{n} ] = \emptyset, \ (\gamma_{0}, \lambda_{0} ) = (\eta^{1}, \eta^{2} ) \Big) \\
=&\;P \Big( \lambda_{1} \cap \gamma_{1} \oplus \widehat{\gamma} = \emptyset, \ \lambda_{2} \cap \gamma_{1} \oplus \widehat{\gamma} \oplus \gamma_{2} = \emptyset, \  (\gamma_{0}, \lambda_{0} ) = (\eta^{1}, \eta^{2} ) \Big) \\
\ge&\; E^{1} \Big[ {\bf 1} \{ \gamma_{1} \in K_{1} \} {\bf 1} \{ \gamma_{2} \in K_{2} \} {\bf 1} \{ \widehat{\gamma} \subset W \} {\bf 1} \{ \gamma_{0} = \eta^{1} \}  P^{2} \Big( \lambda_{1} \cap  ( \gamma_{1} \cup W^{\ast} ) = \emptyset, \ \lambda_{2} \cap (\gamma_{2} \cup A) = \emptyset, \ \lambda_{0} = \eta^{2} \Big) \Big] \\
\ge&\; E^{1} \Big[ \widetilde{X} \widetilde{Y} {\bf 1} \{ \widehat{\gamma} \subset W \} \Big],
\end{align*}
where 
\begin{align*}
&\widetilde{X} = {\bf 1} \{ \gamma_{1} \in K_{1} \} {\bf 1} \{ \gamma_{0} = \eta^{1} \} P^{2} \Big( \lambda_{1} \cap  ( \gamma_{1} \cup W^{\ast} ) = \emptyset,  \ \lambda_{0} = \eta^{2} \Big) \\
&\widetilde{Y} = {\bf 1} \{ \gamma_{2} \in K_{2} \} \min_{x \in \partial \cC{(1-q)n + 2}\setminus W^{\ast} } P^{2}_{x} \Big( S^{2} [0, T^{n} ] \cap (\gamma_{2} \cup A) = \emptyset \Big).
\end{align*}
Note that $\widetilde{X}$ is a function of $\gamma_{1}$ while $\widetilde{Y}$ is a function of $\gamma_{2}$.

By domain Markov property of LERW and Lemma 6.2.3 of \cite{S}, it follows that there exists $c > 0$ such that for all $\eta \in K_{1}$ and $\eta' \in K_{2}$
\begin{equation*}
P^{1} \Big( \widehat{\gamma} \subset W \ \Big| \ \gamma_{1} = \eta, \ \gamma_{2} = \eta' \Big) \ge c.
\end{equation*}
This gives
\begin{equation*}
P \Big( A_{n}, F (\eta^{1}, \eta^{2} ) \Big)  \ge c E^{1} \Big[ \widetilde{X} \widetilde{Y} \Big].
\end{equation*}
However, by Proposition 4.6 of \cite{Mas}, we see that 
\begin{equation*}
P \Big( A_{n}, F (\eta^{1}, \eta^{2} ) \Big)  \ge c E^{1} ( \widetilde{X} ) E^{1} (\widetilde{Y}).
\end{equation*}

It follows from (6.43) of \cite{S} that 
\begin{equation*}
E^{1} (\widetilde{Y} ) \ge c\Es\big( 2^{(1-q)n }, 2^{n} \big).
\end{equation*}
Therefore, it suffices to show that 
\begin{equation}\label{SEPARATE}
E^{1} ( \widetilde{X} ) \ge c P \Big( \gamma_{1} \cap \lambda [1, T^{(1-q)n +1 }] = \emptyset, \  (\gamma_{0}, \lambda_{0} ) = (\eta^{1}, \eta^{2} ) \Big).
\end{equation}
Write $\lambda' = \lambda [1, T^{(1-q)n +1 }] $ and $\lambda'' = \lambda [T^{(1-q)n +1 }, T^{(1-q)n + 2}]$ so that $\lambda_{1} = \lambda' \oplus \lambda''$.

To prove \eqref{SEPARATE}, by the separation lemma (see (6.13) of \cite{S} for the version of the separation lemma that we need here), we see that there exists some universal constant $c  > 0$ such that (see \ref{eq:cwell} for definition of being well-separated):
\begin{align*}
E^{1} ( \widetilde{X} )&\;\ge P \Big( (\gamma_{0}, \lambda_{0} ) = (\eta^{1}, \eta^{2} ), \ \gamma_{1} \cap \lambda' = \emptyset, \ \gamma_{1} \text{ and } \lambda' \text{ are } c \text{-well-separated},   \lambda'' \cap (\gamma_{1} \cup W^{\ast} ) = \emptyset \Big) \\
&\; \ge c  P \Big( (\gamma_{0}, \lambda_{0} ) = (\eta^{1}, \eta^{2} ), \ \gamma_{1} \cap \lambda' = \emptyset, \ \gamma_{1} \text{ and } \lambda' \text{ are } c \text{-well-separated} \Big) \\
&\;\ge c P \Big( (\gamma_{0}, \lambda_{0} ) = (\eta^{1}, \eta^{2} ), \ \gamma_{1} \cap \lambda' = \emptyset \Big),
\end{align*}
which finishes the proof of \eqref{eq:3rd-lem}.
\end{proof}

\subsection{Proof of Proposition \ref{CONDCOMP}}\label{sec:3.7}
\begin{proof}[Proof of Proposition \ref{CONDCOMP}]
We only prove \eqref{couple-4} as \eqref{couple-4p} follows in a similar manner. To show  \eqref{couple-4}, it suffices to show
\begin{equation}\label{couple-4w}
\Big| P \Big( A_{n} \ \Big|  \  A_{n, q}  \Big) - P \Big( B_{n} \ \Big|  \  B_{n, q}  \Big)\Big| \leq  \big( 1 + O (2^{- \delta q n} ) \big) P \Big( A_{n} \ \Big|  \  A_{n, q}  \Big) .
\end{equation}
We observe that applying Prop.\ \ref{prop:coup2} with $(k,N)$ there equal to $\big((1-3q)n,(1-q)n\big)$, we have
\begin{equation}\label{couple}
\Big|\Big|\overline{ \mu}_{n,q} (\overline{\eta}^{1}, \overline{\eta}^{2} ) - \overline{ \mu}^{\star}_{n,q} (\overline{\eta}^{1}, \overline{\eta}^{2} )\Big|\Big|_{\rm TV} \leq c2^{-\delta qn},
\end{equation}
where $||\cdot||_{\rm TV}$ stands for the total variation distance. Hence,
\begin{align*}
\;\Bigg| \sum_{(\overline{\eta}^{1}, \overline{\eta}^{2} ) \in \overline{ {\cal C}}} \overline{g} (\overline{\eta}^{1}, \overline{\eta}^{2} ) &\overline{ \mu}_{n,q} (\overline{\eta}^{1}, \overline{\eta}^{2} ) - \sum_{(\overline{\eta}^{1}, \overline{\eta}^{2} ) \in \overline{ {\cal C}}} \overline{g} (\overline{\eta}^{1}, \overline{\eta}^{2} ) \overline{ \mu}^{\star}_{n,q} (\overline{\eta}^{1}, \overline{\eta}^{2} ) \Bigg|  \\
\overset{\eqref{couple-2}}{\le}&\; C \Es (2^{(1-q) n}, 2^{n} )  \sum_{(\overline{\eta}^{1}, \overline{\eta}^{2} ) \in \overline{ {\cal C}}}  \Big| \overline{ \mu}_{n,q} (\overline{\eta}^{1}, \overline{\eta}^{2} ) - \overline{ \mu}^{\star}_{n,q} (\overline{\eta}^{1}, \overline{\eta}^{2} ) \Big|  \\
\overset{\eqref{couple}}{\le} &\;C' \Es (2^{(1-q)n}, 2^{n} ) 2^{- \delta q n}  \overset{ \eqref{henkei-4} }{\underset{\eqref{couple-3}}{\le}} C'' 2^{- \delta q n} \sum_{(\overline{\eta}^{1}, \overline{\eta}^{2} ) \in \overline{ {\cal C}}} \overline{g} (\overline{\eta}^{1}, \overline{\eta}^{2} ) \overline{ \mu}_{n,q} (\overline{\eta}^{1}, \overline{\eta}^{2} ).
\end{align*}
 Thus, \eqref{couple-4w} follows by rewriting the leftmost and rightmost expression above back to conditional probabilities, thanks to \eqref{henkei-4} and \eqref{henkei-5}. This finishes the proof of \eqref{couple-4}.
\end{proof}

\subsection{Proof of Proposition \ref{difscale}}\label{sec:difscale}
We recall that $\lambda_1$ and $\lambda_{2}$ stands for the SRW on $\mathbb{Z}^{3}$ started at $y_1$ and $y_{2}$, respectively. We also recall that
$$B_n=\{\gamma_{1} \cap \lambda_{1} = \emptyset\}\mbox{ and }C_n=\{\gamma_{2} \cap \lambda_{2} = \emptyset\}.$$
In order to keep coherence of notation in this subsection we will use  notations on the right hand side above in the proposition below.
\begin{prop}  There exist universal constants $c_{3} > 0$ and $q_{1} > 0$  such that for all $n \ge 1$  and $q \in (0, q_{1} )$,
 \begin{equation}\label{g1g2}
 \Big| P \Big( \gamma_{1} \cap \lambda_{1} = \emptyset \Big) - P \Big( \gamma_{2} \cap \lambda_{2} = \emptyset \Big) \Big| \le  c_{3} 2^{-10 q n}.
 \end{equation}
\end{prop}
\begin{proof}
We will closely follow the proof of Proposition 7.1.1 of \cite{S}. We seek to replace the SRW in both probabilities in \eqref{g1g2} by Wiener sausages (see \eqref{g1r} and \eqref{g2r}), and establish an inequality between them (which is \eqref{kozma}). 

Lemma 3.2 of \cite{Lawcut} proves that it is possible to couple $\lambda_{2}$ and $W$, a Brownian motion in $\mathbb{R}^{3}$ started at $y_{2}$, on the same probability space $P^2$ such that 
\begin{equation}
P^{2} \Big( J_1^c \Big) \le a e^{-2^{b n} },
\end{equation}
where
\begin{equation}
J_{1} = \Big[ \max_{0 \le t \le T^{n} } \big| W(t) - \lambda_{2} (3 t) \big| \le 2^{\frac{2n}{3}-1} \Big],
\end{equation}
for some universal constants $a, b >0$. Throughout this proof, we will assume $(\lambda_{2}, W)$ is defined on the same probability space as above. We also write $E^2$ for the corresponding expectation.
Define the event $J_{2}$ by 
\begin{equation}
J_{2} = \Big[ W[0, \infty ) \cap B \big( x_{2}, 2^{(1-15q ) n-1} \big) = \emptyset \Big].
\end{equation}
Then by Theorem 3.17 of \cite{MP}, it follows that
\begin{equation}
P^2 (J_{2} ) \ge 1 - c2^{-10qn}.
\end{equation}

\vspace{2mm}

We now consider the Wiener sausage.
For a discrete or continuous path $\eta$ and $L\in\mathbb{R}$, write 
\begin{equation}
(\eta)^{+L} := \Big[ x \in \mathbb{R}^{3} \ \big| \ \text{ there exists } y \in \eta [0, \text{len} (\eta ) ] \text{ such that } |x-y| \le 2^L \Big]
\end{equation}
for its sausage of radius $2^L$. We also write 
\begin{equation}
W_{2} = W[0, T^{n-1}], \ W'_{2} = 2 W_{2}.
\end{equation}
We let $R = (R (j) )_{j \ge 0}$ be the simple random walk on $2 \mathbb{Z}^{3}$ started at $x_{1}$ and let $\widetilde{\gamma}_{2} = \LE(R [0, T^{n} ])$ be its loop-erasure up to $T^{n}$. 

\vspace{2mm}

By Theorem 5 of \cite{Koz}, there exist deterministic universal constants $q_{0} \in (0, 1)$,  $c_{0} \in (0, \frac{1}{4} )$ and $c_{1} < \infty$ such that for all $q \in (0, q_{0})$ and $W$ satisfies $J_{2}$ then it follows that 
\begin{equation}\label{kozma}
 P^{R} \Big( \widetilde{\gamma}_{2} \cap (W'_{2})^{+ (\frac{2n}{3}+1)} = \emptyset \Big) \ge P^{1} \Big( \gamma_{1} \cap (W'_{2})^{+ {(1-c_{0})n} } = \emptyset \Big) - c_{1} 2^{-c_{0} n},
 \end{equation}
 where $P^{R}$ stands for the probability law of $R$ while $P^{1}$ stands for the law of $S^{1}$ (or equivalently law of $\gamma_{1}$). Thus, the two probabilities in  \eqref{kozma} are functions of $W_{2}$.  (Note that we can take $q_{0} = \frac{1}{15} \times \min \{ \frac{1}{6}, \frac{\epsilon}{8}, \frac{\delta_{2}}{8} \}$ where $\epsilon$ and $\delta_{2}$ are universal constants as in the proof of Theorem 5 of \cite{Koz} for the case that $G^{1} = \mathbb{Z}^{3}$ and $G^{2} = 2 \mathbb{Z}^{3}$. Taking $q_{0}$ like this form and conditioned $W$ on $J_{2}$, we can take universal deterministic constants $c_{0}$ and $c_{1}$ such that they do not depend on the starting point, see (132) of \cite{Koz}.)
 
 \vspace{2mm}
 
 Next, we will replace each probability of \eqref{kozma} by the corresponding non-intersection probability of LERW and SRW as in the proof of Proposition 7.1.1 of \cite{S}. We start with the left one. Note that 
 \begin{equation}
 P^{R} \Big( \widetilde{\gamma}_{2} \cap (W'_{2})^{+ (\frac{2n}{3}+1)}  = \emptyset \Big) = P^{1} \Big(  \gamma_{2} \cap (W_{2})^{+\frac{2n}{3}} = \emptyset \Big).
 \end{equation}
 Therefore, taking expectation with respect to $W$, we have 
 \begin{align}
  &\; P \Big( \gamma_{2} \cap (W_{2})^{+\frac{2n}{3}} = \emptyset \Big)\notag 
 \ge E^{2} \Big[ P^{R} \Big( \widetilde{\gamma}_{2} \cap (W'_{2})^{+ (\frac{2n}{3}+1)}  = \emptyset \Big)  \ ; \ J_{2}  \Big]   \\
\ge&\; E^{2} \Big[ P^{1} \Big( \gamma_{1} \cap (W'_{2})^{+ {(1-c_{0})n} } = \emptyset \Big) - c_{1} 2^{-c_{0} n} \ ; \ J_{2}  \Big]  
\ge E^{2} \Big[ P^{1} \Big( \gamma_{1} \cap (W'_{2})^{+ {(1-c_{0})n} } = \emptyset \Big)\Big]  - c_{1} 2^{-c_{0} n} - c 2^{-10qn}, \notag
\end{align}
for all $q \in (0, q_{0} )$. But the scaling property of the Brownian motion ensures that the law of $W'_{2}$ with $W_2$ started from $y_2$ coincides with the law of the Brownian motion $B_{1} := B[0, T^{n}]$ started from $y_{1}$ up to $T^{n}$. Thus, we have 
\begin{equation}\label{sausage}
 P \Big( \gamma_{2} \cap (W_{2})^{+\frac{2n}{3}} = \emptyset \Big) \ge P \Big( \gamma_{1} \cap (B_{1})^{+ {(1-c_{0})n} } = \emptyset \Big) - c_{1} 2^{-c_{0} n} - c 2^{-10qn},
 \end{equation} 
 for all $q \in (0, q_{0} )$.

 \vspace{2mm}
 
 Now we compare the probability of RHS of \eqref{sausage} and $ P (B_{n} ) $.
 We again assume that $\lambda_{1}$ and $B$ are coupled such that the Hausdorff distance between them is $\le  2^{\frac{2n}{3}}$ with probability at least $1- a e^{-2^{b n}}$ for some universal constants $a, b > 0$ (this is possible by Lemma 3.2 of \cite{Lawcut}). Applying Lemma 4.8 of \cite{Koz} (see Theorem 3.1 of \cite{SS} for a stronger version of it), it follows that there exists universal constants $c_{2}, \rho > 0$ such that for all $q \in (0, q_{0} )$,
 \begin{equation*}
 \Big| P \Big( \gamma_{1} \cap \lambda_{1} = \emptyset \Big) - P \Big( \gamma_{1} \cap (\lambda_{1})^{+ {(1-c_{0})n} }  = \emptyset \Big) \Big| \le c_{2} 2^{-c_{0} \rho  n}.
 \end{equation*}
 Combining this with our coupling of $\lambda_{1}$ and $B$, we see that  
 \begin{equation} \label{g1r}
 \Big| P \Big( \gamma_{1} \cap \lambda_{1} = \emptyset \Big) -P \Big( \gamma_{1} \cap (B_{1})^{+ {(1-c_{0})n} } = \emptyset \Big) \Big| \le c_{2} 2^{-c_{0} \rho  n}.
 \end{equation} 
 Similarly, we see that 
 \begin{equation}\label{g2r}
 \Big| P \Big( \gamma_{2} \cap \lambda_{2} = \emptyset \Big) -P \Big( \gamma_{2} \cap (W_{2})^{+\frac{2n}{3}} = \emptyset \Big) \Big| \le c_{2} 2^{-c_{0} \rho  n}.
 \end{equation}  
 
 \vspace{2mm}
 
 Set $q_{1} = \min \{  \frac{c_{0} \rho}{10} , \frac{q_{0}}{10} \}$. Note that $q_{1}$ is a universal constant. We have showed that there exists a universal constant $c_{3}$ such that for all $q \in (0, q_{1} )$, 
 \begin{equation*}
 P \Big( \gamma_{2} \cap \lambda_{2} = \emptyset \Big) \ge P \Big( \gamma_{1} \cap \lambda_{1} = \emptyset \Big) - c_{3} 2^{-10 q n}.
 \end{equation*}
  An inequality in the opposite direction also follows similarly.  This gives \eqref{g1g2}.
\end{proof}
 
 \begin{proof}[Proof of Proposition \ref{difscale}]
We will only prove \eqref{Goal1-1} as the second claim \eqref{Goal2-1} follows similarly. 
 Since
 \begin{equation*}
 P \Big( \gamma_{1} \cap \lambda_{1} = \emptyset \Big)  \asymp  P \Big( \gamma_{2} \cap \lambda_{2} = \emptyset \Big) \asymp \Es \big( 2^{(1-5q)n }, 2^{n} \big) \ge c 2^{-5 q n},
 \end{equation*}
 by \eqref{g1g2}, it follows that for $q \in (0, q_{1} )$,
 \begin{equation}
 P (B_{n} ) = P (C_{n} ) + O \big(  2^{-10 q n} \big)  = P (C_{n} ) \big( 1 + O (2^{-5qn} ) \big),   
 \end{equation}
 which gives \eqref{Goal1-1} and completes the proof.
 \end{proof}

%
\section{Coupling}\label{sec:coupling}
In this section, we establish various couplings of pairs of loop-erased walk and simple random walk conditioned to avoid each other up to some point under different setup and different initial configurations. As a corollary we obtain \eqref{couple} which is a key ingredient in Section \ref{sec:noninter}. As the prototype of such couplings already appears in \cite{Lawrecent}, in this work we will give a direct proof, but rather argue through fine-tuning the coupling result from \cite{Lawrecent}. For more discussion, see the beginning of Section \ref{sec:Gregvar1}.

\subsection{Setup and statement}

We start by giving a brief introduction to our coupling. Pick $k,n>0$ ({\bf not necessarily an integer}) and $N\geq 2n+k$. Let $\gamma$ be an ILERW and $\lambda$ be a SRW both in $\mathbb{Z}^3$ with $\gamma(0)=\lambda(0)=0$, independent of each other. We write $\bfeta=(\gamma,\lambda)$ for the pair of walks and write $\bfP=\bfP_{0,0}$ for its law. We write 
\begin{equation}
\bfeta_N:=(\gamma_N,\lambda_N):=(\gamma[0,T^N],\lambda[0,T^N])
\end{equation} for the walks  truncated at the first exit of $\cC_N$. 

Let $x,y\in\overline{\cC}_k$. Similarly, let $\bfeta'=(\gamma',\lambda')$ where $\gamma'$ is an ILERW with $\gamma'(0)=x$ and $\lambda'$ is a SRW with $\lambda'(0)=y$, again, independent from each other. We write $\bfP_{x,y}$ for its joint law. For $N\geq k$, we define $\gamma'_N$, $\lambda'_N=\lambda'[0,T^N]$ and $\bfeta'_N$ similarly.

We write 
\begin{equation}
U_N:= \Big\{\gamma_N\cap\lambda_N[1,T^N]=\emptyset\Big\}
\end{equation} 
for the event that $\gamma_N$ and $\lambda_N$ have no intersection and define the event 
$$
U'_N:=\Big\{\gamma_N \bigcap 
\begin{array}{c}
\lambda_N [1,T^N]\\
\lambda_N
\end{array} =\emptyset\Big\}\mbox{ if}
\begin{array}{l}
\gamma_N(0)=\lambda_N(0);\\
\gamma_N(0)\neq\lambda_N(0)
\end{array}
$$
similarly. Note that in the second case, there is no need to exclude $\lambda_N(0)$.
We write 
\begin{equation}
\bbfP^N[{\bfeta}\in\cdot]=\bfP[\bfeta\in\cdot |U_N]
\end{equation} 
and
\begin{equation}
\bbfP^N_{x,y}[{\bfeta}\in\cdot]=\bfP_{x,y}[\bfeta\in\cdot |U'_N],
\end{equation}
for the laws of $\bfeta$ and $\bfeta'$ conditioned on $U_N$ and $U'_N$, respectively.

For $N\geq n+k$, we write $\bfeta_N =_n \bfeta'_N$ if the paths agree from their first exit from $\cC_{N-n}$ onwards, i.e., 
$$
\gamma_{N-n,N}=\gamma'_{N-n,N},\quad\lambda_{N-n,N}=\lambda'_{N-n,N}.
$$
where $\gamma_{N-n,N}=\gamma[T^{N-n},T^N]$, with other notations defined similarly.

We are now ready to state the our coupling. 

\begin{prop}\label{prop:coup1}
There exist $\beta_1>0$ and $c_1<\infty$ such that for all $n,k>0$, $N\geq 2n+k$, $x,y\in \overline{\cC}_k$, there is a coupling $\bfQ$ of $\bfeta$ under $\bbfP^N$ and $\bfeta'$ under $\bbfP^N_{x,y}$, such that 
\begin{equation}\label{eq:coup1}
\bfQ[\bfeta_{N}=_{n} \bfeta'_{N} ]\geq1-c_12^{-\beta_1 n}.
\end{equation}
\end{prop}

In the coupling above, we can also replace the ILERW in $\bfeta'$  by a long LERW from $x$ to $\partial\cC_M$, with $M>N$ (in applications we would like $2^M \gg 2^N$). 
Since the law of the beginning part of an ILERW and that of LERW  are almost the same, such replacement should not change much if we only look at what is happening inside $\cC_N$.

Now let $M>N$ and let $\gamma^M$ be a LERW started from $x\in\overline{\cC}_k$ stopped at exiting $ \cC_M$. Replacing  $\gamma'$ by $\gamma^M$ in the definition above, we define $\bfeta^M$, $\bfeta^M_N$, $U_N^M$, $\bfP_{x,y}^M$ and $\bbfP^{N,M}_{x,y}$ accordingly. 

As a corollary of Proposition \ref{prop:coup1}, we have the following coupling.
\begin{prop}\label{prop:coup2}
There exist $\beta_2>0$ and $c_2<\infty$ such that for all $k,n>0$, $M\geq N\geq 2n+k$, and $x,y\in \overline{\cC}_k$, there is a coupling $\bfQ^M$ of $\bfeta$ under $\bbfP^N$ and $\bfeta^M$ under $\bbfP^{N,M}_{x,y}$, such that
\begin{equation}\label{eq:coup2}
\bfQ^M[\bfeta_{N}=_{n} \bfeta^M_{N} ]\geq1-c_2 2^{-\beta_2\min(n,M-N)}.
\end{equation}
\old{Here constants will NOT depend on $k$, $x$, $y$, $n$, $N$ or $M$.}
\end{prop}
\noindent Note that \eqref{couple} follows as a corollary of Proposition \ref{prop:coup2}.
\begin{proof}[Proof of Prop.\ \ref{prop:coup2} given Prop.\ \ref{prop:coup1}]
In the setting of Prop.\ \ref{prop:coup1}, instead of sampling $\gamma'$, we  sample $\gamma'$ and $\gamma^M$ coupled in the same probability space $\cal P$, such that there exists $c'>0$, 
$$
{\cal P}[\gamma^M_N=\gamma'_N|\gamma'_N=\gamma^\circ]>1-c'2^{-M+N}
$$
uniformly for any self avoiding path $\gamma^\circ$ from $x$ to $\partial\cC_N$. This is possible thanks to Lemma \ref{lem:LEWILEW}.
Note that conditioning on $\gamma^M_N=\gamma'_N$, we have $1_{U'_N}=1_{U_N^M}$. Hence,
$$
\bigg|\frac{\bfP_{x,y}[U'_N]}{\bfP_{x,y}^N[U^M_N]}-1\bigg|<c2^{-M+N}.$$
This observation, along with the coupling $\cal P$ between $\gamma'$ and $\gamma^N$, allows us to modify the coupling in Proposition \ref{prop:coup1} to obtain a new coupling $\bfQ^M$ of $\bfeta$ under $\bbfP^N$ and $\bfeta^M$ under $\bbfP^{N,M}_{x,y}$ such that \eqref{eq:coup2} is satisfied with $\beta_2=\min(\beta_1,1)$.
This finishes the proof.
\end{proof}

\begin{rem}
If both $\bfeta$ and $\bfeta'$ are replaced by LERW stopped at exiting $\partial \cC_M$, then, since boundary issues are no longer a problem, one can have a coupling with better error probability estimates. In fact,  in Prop.\ \ref{prop:coup3} we are going to state a version with better error bounds in that case, which essentially replaces $\min(n,M-N)$ in \eqref{eq:coup2} by $n$.
\end{rem}

\subsection{Variations on a coupling by Lawler}\label{sec:Gregvar1}
In this subsection,we are going to restate the coupling result from \cite{Lawrecent} under the setup that suits our needs.

In the course of proving the existence of infinite two-sided infinite loop-erased random walk (ITLERW), Greg Lawler considered a pair of ILERW's started from the origin, conditioned to not intersect each other up to some level and then tilted by a loop term. Then he constructed a coupling between such a pair of loop-erased walks and another pair conditioned on some prefixed initial configurations. As we are going to see, this coupling is intimately related to the non-intersection probability of a LERW and a SRW. For instance, if we consider $\bfeta$ under $\bbfP_n$ (see the previous subsection for precise definition), and let $\iota$ be the loop erasure of $\lambda_n$, then the law of $(\gamma_n,\iota)$ can also be described through a tilting by loop terms. Hence, it is possible to modify the coupling from \cite{Lawrecent} to obtain Prop.\ \ref{prop:coup1}. However, as the setup and tilting terms are slightly different in \cite{Lawrecent} and in our case, some care must be taken.

\medskip

We start by restating the coupling in \cite{Lawrecent}.

Let $\ovg^1$ and $\ovg^2$ be two independent ILERW starting from $0$ and record their joint law by $\bfM$. For $0<k<N$, write $\bbfM^N$ for the law of $\bfM$ tilted by 
\begin{equation}\label{eq:tilting}
G_N(\ovg^1_N,\ovg^2_N):{=}1_{\ovg^1_N\cap\ovg^2_N=\ovg^1(0)\cap\ovg^2(0)}\exp(-L_N(\ovg^1_N,\ovg^2_N)),
\end{equation}
where $L_N(\ovg^1_N,\ovg^2_N)$ is the loop measure of loops in $\cC_N$ that touch both $\ovg^1_N$ and $\ovg^2_N$. Let  $g_1,g_2$ be two SAP's started from $0$ and stopped at first exiting ${\cC}_k$, such that 
\begin{equation}\label{eq:initial}
g_1\cap g_2=\{0\}.
\end{equation}
Let $\bfM_{\rm g}$ be the law of $\ovg^1$ and $\ovg^2$ conditioned on $(\ovg^1_k,\ovg^2_k)=(g_1,g_2)$ (in this case we  write $\ovg^{1,{\rm g}}$ and $\ovg^{2,{\rm g}}$ for $\ovg^1$ and $\ovg^2$) and let $\bbfM^N_{\rm g}$ be $\bfM_{\rm g}$ tilted by $G_N(\ovg^{1,{\rm g}}_N,\ovg^{2,{\rm g}}_N)$.

\begin{rem}\label{rem:loop0}
Note that in {\rm \cite{Lawrecent}}, the definition of $L_N(\ovg^1_N,\ovg^2_N)$ for $d=3$ is the measure of loops in $\cC_N\backslash\{0\}$ that touches both $\ovg^1_N$ and $\ovg^2_N$. Our choice in \eqref{eq:tilting} does not change the tilted probability law but gives us some convenience in notation below when we do not start both walks from the same point any more.
\end{rem}

We are now ready to state the original version of this coupling. Note that although the original version used $e$ as the ratio between exponential scales, it is not a problem for us since it was stated explicitly  in\cite{Lawrecent} that exponents do not have to be integers.
\begin{prop}[Proposition 2.31 of \cite{Lawrecent}]\label{prop:Greg1}
There exists $\beta_3>0$ and $c_3<\infty$ such that for all $k,n>0$, $N\geq 2n+k$ and any ${\rm g}=(g_1,g_2)$ satisfying \eqref{eq:initial}, we can find a coupling $\bfQ^\infty$ of $\bbfM^N$ and $\bbfM^N_{\rm g}$ such that 
$$
\bfQ^\infty\big[(\ovg^1_N,\ovg^2_N)=_{n} (\ovg^{1,{\rm g}}_N,\ovg^{2,{\rm g}}_N)\big]>1-c_32^{-\beta_3 n}.
$$
\end{prop}
\begin{rem}\label{rem:tip}
Although it is tempting to claim that we can obtain the coupling in Proposition \ref{prop:coup1} by appropriately ``adding back'' loops from an independent loop soup to $\ovg_2$ and $\ovg^{2,{\rm g}}$ simultaneously, it is in fact imprecise due to the fact that the distributions of the ``tip'' of an LERW and an ILERW differ greatly (also, there are a few stitches in choice of loop terms). We will not discuss this in detail here but mention that in order to generate objects with the right distribution, one has to be very careful both in the sampling of $\ovg$'s and the choice of loop terms (e.g.\ $L_N(\ovg^1_N,\ovg^2_N)$ in \eqref{eq:tilting}) in tilting procedures.
\end{rem}

In fact, the ``initial configuration'' in the definition above does not have to be a nearest-neighbor SAP. As we now explain, this coupling also works under more general setups, especially when walks start from points other than the origin.

Pick $x,y\in\overline{\cC}_k$ and let $(\ovg^{1,x},\ovg^{2,y})$ be two independent ILERW starting from $x$ and $y$ respectively and record their joint law by $\bfM_{x,y}$. Define $(\ovg^{1,x}_N,\ovg^{2,y}_N)$ accordingly. Again let $\bbfM^N_{x,y}$ be the law of $\bfM_{x,y}$ tilted by $G_N(\ovg^{1,x}_N,\ovg^{2,y}_N)$.%

We now state a variant of Prop.\ \ref{prop:Greg1}. Note that we still denote the coupling by $\bfQ^\infty$.
\begin{prop}[Variant of Prop.\ \ref{prop:Greg1}] \label{prop:Greg1pr}
There exists $\beta_3>0$ and $c_3<\infty$  such that for all $0<k<n$, $N\geq 2n+k$ and any $x,y\in\overline{\cC}_k$, 
we can find a coupling $\bfQ^\infty$ of $\bbfM^N$ and $\bbfM^N_{x,y}$ such that 
\begin{equation}\label{eq:Greg1pr}
\bfQ^\infty\big[(\ovg^1_N,\ovg^2_N)=_{3n/2} (\ovg^{1,x}_N,\ovg^{2,y}_N)\big]>1-c_32^{-\beta_3 n}.
\end{equation}
\end{prop}
\noindent We now explain briefly this variant holds.
\begin{itemize}
\item The constant in the separation lemmas (Lemmas 2.28 and 2.29 in \cite{Lawrecent}) stays unchanged and hence is uniform if one replaces paths from the origin by a pair of paths with different starting points, for it is inherited from Lemma 2.11, ibid., where the constant does not depend on (in the notation of that lemma in \cite{Lawrecent}) the choice of $A'$ as long as it is a subset of $C_n$;
\item Throughout the proof in \cite{Lawrecent} the probability of the coupling getting destroyed is always bounded by the probability that a (conditioned) random walk returns to the ball $\overline{\cC}_k$, see e.g. Lemma 2.32, ibid, hence the argument is still valid for the setup of Prop.\ \ref{prop:Greg1pr}.
\item Also, we note that the change from $n$ to $3n/2$ is merely for the convenience of the coherence of notations in this paper.
\end{itemize}

\begin{rem}
Although it is not needed in this work, we would like to mention that the coupling in {\rm \cite{Lawrecent}} actually works for even more general initial configurations which can just be two subsets of ${\cC}_k$ with a terminal point (in other words, starting points for the walks). For more discussion, see {\rm \cite{Natural}}. 
\end{rem}

\subsection{Proof of Proposition \ref{prop:coup1}}\label{sec:proof2}

We now give a proof of Prop.\ \ref{prop:coup1}  through fine-tuning the coupling in Prop.\ \ref{prop:Greg1pr}.

First, we claim that it suffices to prove the following coupling, which serves as a link between ILERW-SRW couplings of this work and the ILERW-ILERW couplings of \cite{Lawrecent}. For more comments, see the beginning of Section \ref{sec:Gregvar1}.

Pick $N>k>0$. Let $\gamma$ be an ILERW and $\wtg$ be a LERW stopped at exiting $\partial \cC_N$ with $\gamma(0)=\wtg(0)=0$, independent from each other. We write by $\bfN$ for its joint law and write $\widetilde\bfN$ for their joint law tilted by $G_N(\gamma_N,\wtg)$ (see \eqref{eq:tilting}) for the definition of $G_N$). 
 Similarly, let $\gamma^x$ be an ILERW and $\wtg^y$ be a LERW stopped at exiting $\partial \cC_N$ with $\gamma(0)=x$ and $\wtg(0)=y$, independent from each other. We write by $\bfN_{x,y}$ for its joint law and write $\widetilde{\bfN}_{x,y}$ for their joint law tilted by $G_N(\gamma_N^x,\wtg^y)$.

\begin{prop}\label{prop:coup1pp}
There exist $\beta_4>0$ and $c_4<\infty$ such that for all  $n,k>0$, $N\geq 2n+k+1$ and any $x,y\in\overline{\cC}_k$, there is a coupling $\widetilde\bfQ$ of $(\gamma,\wtg)$ under $\widetilde\bfN$ and $(\gamma^x,\wtg^y)$ under $\widetilde\bfN_{x,y}$, such that 
\begin{equation}\label{eq:coup1pp}
\widetilde\bfQ[(\gamma_N,\wtg)=_{3n/2+1} (\gamma^x_N,\wtg^y) ]\geq1-c_4 2^{-\beta_4 n}.
\end{equation}
\end{prop}

\begin{proof}[Proof of Prop.\ \ref{prop:coup1} given Prop.\ \ref{prop:coup1pp}]
Independently from $\widetilde\bfQ$, sample a loop soup $\cal L$ with intensity measure $m$ (see above \eqref{eq:Fetadef} for definition of $m$) and denote the product measure by $\bfQ$. We then add loops  from $\cal L$ that stay inside $\cC_N$ and do not touch $\gamma_N$ to $\wtg$, and add loops inside $\cC_N$ that do not touch $ \gamma_N^x$ to $\wtg^y$, according to the procedure described in Prop.\ \ref{prop:addloops}. Also,  to both $\wtg$ and $\wtg^y$ we attach an independent SRW that starts from the terminal point respectively and  denote the new, concatenated paths by $\lambda$ and $\lambda^y$ respectively.

We claim that $(\gamma,\lambda)$ and $(\gamma^x,\lambda^y)$ have the law of $\bbfP^N$ and $\bbfP^N_{x,y}$ respectively, as required in Prop.\ \ref{prop:coup1}. To verify this claim, it suffices to check the distribution of $(\gamma_N,\lambda_N)$ and $(\gamma^x_N,\lambda^y_N)$. For brevity we only check the first one.

Let 
\begin{equation}\label{eq:Fdef}
F(\gamma^\dagger,\iota^\dagger)=\bfN\big(\gamma_N=\gamma^\dagger,\wtg_N=\iota^\dagger \big)G_N(\gamma^\dagger,\iota^\dagger)
\end{equation}
be the ``energy'' function for $(\gamma_N,\wtg_N)$ under $\widetilde\bfN$. 
Similarly, let 
\begin{equation}\label{eq:Gdef}
H(\gamma^\dagger,\lambda^\dagger)= \bfP[\gamma_N=\gamma^\dagger, \lambda_N=\lambda^\dagger] 1_{\gamma^\dagger\cap\lambda^\dagger=\{0\}}
\end{equation}
be the the ``energy'' function for $(\gamma_N, \lambda_N)$ under $\bbfP_N$.

Given $\iota^\dagger$ and a loop soup ${\cal L}^\dagger$, let $
\lambda^\dagger(\iota^\dagger,{\cal L}^\dagger)$ stand for the path formed by adding back loops in ${\cal L}^\dagger$ to $\iota^\dagger$.
Then, to verify the claim above, it suffices to check 
\begin{equation}\label{eq:FH}
F(\gamma^\dagger,\iota^\dagger)=\sum_{{\cal L}^\dagger}H\big(\gamma^\dagger,\lambda^\dagger(\iota^\dagger,{\cal L}^\dagger)\big) P[{\cal L}={\cal L}^\dagger],
\end{equation}
where summation is over all possible realizations of a loop soup.
Here we 
let $P$ stand for the law of ${\cal L}$. Let $\mu^\infty$ be the law of an ILERW starting from 0 and $\mu_N$ be the law of LERW from 0 stopped at $T^N$. Thus, we can rewrite \eqref{eq:Fdef} as
\begin{equation}\label{eq:Fexpan}
F(\gamma^\dagger,\iota^\dagger)=\mu^\infty[\gamma^\dagger] \mu_N[\iota^\dagger] 1_{\gamma^\dagger\cap\iota^\dagger=\{0\}}e^{-L_N(\gamma^\dagger,\iota^\dagger)}
\end{equation}
Let $p_0$ be the law of simple random walk from $0$ stopped at first exiting $\cC_N$, then
\begin{equation}\label{eq:Hexpan}
H\big(\gamma^\dagger,\lambda^\dagger(\iota^\dagger,{\cal L}^\dagger)\big)=\mu^\infty[\gamma^\dagger]p_0\big[\lambda^\dagger(\iota^\dagger,{\cal L}^\dagger)\big] 1_{\gamma^\dagger\cap\lambda^\dagger=\{0\}}=\mu^\infty[\gamma^\dagger]\mu_N[\iota^\dagger]P[l={\cal L}^\dagger]1_{\gamma^\dagger\cap\iota^\dagger=\{0\}}1_{{\cal L}^\dagger(\gamma^\dagger,\iota^\dagger)=\emptyset},
\end{equation}
where ${\cal L}^\dagger(\gamma^\dagger,\iota^\dagger)$ stands for the set of loops that touch both $\gamma^\dagger$ and $\iota^\dagger$.
Comparing \eqref{eq:Fexpan} and \eqref{eq:Hexpan}, it suffices to show that given $\gamma^\dagger$ and $\iota^\dagger$ such that $\gamma^\dagger\cap\iota^\dagger=\{0\}$,
\begin{equation}
\sum_{{\cal L}^\dagger} P[{\cal L}={\cal L}^\dagger]1_{{\cal L}^\dagger(\gamma^\dagger,\iota^\dagger)=\emptyset}=e^{-L_N(\gamma^\dagger,\iota^\dagger)},
\end{equation}
where the summation is again over all possible realizations of a loop soup.
But this follows from the definition of $L_N$ and the restriction property of Poissonian loop soups. Hence, we have verified \eqref{eq:FH}.

Now it suffices to show that with the construction above, 
\begin{equation}\label{eq:lambdaeq}
\bfQ[\lambda_N=_{n+1}\lambda^{y}_N]\geq 1-c2^{-\beta n}.
\end{equation}
Note that in general $\LE\big[\lambda[T_{N-n-1},T_N]\big]\neq\gamma[T_{N-n-1},T_N]$. 

We observe that if $\lambda_N \neq_{n+1}\lambda^y_N$, then at least one of
$$
{\cal L}\big(\cC_{N-3n/2-1},\cC_{N-n-1}^c\big)\neq \emptyset\quad\mbox{ and }\quad
\big\{(\gamma_N,\wtg)\neq_{n+1} (\gamma^x_N,\wtg^y)\big\}
$$
must happen. 
We can bound the probability of the former by $c2^{-n/2}$ through a classical estimate on loop measures, see for instance Lemma 2.6 in \cite{Lawrecent} and that of the latter by $c2^{-\beta n}$ through \eqref{eq:coup1pp}.
This finishes the proof of \eqref{eq:lambdaeq} as well as Prop.\ \ref{prop:coup1}.
\end{proof}

We now turn to Prop.\ \ref{prop:coup1pp}. To construct the coupling in Prop.\ \ref{prop:coup1pp} for $N+1$ from Prop.\ \ref{prop:Greg1pr} for $N$, we tilt the law of $(\gamma_N,\wtg)$ and $(\gamma'_N,\wtg')$ from $\bfQ^\infty$ to $\widetilde\bfQ$ by ``extra loop terms'' (we will explain what this means immediately below). To show that under the new law $\widetilde\bfQ$ paths are also coupled in the sense of \eqref{eq:coup1pp} with high probability, it suffices to show that 
\begin{itemize}
\item[1)] the ``Radon-Nikodym'' derivative is uniformly bounded;
\item[2)] if paths have been coupled for many steps, then the ``Radon-Nikodym'' derivative should not differ too much and the laws of the ``tips'' we need to add from step $N$ to $N+1$ do not differ too much either.  
\end{itemize}

In order to describe the tilting procedure we need to introduce some notations.

As in the proof above, let $\mu^\infty_i$ be the law of an ILERW $\gamma^i$ started from the origin, $i=1,2$, and $\mu_N$ be the law of LERW $\gamma^2_N$ from 0 stopped at $T_N$. 

Let $\Gamma_N$ be the set of paths from the origin stopped at first exiting $\cC_N$. Let $\upsilon^i\in \Gamma_N$ and $\zeta^i=\upsilon^i\oplus\iota^{i}\in\Gamma_{N+1}$, $i=1,2$. We use bold fonts to denote a pair of paths, i.e.,  $\bullet=(\cdot^1,\cdot^2)$ as a shorthand. Thus the decomposition above is written as $\bfgam_{N+1}=\bfgam_N\oplus \bfiota$.

Note that the law $\widetilde\bfN$ can be written as follows: 
\begin{equation*}
\begin{split}
\widetilde\bfN(\bfgam_{N+1}=\bfzeta)\;=\;&\mu^\infty_1 [\gamma_{N+1}^1=\zeta^1]\mu_{N+1}[\gamma_{N+1}^2=\zeta^2] \frac{G_{N+1}(\bfzeta)}{\bfN[G_{N+1}(\bfgam_{N+1})]},
\end{split}
\end{equation*}
and the law $\bbfM^N$ can be written as follows.
$$
\bbfM^N(\bfgam_{N}=\bfup)=\mu^\infty_1 [\gamma_{N}^1=\upsilon^1]\mu^\infty_2[\gamma_N^2=\upsilon^2] \frac{G_N(\bfup)}{\bfM[G_N(\bfgam_N)]}.
$$
For $g\in \Gamma_{N+1}$, let $z$ be the terminal point of $g_N$ and decompose $g$ as $g_{N}\oplus \iota$ and define a new probability law of $\gamma_{N+1}$ by
$$
\mu^{\infty \prime}_2[\gamma^2_{N+1}=g]= \mu^\infty_2 [\gamma^2_{N}=g_{N}, p_{z}\big[\LE(W[0,T^{N+1}])=\iota \big| W[0,T^{N+1}]\cap g_{N}=\{z\}\big].
$$
where $p_z$ is the probability law of $W$, a simple random walk started from $z$. In other words, the law of $\mu^{\infty\prime}$ can be described as: take an ILERW, truncate it at first exit of $\cC_N$, then regard it as if it were part of $\gamma^2_{N+1}$ under $\mu_{N+1}$, and ``attach the tail'' through the conditional law under $\mu_{N+1}$.
Hence, for all $g\in\Gamma_N$
$$
\mu^{\infty\prime}_2[(\gamma^2_{N+1})_{N}=g]=
\mu^\infty[\gamma_{N}=g].
$$
Therefore, if we define
$$
\widehat{\bfM}[\gamma_{N+1}=\bfup]=\mu^\infty_1 [\gamma_{N+1}^1=\upsilon^1]\mu^{\infty\prime}_2[\gamma_{N+1}^2=\upsilon^2] \frac{G_N(\bfup_N)}{\bfM[G_N(\bfgam_N)]},
$$ 
then  $(\bfgam_{N+1})_{N}$ under $\widehat{\bfM}$ has the same marginal of $\bfgam_N$ under $\bbfM^N$.

We define $\bfgam^{x,y}_{N+1}\sim\widehat{\bfM}_{x,y}$ similarly.

For $\bfup\in \Gamma_{N+1} \times \Gamma_{N+1}$, we write the Radon-Nikodym derivative we need to investigate by
$$
Z(\bfup)=\frac{\widetilde\bfN(\bfgam_{N+1}=\bfup)}{\widehat{\bfM}(\bfgam_{N+1}=\bfup)}.
$$
We define $Z_{x,y}(\bfup^{x,y})$ similarly. 
As in Section 3 of \cite{Lawrecent}, we have the following properties of $Z$ and $Z_{x,y}$. We will only sketch its proof as it is very similar to Prop. 3.1 in \cite{Lawrecent}.
\begin{lem}\label{lem:tilt} There exists $\beta_5>0$, $c_5,C_5<\infty$ such that for all $\bfup, \bfup_{x,y}\in \Gamma_{N+1}\times \Gamma_{N+1}$
\begin{equation}\label{eq:tiltbad}
Z(\bfup), Z_{x,y}(\bfup_{x,y})\leq C_5.
\end{equation}
For $\bfup$ with $\widehat{\bfM}(\bfup)>0$ and $\bfup^{x,y}$ with $\widehat{\bfM}_{x,y}(\bfup^{x,y})>0$, if $\bfup=_n \bfup^{x,y}$, then
\begin{equation}\label{eq:tiltgood}
\bigg|\frac{Z(\bfup)}{Z_{x,y}(\bfup^{x,y})}-1\bigg|\leq c_5e^{-\beta_5 n}. 
\end{equation}
\end{lem}
\begin{proof}[Sketch of proof of Lem.\ \ref{lem:tilt}] To check \eqref{eq:tiltbad}, it suffices to see that both
\begin{equation}\label{eq:ZZ}
\frac{\mu_{N+1}[\gamma_{N}^2=\upsilon^2]}{\mu^\infty[\gamma_{N}^2=\upsilon^2]}\leq C \quad\mbox{ and }\quad \frac{\bfM[G_N(\bfgam_N)]}{\bfN[G_{N+1}(\bfgam_{N+1})]}\leq C'.
\end{equation}
as $G_{N+1}(\bfzeta)\leq G_N(\bfup)$ by definition.
The first claim of \eqref{eq:ZZ} follows Lemma \ref{lem:LEWILEW}. The second claim is a corollary of the asymptotics of one-point functions for LERW.

To check \eqref{eq:tiltgood}, it suffices to express both $Z$ and $Z_{x,y}$ in loop terms and see that if $\bfup=_n \bfup^{x,y}$, then the ratio $Z(\bfup)/Z_{x,y}(\bfup^{x,y})$ can be bounded by the exponential of loop terms of loops connecting $\cC_N^c$ and $\cC_{N-n}$, which gives the right-hand side of the inequality in \eqref{eq:tiltgood}.
\end{proof}

\medskip

Finally, Proposition \ref{prop:coup1pp} follows easily from Prop.\ \ref{prop:Greg1pr} and Lemma \ref{lem:tilt}. 
\begin{proof}[Proof of Prop.\ \ref{prop:coup1pp}]
 We start with $\bfQ^\infty$ from Prop.\ \ref{prop:Greg1pr}. First, sample $(\ovg_N^1,\ovg_N^2)$ and $(\ovg^{1,x}_N,\ovg^{2,y}_N)$ according to $\bfQ^\infty$. Attach to $\ovg_N^1$ an SRW conditioned to avoid $\ovg_N^1 $, erase loops, and stop at exiting $\cC_{N+1}$ and to $\ovg_N^2$ an SRW stopped at exiting $\cC_{N+1}$ conditioned to avoid $\ovg_N^2$, both independent from $(\ovg_N^1,\ovg_N^2)$ and of each other. We denote the pair of attached paths by $(\iota^1,\iota^2)$ and write 
$$\bfgam_{N+1}=(\gamma_{N+1}^1,\gamma_{N+1}^2)=(\ovg_N^1\oplus \iota^1,\ovg_N^2\oplus \iota^2).$$
Similarly, we attach to $(\ovg^{1,x}_N,\ovg^{2,y}_N)$ a pair of  $(\iota^{1,x},\iota^{2,y})$ and write 
$$\bfgam^{x,y}_{N+1}=(\gamma_{N+1}^{1,x},\gamma_{N+1}^{2,y})=(\ovg_N^{1,x}\oplus \iota^{1,x},\ovg_N^{2,y}\oplus \iota^{2,y}).$$

Then it is easy to see that $\bfgam_{N+1}$ and $\bfgam^{x,y}_{N+1}$ has the law of $\widehat{\bfM}$ and $\widehat{\bfM}_{x,y}$.

We now claim that it is still possible to couple $\bfgam_{N+1}$ and $\bfgam^{x,y}_{N+1}$ (with little abuse of notation we still call it  $\bfQ^\infty$) such that for some $\beta>0$,
\begin{equation}\label{eq:goodNp1}
\bfQ^\infty [\bfgam_{N+1}=_{3n/2+1}\bfgam^{x,y}_{N+1}] \geq 1-c 2^{-\beta n}.
\end{equation}

To prove this, it suffices to show that on the event
$\big\{(\ovg^1_N,\ovg^2_N)=_{3n/2} (\ovg^{1,x}_N,\ovg^{2,y}_N)\big\},
$
the conditional law of $(\iota^1,\iota^2)$ and $(\iota^{1,x},\iota^{2,y})$ under $\widehat{\bfM}$ and $\widehat{\bfM}_{x,y}$ respectively has a total variation distance uniformly bounded by $c2^{-3n/2}$. Here ``uniformly'' means regardless of actual configuration of $(\ovg^1_N,\ovg^2_N)$ and $(\ovg^{1,x}_N,\ovg^{2,y}_N)$. This follows from Lemma \ref{lem:LEWILEW}.

We finish by constructing a new measure $\widetilde\bfQ$ through ``tilting'' $\bfgam_{N+1}$ and $\bfgam_{N+1}^{x,y}$ in $\bfQ^\infty $ by $Z$ and $Z_{x,y}$ respectively. Then,
 \eqref{eq:goodNp1} combined with \eqref{eq:tiltbad} for the ``bad'' case $\bfeta_{N}\neq_{3n/2}\bfetag_{N}$ and \eqref{eq:tiltgood} for the good case $\bfeta_{N}=_{3n/2}\bfetag_{N}$ guarantees that for some $c_4<\infty$ and $\beta_4>0$,
\begin{equation}\label{eq:newmes}
\widetilde\bfQ[ \bfeta_{N+1}=_{3n/2+1}\bfetag_{N+1}] \geq 1- c_4 2^{-\beta_4 n}.
\end{equation}
This finishes the proof of \eqref{eq:coup1pp} (note that $N$ in the setup of \eqref{eq:coup1pp} is $N+1$ here).
\end{proof}
\begin{rem}\label{rem:doubletilt} The crucial observation here that leads to the proof above is that the ratio between $\mu^\infty[\gamma_N\in\cdot]$ and $\mu_{N+1}[\gamma_N\in\cdot]$ are uniformly bounded from above and below. This is not true for  $\mu^\infty[\gamma_N\in\cdot]$ and $\mu_N[\gamma_N\in\cdot]$. See also Remark \ref{rem:tip}.

\end{rem}

At the end of this subsection, we state another coupling which is related to but not a direct consequence of Prop.\ \ref{prop:coup1}.  Although we do not need it in this work, we still state it here as it is a strengthened version of Prop.\ \ref{prop:coup2} which should have a place in the family portrait of couplings that appear in this section. We will not provide its proof here but remark that it follows from a modification of the tilting arguments in the proof above. In this case, we will need to tilt both $\gamma^1_N$ and $\gamma^2_N$ and derive bounds similar to \eqref{eq:tiltbad} and \eqref{eq:tiltgood}.

In the notation of Proposition \ref{prop:coup2}, we consider $\bfeta^M$ under $\bbfP^{N,M}_{0,0}$ and for $x,y\in\overline{\cC}_k$  consider $\bfeta^{M,x,y}$ under the tilted law $\bbfP^{N,M}_{x,y}$. Then one has the following coupling.
\begin{prop}\label{prop:coup3}
There exist $\beta_6>0$ and $c_6<\infty$ such that for all $k,n>0$, $M\geq N\geq 2n+k$, any $x,y\in\overline{\cC}_k$, there is a coupling $\bbfQ^N$ of $\bfeta^M$ under $\bbfP^{N,M}_{0,0}$ and $\bfeta^{M,x,y}$ under $\bbfP^{N,M}_{x,y}$, such that 
\begin{equation}\label{eq:coup3}
\bbfQ^M[\bfeta^M=_{(n+M-N)} \bfeta^{M,x,y} ]\geq1-c_6 2^{-\beta_6 n}.
\end{equation}
\end{prop}

\section{One-point function estimates for LERW}\label{sec:length}
The goal of this section is to establish the main result of this work, namely Theorem \ref{ONE.POINT}. 
We lay out the main structure of the proof in Section \ref{sec:outlinelength}, and then give the proof two key propositions in Sections \ref{sec:5.2} and \ref{sec:5.3} respectively.

\subsection{Outline of the proof}\label{sec:outlinelength}
We start by a recap on the setup. Let $\mathbb{D}$ be the unit open ball in $\mathbb{R}^{3}$ and let $\overline{\mathbb{D}}$ be its closure. Fix $x \in \mathbb{D} \setminus \{0 \}$. We write $x_{n} $ for the nearest point from $2^{n} x$ in $\mathbb{Z}^{3}$. As introduced in Section \ref{sec:intro}, we are interested in 
\begin{equation}\label{one-point}
a_{n,x} := P \big( x_{n} \in\LE(S[0, T^n] ) \big),
\end{equation}
where $S$ is the SRW started from the origin and $T^n=T_{2^{n}}(S)$.
We first claim that in order to establish \eqref{one-goal}, it suffices to estimate Green's function and a non-intersection probability under a setup which is slightly different from that of Section \ref{sec:noninter}.

Let $X=X_n$ be a simple random walk started at $x_{n}$ conditioned that $\tau_{0}<T^n$, where $\tau_{0}$ stands for the first time that it hits the origin. When  no confusion arises, we write $\overline{X}^\circ$ for $\LE(X[0, \tau_{0}])$ as a shorthand and keep the dependence on $n$ implicit. As a convention, we will (and will only) omit the $X$ in the notation when it comes to hitting times for $X_n$. Let $Y$ be an independent simple random walk started at $x_n$ and write $Y^\circ$ for $Y[1,T^n(Y)]$.
\begin{lem} With the notation above, 
\begin{equation}\label{one}
a_{n,x} = G_{\cC_n} (0,x_n) P \Big(\LE(X[0, \tau_{0}]) \cap Y^\circ = \emptyset \Big).
\end{equation}
\end{lem}
\begin{proof}
Let $Z$ be a random walk started at the origin conditioned that it hits $x_n$ before hitting $\partial B(2^n)$, independent of $Y$. 
Write
\begin{equation}
u= \max \{ k \le T^{n} \ | \ Z(k) = x_n \}
\end{equation}
for the last time that $Z$ hits $x_n$ up to $T^{n}(Z)$. Then, by Proposition 8.1.1 of \cite{S}, we have 
\begin{equation}
a_n = G_{\cC_n} (0,x_n) P \Big(\LE(Z[0, u]) \cap Y^\circ = \emptyset \Big).
\end{equation}
Then \eqref{one} follows from the reversibility of LERW (see Lemma \ref{reversal}).
\end{proof}

\begin{rem}
The one-point function $a_n$ and the expected length $M_{2^n}$ are intimately related quantities. Loosely speaking, for a `typical' point $x$, in \eqref{one},  $G_{\cC_n} (0,x_n) \asymp2^{-n}$ while the non-intersection  probability in the RHS is comparable to $\Es(2^n)$. Thus, taking sum for $x \in \mathbb{Z}^{3}$, we see that $E (M_{2^n} )$ is comparable to $2^{2n} \Es(2^n)$ which gives an intuitive explanation for \eqref{esaymp}. 
\end{rem}
It is known (see Proposition 1.5.9 of \cite{Lawb}) that there exists a universal constant $a > 0$ such that 
\begin{align}\label{Green}
G_{\cC_n} (0, x_{n} ) &= \frac{a (1- |x| )}{2^{n} |x|}+ O \big( |x|^{-2} 2^{-2n} \big)=\frac{a (1- |x| )}{2^{n} |x|} \Big\{ 1 + O \Big( 2^{-n} |x|^{-1} (1- |x| )^{-1} \Big) \Big\}.
\end{align}
Thus, it suffices to estimate 
\begin{equation}\label{frac}
P \Big(\overline{X}^\circ\cap Y^\circ = \emptyset \Big)
= {P \Big( \tau_{0} < T^{n}, \ \overline{X}^\circ \cap Y^\circ = \emptyset \Big)}\Big/{P \Big( \tau_0 < T^{n} \Big)},
\end{equation}
By Proposition 1.5.10 of \cite{Lawb}, it follows that there exists a universal constant $b > 0$ such that 
\begin{equation}\label{Green2}
P \Big( \tau_{0} < T^{n} \Big)=\frac{b (1- |x| )}{2^{n} |x|} \Big[1 + O \Big( 2^{-n} |x|^{-1} (1- |x| )^{-1} \Big) \Big].
\end{equation}
(Compare this with \eqref{Green}.) Therefore, what we really need to estimate is the numerator of the fraction in \eqref{frac}.

\medskip

As in the proof of Proposition \ref{difscale}, we want to compare $LE (X_{n} [0, \tau_0^{X_n}]) $ and $LE (X_{n+1} [0, \tau^{X_{n+1}}_{0}]) $ via Theorem 5 of \cite{Koz}. We will accomplish this in two steps:
\begin{itemize}
\item[(A):] show that the shape of $\overline{X}^\circ  \cap \cC_{(1-q) n}  $ is not important for the probability of the numerator of the fraction in \eqref{frac} if $q \in (0,1)$ is chosen suitably;
\item[(B):] replace the starting points of $(X_{n}, Y_{n})$ and $(X_{n+1}, Y_{n+1})$ appropriately.
\end{itemize} 

We will first deal with part (A). 
The next lemma show that we may consider $\overline{X}^q:{=}\LE (X_{n} [0, T^{(1-q)n}]) $ instead of $\LE (X_n[0, \tau_0])$ (i.e.\ $\overline{X}^\circ$) for the non-intersection probability. As its proof is long and technical, we postpone its proof to Section \ref{sec:5.2}.
\begin{lem}\label{LEMMA}
There exists a universal constant $\delta > 0$ such that for all $n $, $q \in (0, 1)$ and $x \in \mathbb{D} \setminus \{ 0 \}$,
\begin{equation}\label{trimtip}
P \Big( \tau_0 < T^{n}, \ \overline{X}^\circ  \cap Y^\circ = \emptyset \Big)= P \Big( \tau_0 < T^{n}, \ \overline{X}^q \cap Y^\circ = \emptyset \Big) \big( 1 + O ( |x|^{-1} 2^{- \delta q n } ) \big).
\end{equation}
\end{lem}

%

\vspace{3mm}

We now discuss part (B). By the strong Markov property and Proposition 1.5.10 of \cite{Lawb}, we have
\begin{align}\label{Strong}
&\quad \;P \Big( \tau_0 < T^{n}, \ \overline{X}^q \cap Y^\circ = \emptyset \Big) \notag \\
&=\sum_{y \in \partial \cC_{(1-q)n}} P \Big( T^{(1-q)n} < T^{n}, \ X (T^{(1-q)n}) = y,  \overline{X}^q \cap Y^\circ = \emptyset \Big)  \cdot P^{y} \Big( \tau_{0}(W) < T^{n}(W) \Big) \notag  \\
&=c 2^{-(1-q)n } \big( 1 + O (2^{-qn} ) \big) P \Big( T^{(1-q)n} < T^{n}, \ \overline{X}^q \cap Y^\circ= \emptyset \Big),
\end{align}
where $P^y$ is the law of the SRW $W$ that starts from $y$ and $c > 0$ in the last line is a universal constant. Therefore, it suffices to estimate 
\begin{equation}
P \Big( T^{(1-q)n} < T^{n}, \ \overline{X}^q \cap Y^\circ= \emptyset \Big)
\end{equation}
i.e., we do not need to worry about $X[T^{(1-q)n}, \tau_{0} ]$. 

\vspace{4mm}

Let 
\begin{equation}
f_{n, x} = P \Big( \tau_0 < T^{n}, \ \overline{X}^q  \cap Y^\circ = \emptyset \Big) \label{fgnx} \quad\mbox{ and }\quad
 g_{n, x} = \frac{f_{n,x}}{f_{n-1,x}}. 
\end{equation}
Note that $\overline{X}^q$ and $Y^\circ$ implicitly depend on $n$.
We will show that there exist a universal constant $\rho > 0$ and a constant $c_{x} > 0$ depending only on $x$ such that for all $n$ 
\begin{equation}
f_{n,x} = c_{x} 2^{- (1 + \alpha ) n} \big( 1 + O_{x} ( 2^{- \rho n} ) \big)
\end{equation}
by proving that there exists universal constants $r > 0$ (in fact, $r = 2^{-(1 + \alpha )}$) and $\rho > 0$ such that 
\begin{equation}
g_{n, x} = r \big( 1 + O_{x} (2^{-\rho n } ) \big).
\end{equation}
This is in turn proved through the following proposition.
\begin{prop}\label{PROP}
Let $d_{x} = \min \{ |x|, 1- |x| \}$. There exist universal constants $c >0$, $\delta >0$ and $q_{0} > 0$ such that for all $n$ and $x \in \mathbb{D} \setminus \{ 0 \}$, if we let $q = q_{0}$,
\begin{equation}
g_{n, x} = g_{n-1, x}  \cdot \Big\{  1 + O \Big( d_{x}^{-c} 2^{- \delta q_{0} n} \Big) \Big\}. \label{eq32}
\end{equation}
\end{prop}

Now we are ready to prove the main theorem of this paper.
\begin{proof}[Proof of Theorem \ref{ONE.POINT}]Using Proposition 8.1.2 and Proposition 8.1.5 of \cite {S} together with Corollary \ref{COR}, it follows that there exist universal constants $a_{1}, a_{2} > 0$ and $N_{x} \in \mathbb{N}$ depending only on $x \in \mathbb{D} \setminus \{ 0 \}$ such that for all $n \ge N_{x}$ and $x \in \mathbb{D} \setminus \{ 0 \}$ with $|x| \in (0, \frac{1}{2} ]$,
\begin{equation}\label{eq33}
a_{1} 2^{-(1+ \alpha ) n} |x|^{-1-\alpha } \le f_{n, x} \le a_{2} 2^{-(1+ \alpha ) n} |x|^{-1-\alpha }, 
\end{equation}
and that for all $n \ge N_{x}$ and $x \in \mathbb{D} \setminus \{ 0 \}$ with $|x| \in [\frac{1}{2}, 1)$,
\begin{equation}\label{eq34}
a_{1} 2^{-(1+ \alpha ) n} (1-|x|)^{1-\alpha } \le f_{n, x} \le a_{2} 2^{-(1+ \alpha ) n} (1-|x|)^{1-\alpha }. 
\end{equation}
This shows that there exist universal constants $b_{1}, b_{2} > 0$ and $N_{x} \in \mathbb{N}$ depending only on $x \in \mathbb{D} \setminus \{ 0 \}$ such that for all $n \ge N_{x}$ and $x \in \mathbb{D} \setminus \{ 0 \}$.
\begin{equation}\label{eq35}
b_{1} \le g_{n, x} \le b_{2},
\end{equation}
It follows from \eqref{eq32} and \eqref{eq35} that 
$\{ g_{n, x} \}_{n \ge 1}$ is a Cauchy sequence for each $x \in \mathbb{D} \setminus \{ 0 \}$. So let 
\begin{equation}
r_{x} := \lim_{n \to \infty} g_{n,x}.
\end{equation}
 We know that $b_{1} \le r_{x} \le b_{2}$ for all $x \in \mathbb{D} \setminus \{ 0 \}$. Moreover, by \eqref{eq32}, we have
\begin{equation}
 g_{n, x} = r_{x} \Big\{  1 + O \Big( d_{x}^{-c} 2^{- \delta q_{0} n} \Big) \Big\}.
\end{equation}
However, by \eqref{eq33} and \eqref{eq34}, we see that for each $x \in \mathbb{D} \setminus \{ 0 \}$
\begin{equation}
0 < \liminf_{n \to \infty } 2^{(1 + \alpha ) n } f_{n, x} \le  \limsup_{n \to \infty } 2^{(1 + \alpha ) n } f_{n, x} < \infty.
\end{equation}
This ensures that $r_{x} = 2^{-(1 + \alpha )}$ for all $x \in \mathbb{D} \setminus \{ 0 \}$ and that
\begin{equation}\label{eq36}
f_{n, x} = c_{x} 2^{-(1 + \alpha ) n } \Big\{  1 + O \Big( d_{x}^{-c} 2^{- \delta q_{0} n} \Big) \Big\},
\end{equation}
for some $c_{x} > 0$ depending only on $x$. 

\vspace{2mm}

We recall (see \eqref{one}, \eqref{frac} and \eqref{trimtip}) that
\begin{equation*}
a_{n, x} = G_{\cC_n}(0, x_{n} ) \cdot \frac{f_{n, x}}{P^{x_{n}} \big( \tau_{0} < T^{n} \big) }\big( 1 + O ( |x|^{-1} 2^{- \delta q n } ) \big).
\end{equation*}
It follows from \eqref{Green} and \eqref{Green2} that 
\begin{equation*}
a_{n, x} = \frac{a}{b} \cdot f_{n,x} \cdot \Big( 1 + O \big( d_{x}^{-1} 2^{-n} \big) \Big)\big( 1 + O ( |x|^{-1} 2^{- \delta q n } ) \big).
\end{equation*}
Therefore, by \eqref{eq36}, we have
\begin{equation}\label{eq37}
a_{n, x} =  c_{x}' 2^{-(1 + \alpha ) n } \Big\{  1 + O \Big( d_{x}^{-c} 2^{- \delta q_{0} n} \Big) \Big\}.
\end{equation}
Here $c_{x}' = \frac{a}{b} \cdot c_{x}$. It follows from \eqref{eq33} and \eqref{eq34} that $c_{x}'$ satisfies
\begin{align}
& a_{1}'  |x|^{-1-\alpha } \le c_{x}' \le a_{2}'  |x|^{-1-\alpha } \  \ \ \ \  \   \ \ \ \ \  \ \  \Big(\text{if } 0 < |x| \le  \frac{1}{2} \Big)\\
& a_{1}' (1-|x|)^{1-\alpha } \le c_{x}' \le a_{2}' (1-|x|)^{1-\alpha } \ \  \Big(\text{if }  \frac{1}{2} \le |x| < 1 \Big),
\end{align}
where $a_{1}', a_{2}' > 0$ are universal constants. Thus, we finish the proof of the theorem. 
\end{proof}
%
%
%
%

\subsection{Proof of Lemma \ref{LEMMA}}\label{sec:5.2}
We define $k_{0} , k_{1} \in \mathbb{N}$ as follows (note that $d_{x} > 0$). 
\begin{itemize}
\item $k_{0}$ is a unique integer satisfying 
\begin{equation*}
2^{(1-q)n + k_{0} } \le {d_{x} 2^{n} }\big/{3} < 2^{(1-q)n + k_{0} +1}.
\end{equation*}

\item $k_{1}$ is the smallest integer satisfying 
\begin{equation*}
{d_{x} 2^{n-k_{1}} }\big/{3} < 1.
\end{equation*}
\item For $k \in \{ 0, 1, \cdots , k_{0} \}$, we write 
\begin{equation*}
D_{k} = \cC_{(1-q)n + k } .
\end{equation*}
\item For $k \in \{ 0, 1, \cdots , k_{1} \}$, we write 
\begin{equation*}
D'_{k} = B \Big( x_{n}, {d_{x} 2^{n-k }}\big/{3} \Big).
\end{equation*}
\end{itemize}
Note that 
\begin{align*}
&\cC_{(1-q)n  }  = D_{0} \subset D_{1} \subset \cdots \subset B \big( {d_{x} 2^{n-1} }\big/{3} \big)  \subset D_{k_{0}}  \subset B \big( {d_{x} 2^{n} }\big/{3} \big); \\
&\{ x_{n} \} = D'_{k_{1}} \subset D'_{k_{1} -1 } \subset \cdots \subset D'_{0} = B \Big( x_{n},{d_{x} 2^{n}}\big/{3} \Big); \quad D_{k_{0}} \cap D'_{0} = \emptyset \text{ and } D_{k_{0}} \cup D'_{0} \subset B (2^{n} ).
\end{align*} 
\begin{proof}[Proof of Lemma \ref{LEMMA}]
Suppose that $\tau_0 < T^{n}$. Then there are three cases for the shape of $X[T^{(1-q)n}, \tau_0]:{=}X^{\sim}$ as follows.
\begin{itemize}
\item[$\phantom{ }$]{\it Case} 1: $X^{\sim} \subset D_{k_{0}}$. In this case, we define 
\begin{equation*}
 k_{2} := \min \{ k \geq 0\ | \  X^{\sim} \subset D_{k}\}\in\{1, 2, \cdots , k_{0} \}.
\end{equation*}
 In other words, this is the case where $X^\sim \subset D_{k_{0}}$.
 
\item[$\phantom{ }$] {\it Case} 2: $X^{\sim} \not\subset D_{k_{0}}$ and $X^{\sim} \cap D'_{0} = \emptyset $.
 
\item[$\phantom{ }$]{\it Case} 3: $X^{\sim} \cap D'_{0} \neq \emptyset$. In this case, we define
 \begin{equation*}
 k_{3} := \max \{ k \le k_1 \ | \  X^{\sim} \cap D'_{k}\neq\emptyset\}\in \{ 0, 1, \cdots , k_{1} \}.
 \end{equation*}
\end{itemize}
Remind that $\overline{X}^q=\LE \big( X  [0, T^{(1-q)n } ]  \big) $ and $\overline{X}^\circ=\LE \big( X  [0, \tau_0]  \big)$. Let 
\begin{equation}\label{CASE}
P \Big( \tau_0 < T^{n},\; \overline{X}^\circ  \cap \overline{Y}^\circ = \emptyset \Big)= \sum_{i=1}^{3} P (H_{i} ),\mbox{ where }
H_{i} = \Big\{ \tau_0 < T^{n},\; \overline{X}^\circ  \cap \overline{Y}^\circ = \emptyset, \ {\it Case }\ i \Big\},\; i=1,2,3.
\end{equation}

We will first deal with $P(H_{1} )$. Note that 
\begin{equation*}
P (H_{1} ) = \sum_{k= 1}^{k_{0}} P (H_{1}, \ k_{2} = k ).
\end{equation*}
Suppose that $H_{1} \cap \{ k_{2} = k \}$ occurs. Then we see that 
\begin{equation*}
\overline{X}^q[0, T^{(1-q)n + k } ] = \overline{X}^\circ[0, T^{(1-q)n + k } ],
\end{equation*}
i.e., the loop-erased walk $\overline{X}^q$ up to the first time that it hits $\partial \cC_{(1-q)n + k }  $ coincides with that for $\overline{X}^\circ$ since $X^\sim$ does not ``destroy" the initial part of  $\overline{X}^q$. Therefore, 
\begin{equation*}
P (H_{1}, \ k_{2} = k ) \le c \big( d_{x} 2^{n} \big)^{-\alpha } \frac{1}{|x|} 2^{- q n} 2^{-k } 2^{-(1-q)n} = c   \big( d_{x} 2^{n} \big)^{-\alpha } \frac{1}{|x|}  2^{-k } 2^{-n}.
\end{equation*}
We remark that 
\begin{equation*}
P \Big( \tau_0 < T^{n}, \overline{X}^\circ  \cap \overline{Y}^\circ = \emptyset \Big)\asymp P \Big( \tau_0 < T^{n}, \overline{X}^q \cap \overline{Y}^\circ = \emptyset \Big) 
 \asymp \big( d_{x} 2^{n} \big)^{-\alpha } \frac{1}{|x|}   2^{-n}.
\end{equation*}
Thus, 
\begin{align*}
P (H_{1} ) = \sum_{k= 1}^{k_{0}} P (H_{1}, \ k_{2} = k ) \le P \Big( \tau_0 < T^{n}, \,   \overline{X}^q  [0, T^{(1-\frac{q}{2})n  } ] \cap \overline{Y}^\circ = \emptyset \Big) 
+ P \Big( \tau_0 < T^{n}, \, \overline{X}^q\cap \overline{Y}^\circ = \emptyset \Big) O \big( 2^{-\frac{q n}{2}} \big).
\end{align*}
However, it follows that 
\begin{align*}
P \Big( \tau_0 < T^{n}, \  \overline{X}^q [0, T^{(1-\frac{q}{2})n  } ] \cap \overline{Y}^\circ = \emptyset \Big) 
=P \Big( \tau_0 < T^{n}, \ \overline{X}^q\cap \overline{Y}^\circ = \emptyset \Big)  \Big( 1 + O \big( |x|^{-1} 2^{-\frac{qn}{4} } \big) \Big).
\end{align*}
Thus, we have
\begin{equation*}
P (H_{1} ) \le P \Big( \tau_0 < T^{n}, \ \overline{X}^q \cap \overline{Y}^\circ = \emptyset \Big)  \Big( 1 + O \big( |x|^{-1} 2^{-\frac{qn}{4} } \big) \Big).
\end{equation*}
Similarly, we have
\begin{align*}
&P (H_{2} ) \le P \Big( \tau_0 < T^{n}, \ \overline{X}^q \cap \overline{Y}^\circ = \emptyset \Big) O \big( 2^{-\frac{q n}{2}} \big) \mbox{ and }P (H_{3} ) \le P \Big( \tau_0 < T^{n}, \ \overline{X}^q \cap \overline{Y}^\circ = \emptyset \Big) O \big( 2^{-\frac{q n}{2}} \big).
\end{align*}
Thus, it follows that 
\begin{align*}
P \Big( \tau_0 < T^{n},\ \overline{X}^\circ  \cap \overline{Y}^\circ = \emptyset \Big) 
&\le P \Big( \tau_0 < T^{n}, \ \overline{X}^q \cap \overline{Y}^\circ = \emptyset \Big)  \Big( 1 + O \big( |x|^{-1} 2^{-\frac{qn}{4} } \big) \Big).
\end{align*}
On the other hand, it is not difficult to see the above inequality in the other direction as well. This completes the proof. 
\end{proof}
\subsection{Proof of Proposition \ref{PROP}}\label{sec:5.3}
As the proof is very similar to that in Section \ref{sec:noninter}, we will state results in a parallel way. As for the proof we will be brief and less pedagogical in the presentation of the argument. Notations are introduced right before the proposition where it first appears. The proof of Proposition \ref{PROP} is at the end of this subsection.

Recall that $q \in (0,1)$.

\begin{itemize}
\item Let  $B^{1}_{i, q} := B \Big(x_{n} , 2^{\{ 1- (2i-1) q \} n }\Big) $ for $i = 1,2, \cdots$; $B^{2}_{i, q} := B \Big( x_{n-1} ,2^{\{ 1- (2i-1) q \} n }\Big) $ for $i = 1,2, \cdots$.

\item Let $R^{1}_{1}, R^{1}_{2}$ and    $R^{2}_{1}, R^{2}_{2}$ be two paris of independent SRWs started from $x_{n}$ and $x_{n-1}$, respectively. We sometimes write $R^{1}_{2}$ and $R^{2}_{2}$ for $R^{1}_{2} [1, T^{n} ]$ and $R^{2}_{2} [1, T^{n-1}]$, respectively.

\item Let $T^{l}_{R^{i}_{1}} := \inf \{ k \ge 0 \ | \ R^{i}_{1} (k) \in \partial \cC_l) \}$ for $l \ge 1$ and $i =1,2$ and let
\begin{itemize}
\item[$\dagger$] $t^{1} = T^{n}_{R^{1}_{1}} \wedge T^{(1-q)n}_{R^{1}_{1}}$; $t^{2} = T^{n-1}_{R^{2}_{1}} \wedge T^{(1-q)n}_{R^{2}_{1}}$.
\end{itemize}

\item Write $\LE (R^{i}_{1} ) =\LE( R^{i}_{1}  [0, t^{i} ] )$ for $i=1,2$.

\item Let $F^{i} := \{ t^{i} = T^{(1-q)n}_{R^{i}_{1}} \}$ for $i=1,2$ and  write $Z^{i}$ for $R^{i}_{1}$ conditioned on $F^{i}$ for $i=1,2$.

\item Write $\LE (Z^{i} ) =\LE(Z^{i} [0, T^{(1-q)n } ] )$ for $i=1,2$.

\item For a path $\lambda$, let $U^{i}_{j} := \inf \{ k \ge 0 \ | \ \lambda (k) \in \partial B^{i}_{j, q} \}$ for $i=1,2$ and $j= 1,2, \cdots$.

\item Let $H^{i} := \Big\{\LE(Z^{i} ) [0, U^{i}_{1}] \cap R^{i}_{2} [1, U^{i}_{1} ] = \emptyset \Big\}$ for $i=1,2$.

\end{itemize}

The following proposition is similar to Lemma \ref{1st-lem}.
\begin{prop} One has
\begin{equation}\label{eq1}
g_{n, x} =\frac{ P (F^{1} ) P \Big( \LE(Z^{1} ) \cap R^{1}_{2} = \emptyset \ \Big| \ H^{1} \Big)  }{P(F^{2} ) P \Big(\LE(Z^{2} ) \cap R^{2}_{2} = \emptyset \ \Big| \ H^{2} \Big) } \cdot \Big\{  1 + O \Big(  d_{x}^{-1} 2^{- qn} \Big) \Big\}.
\end{equation}
\end{prop}
\begin{proof}

\vspace{2mm}

Note that by Lemma \ref{LEMMA} and \eqref{Strong}, we have 
\begin{equation}
f_{n, x} = c 2^{-(1-q) n } P \Big( F^{1}, \ \LE(R^{1}_{1} \cap R^{1}_{2} = \emptyset \Big) \Big( 1 + O \big( d_{x}^{-1} 2^{-\delta q n } \big)\Big),
\end{equation}
where $c > 0, \delta > 0$ are universal constants. Also we recall that $f_{n, x}$ is defined as in \eqref{fgnx}. Using the same constants $c, \delta$ as above, similarly we have
\begin{equation*}
f_{n-1, x} = c 2^{-(1-q) n } P \Big( F^{2}, \ LE(R^{2}_{1} \cap R^{2}_{2} = \emptyset \Big) \Big( 1 + O \big( d_{x}^{-1} 2^{-\delta q n } \big)\Big).
\end{equation*}
Therefore, we have
\begin{equation}\label{EQ}
g_{n, x} := \frac{f_{n, x}}{f_{n-1, x}} = \frac{ P \Big( F^{1}, \ \LE(R^{1}_{1} ) \cap R^{1}_{2} = \emptyset \Big)}{P \Big( F^{2}, \ \LE(R^{2}_{1} ) \cap R^{2}_{2} = \emptyset \Big)} \cdot  \Big\{  1 + O \Big( d_{x}^{-1} 2^{- \delta qn} \Big) \Big\}. 
\end{equation}

\vspace{2mm}

By Proposition 4.2 and 4.4 of \cite{Mas}, it follows that 
\begin{equation}
P (H^{1} ) = P (H^{2} ) \cdot \Big\{  1 + O \Big(  d_{x}^{-1} 2^{- qn} \Big) \Big\}.
\end{equation}
This gives
\begin{align}
\;&\frac{ P \Big( F^{1}, \ \LE(R^{1}_{1} ) \cap R^{1}_{2} = \emptyset \Big)}{P \Big( F^{2}, \ \LE(R^{2}_{1} ) \cap R^{2}_{2} = \emptyset \Big)} = \frac{ P (F^{1} ) P \Big( \LE(Z^{1} ) \cap R^{1}_{2} = \emptyset \Big)}{P(F^{2} ) P \Big(\LE(Z^{2} ) \cap R^{2}_{2} = \emptyset \Big)} \notag \\
=\;& \frac{ P (F^{1} ) P \Big( \LE(Z^{1} ) \cap R^{1}_{2} = \emptyset \ \Big| \ H^{1} \Big) P \Big( H^{1} \Big) }{P(F^{2} ) P \Big(\LE(Z^{2} ) \cap R^{2}_{2} = \emptyset \ \Big| H^{2} \Big) P \Big( H^{2} \Big)} \notag = \frac{ P (F^{1} ) P \Big( \LE(Z^{1} ) \cap R^{1}_{2} = \emptyset \ \Big| \ H^{1} \Big)  }{P(F^{2} ) P \Big(\LE(Z^{2} ) \cap R^{2}_{2} = \emptyset \ \Big| \ H^{2} \Big) } \cdot \Big\{  1 + O \Big(  d_{x}^{-1} 2^{- qn} \Big) \Big\},
\end{align}
which finishes the proof of \eqref{eq1}.
\end{proof}

We now decompose the conditional probabilities in \eqref{eq1}, just as in Section \ref{sec:3.3}. Before stating the parallel result, let us first introduce a few path spaces and probability measures associated with them.
\begin{itemize}

\item Let $\Pi^{1} = \{ (\gamma, \lambda ) \  | \ (\gamma, \lambda ) \text{ satisfies (i), (ii) and  (iii)}  \} $ be a set of pairs of paths $(\gamma, \lambda )$ satisfying
\begin{align*}
&\text{(i) }  \gamma (0), \lambda (0) \in \partial B^{1}_{2,q} \text{ and }   \gamma \big(  \text{len} (\gamma) \big) , \lambda \big( \text{len} (\lambda) \big) \in \partial B^{1}_{1,q}; \\
&\text{(ii) } \gamma \big[ 0, \text{len} (\gamma)-1 \big] \subset  B^{1}_{1,q} \text{ and } \lambda \big[ 0, \text{len} (\lambda)-1 \big] \subset  B^{1}_{1,q}; \\
&\text{(iii) } \gamma \cap \lambda = \emptyset.
\end{align*}

\item Let $\Pi^{2} = \{ (\gamma, \lambda ) \  | \ (\gamma, \lambda ) \text{ satisfies (i'), (ii') and  (iii')}  \} $ be a set of pairs of paths $(\gamma, \lambda )$ satisfying
\begin{align*}
&\text{(i') }  \gamma (0), \lambda (0) \in \partial B^{2}_{2,q} \text{ and }   \gamma \big(  \text{len} (\gamma) \big) , \lambda \big( \text{len} (\lambda) \big) \in \partial B^{2}_{1,q}; \\
&\text{(ii') } \gamma \big[ 0, \text{len} (\gamma)-1 \big] \subset  B^{2}_{1,q} \text{ and } \lambda \big[ 0, \text{len} (\lambda)-1 \big] \subset  B^{2}_{1,q}; \\
&\text{(iii') } \gamma \cap \lambda = \emptyset.
\end{align*}

\item Take $ (\gamma, \lambda ) \in \Pi^{1} $. Define $g^{1} (\gamma, \lambda )$ by 
\begin{equation*}
g^{1} (\gamma, \lambda ) = P \Big( \big(\LE(Z_{1}') \cup \gamma \big) \cap \big( R_{1}' \cup \lambda \big) = \emptyset \Big),\mbox{ where}
\end{equation*}
\begin{itemize}
\item $Z_{1}$ is a random walk started at $\gamma \big(  \text{len} (\gamma) \big)$ conditioned not to exit from $\cC_n$ before hitting $\partial \cC_{(1-q)n} $;
\item $Z_{1}'$ is $Z_{1}$ conditioned that $Z_{1}[1, T^{(1-q)n}] \cap \gamma = \emptyset$;
\item $R_{1}'$ is a SRW started at $\lambda \big( \text{len} (\lambda) \big)$ which is independent of $Z_{1}'$.
\end{itemize}
\item Take $ (\gamma, \lambda ) \in \Pi^{2} $. Define $g^{2} (\gamma, \lambda )$ by 
\begin{equation*}
g^{2} (\gamma, \lambda ) = P \Big( \big(\LE(Z_{2}') \cup \gamma \big) \cap \big( R_{2}' \cup \lambda \big) = \emptyset \Big),\mbox{ where}
\end{equation*}
\begin{itemize}
\item $Z_{2}$ is a random walk started at $\gamma \big(  \text{len} (\gamma) \big)$ conditioned not to exit from $\cC_{n-1}$ before hitting $\partial \cC_{(1-q)n} $;
\item $Z_{2}'$ is $Z_{2}$ conditioned that $Z_{2}[1, T^{(1-q)n}] \cap \gamma = \emptyset$;
\item $R_{2}'$ is a SRW started at $\lambda \big( \text{len} (\lambda) \big)$ which is independent of $Z_{2}'$.
\end{itemize}
%
%
%
\item Let $J^{1}_{1}, J^{1}_{2}$ be two independent SRW's started at $x_{n}$. For $ (\gamma, \lambda ) \in \Pi^{1} $, let
\begin{equation*}
\mu^{1}_{0}  (\gamma, \lambda )  = P \Big\{ \Big(\LE(J^{1}_{1}  )  \big[ U^{1}_{2}, U^{1}_{1} \big],  J^{1}_{2} \big[ U^{1}_{2}, U^{1}_{1} \big] \Big)  = (\gamma, \lambda )  \ \Big| \ H^{1}_{0} \Big\},\mbox{ where}
\end{equation*}
\begin{itemize}
\item $u= \inf \{ k \ge 0 \ | \ J^{1}_{1} (k) \in \partial B \big( x_{n}, d_{x} 2^{n} \big) \}$;
\item $\LE (J^{1}_{1} ) =\LE(J^{1}_{1} [0, u] )$;
\item $H^{1}_{0} = \{\LE(J^{1}_{1}  )  \big[ 0, U^{1}_{1} \big] \cap J^{1}_{2} \big[ 1, U^{1}_{1} \big] = \emptyset \}$.
\end{itemize}

\item Let $J^{2}_{1}, J^{2}_{2}$ be two independent SRW's started at $x_{n-1}$. For $ (\gamma, \lambda ) \in \Pi^{2} $, let
\begin{equation*}
\mu^{2}_{0}  (\gamma, \lambda )  = P \Big\{ \Big(\LE(J^{2}_{1}  )  \big[ U^{2}_{2}, U^{2}_{1} \big],  J^{2}_{2} \big[ U^{2}_{2}, U^{2}_{1} \big] \Big)  = (\gamma, \lambda )  \ \Big| \ H^{2}_{0} \Big\},\mbox{ where}
\end{equation*}
\begin{itemize}
\item $t= \inf \{ k \ge 0 \ | \ J^{2}_{1} (k) \in \partial B \big( x_{n-1}, d_{x} 2^{n-1} \big) \}$;
\item $\LE (J^{2}_{1} ) =\LE(J^{2}_{1} [0, t] )$;
\item $H^{2}_{0} = \{\LE(J^{2}_{1}  )  \big[ 0, U^{2}_{1} \big] \cap J^{2}_{2} \big[ 1, U^{2}_{1} \big] = \emptyset \}$.
\end{itemize}

\item Let $L^{1}_{1}, L^{1}_{2}$ be two independent SRW's started at $x^{1}_{q}$ and $y^{1}_{q}$. For $ (\gamma, \lambda ) \in \Pi^{1} $, let
\begin{equation*}
\nu^{1}_{0}  (\gamma, \lambda )  = P \Big\{ \Big(\LE(L^{1}_{1}  )  \big[ U^{1}_{2}, U^{1}_{1} \big],  L^{1}_{2} \big[ U^{1}_{2}, U^{1}_{1} \big] \Big)  = (\gamma, \lambda )  \ \Big| \ I^{1}_{0} \Big\},\mbox{ where}
\end{equation*}
\begin{itemize}
\item $u'= \inf \{ k \ge 0 \ | \ L^{1}_{1} (k) \in \partial B \big( x_{n}, d_{x} 2^{n} \big) \}$;
\item $\LE (L^{1}_{1} ) =\LE(L^{1}_{1} [0, u'] )$;
\item $I^{1}_{0} = \{\LE(L^{1}_{1}  )  \big[ 0, U^{1}_{1} \big] \cap L^{1}_{2} \big[ 1, U^{1}_{1} \big] = \emptyset \}$.
\end{itemize}

\item Let $L^{2}_{1}, L^{2}_{2}$ be two independent SRW's started at $x^{2}_{q}$ and $y^{2}_{q}$. For $ (\gamma, \lambda ) \in \Pi^{2} $, let
\begin{equation*}
\nu^{2}_{0}  (\gamma, \lambda )  = P \Big\{ \Big(\LE(L^{2}_{1}  )  \big[ U^{2}_{2}, U^{2}_{1} \big],  L^{2}_{2} \big[ U^{2}_{2}, U^{2}_{1} \big] \Big)  = (\gamma, \lambda )  \ \Big| \ I^{2}_{0} \Big\},\mbox{ where}
\end{equation*}
\begin{itemize}
\item $u''= \inf \{ k \ge 0 \ | \ L^{2}_{1} (k) \in \partial B \big( x_{n-1}, d_{x} 2^{n-1} \big) \}$;
\item $\LE (L^{2}_{1} ) =\LE(L^{2}_{1} [0, u''] )$;
\item $I^{2}_{0} = \{\LE(L^{2}_{1}  )  \big[ 0, U^{2}_{1} \big] \cap L^{2}_{2} \big[ 1, U^{2}_{1} \big] = \emptyset \}$.
\end{itemize}

\end{itemize}

We are now ready to state and prove the decomposition result which is  parallel to \eqref{henkei-4}.
\begin{prop}\label{par4}
For some universal constants $c, \delta > 0$, one has
\begin{align}
&P \Big( \LE(Z^{1} ) \cap R^{1}_{2} = \emptyset \ \Big| \ H^{1} \Big) = \Big( 1 + O \big( d_{x}^{-c} 2^{-\delta q n} \big) \Big) \sum_{(\gamma, \lambda) \in \Pi^{1} } g^{1} (\gamma, \lambda ) \nu^{1}_{0}  (\gamma, \lambda ); \label{eq16} \\
&P \Big( \LE(Z^{2} ) \cap R^{2}_{2} = \emptyset \ \Big| \ H^{2} \Big) = \Big( 1 + O \big( d_{x}^{-c} 2^{-\delta q n} \big) \Big) \sum_{(\gamma, \lambda) \in \Pi^{2} } g^{2} (\gamma, \lambda ) \nu^{2}_{0}  (\gamma, \lambda ), \label{eq17}
\end{align}

\end{prop}
\begin{proof}

\vspace{2mm}

By the same argument as in the proof of Corollary \ref{cor1}, and using Proposition 4.2 and 4.4 of \cite{Mas} again, it follows that there exists a universal constant $c < \infty$ such that for all $n$ and $q \in (0,1)$
\begin{align}
&P \Big( \LE(Z^{1} ) \cap R^{1}_{2} = \emptyset \ \Big| \ H^{1} \Big) = \Big( 1 + O \big( d_{x}^{-c} 2^{-\frac{q n}{2}} \big) \Big) \sum_{(\gamma, \lambda) \in \Pi^{1} } g^{1} (\gamma, \lambda ) \mu^{1}_{0}  (\gamma, \lambda ); \label{eq2} \\
&P \Big( \LE(Z^{2} ) \cap R^{2}_{2} = \emptyset \ \Big| \ H^{2} \Big) = \Big( 1 + O \big( d_{x}^{-c} 2^{-\frac{q n}{2}} \big) \Big) \sum_{(\gamma, \lambda) \in \Pi^{2} } g^{2} (\gamma, \lambda ) \mu^{2}_{0} (\gamma, \lambda ). \label{eq3}
\end{align}

%
%
%
%

\vspace{2mm}

Again, by Prop.\ \ref{prop:coup2}, it follows that there exists universal constants $\delta > 0$ and $C$ such that 
\begin{align}
\parallel  \mu^{1}_{0} - \nu^{1}_{0} \parallel_{\rm TV}   \le C 2^{- \delta q n}, \label{eq8} \mbox{ and } \parallel \mu^{2}_{0} - \nu^{2}_{0} \parallel_{\rm TV}  \le C  2^{- \delta q n}. 
\end{align}
(See also \eqref{couple}). 
The separation lemma (see (6.13) of \cite{S} for exact form of the separation lemma we need) ensures that $(\gamma, \lambda) $ is well-separated with positive probability with respect to $\mu^{1}_{0} $. For such a pair $(\gamma, \lambda) $, it follows that $g^{1} (\gamma, \lambda ) \ge c d_{x}^{-\alpha} 2^{-\alpha q n}$ for some universal constant $c > 0$ (see Lemma \ref{3rd-lem} for the use of the separation lemma). On the other hand, using the same argument as in the proof of Lemma \ref{2nd-lem}, we see that for all $(\gamma, \lambda) \in \Pi^{1}$, $g^{1} (\gamma, \lambda ) \le C d_{x}^{-\alpha} 2^{-\alpha q n}$ for some universal constant $C< \infty$. Therefore, we have 
\begin{equation}
\sum_{ (\gamma, \lambda) \in \Pi^{1} } g^{1} (\gamma, \lambda ) \mu^{1}_{0}  (\gamma, \lambda ) \ge c d_{x}^{-\alpha} 2^{-\alpha q n}\mbox{, and }\label{eq10} g^{1} (\gamma, \lambda ) \le C d_{x}^{-\alpha} 2^{-\alpha q n} \text { for all } (\gamma, \lambda) \in \Pi^{1}. 
\end{equation}
Similarly, we see that 
\begin{equation}
\sum_{ (\gamma, \lambda) \in \Pi^{2} } g^{2} (\gamma, \lambda ) \mu^{2}_{0}  (\gamma, \lambda ) \ge c d_{x}^{-\alpha} 2^{-\alpha q n}\mbox{, and } \label{eq12} 
g^{2} (\gamma, \lambda ) \le C d_{x}^{-\alpha} 2^{-\alpha q n} \text { for all } (\gamma, \lambda) \in \Pi^{2}. 
\end{equation}
Combining \eqref{eq8} with \eqref{eq10}
, we have
\begin{align*}
&\;\Big|  \sum_{ (\gamma, \lambda) \in \Pi^{1} } g^{1} (\gamma, \lambda ) \mu^{1}_{0}  (\gamma, \lambda ) - \sum_{ (\gamma, \lambda) \in \Pi^{1} } g^{1} (\gamma, \lambda ) \nu^{1}_{0}  (\gamma, \lambda ) \Big|  \le \sum_{ (\gamma, \lambda) \in \Pi^{1} } g^{1} (\gamma, \lambda ) \big| \mu^{1}_{0}  (\gamma, \lambda )- \nu^{1}_{0}  (\gamma, \lambda ) \big| \\
\le &\;C d_{x}^{-\alpha} 2^{-\alpha q n} \sum_{ (\gamma, \lambda) \in \Pi^{1} }  \big| \mu^{1}_{0}  (\gamma, \lambda )- \nu^{1}_{0}  (\gamma, \lambda ) \big| \notag \le C d_{x}^{-\alpha} 2^{-\alpha q n} 2^{-\delta q n}  \le C 2^{-\delta q n}  \sum_{ (\gamma, \lambda) \in \Pi^{1} } g^{1} (\gamma, \lambda ) \mu^{1}_{0}  (\gamma, \lambda ), 
\end{align*}
i.e.,
\begin{equation}\label{eq14}
\sum_{ (\gamma, \lambda) \in \Pi^{1} } g^{1} (\gamma, \lambda ) \nu^{1}_{0}  (\gamma, \lambda ) = 
 \sum_{ (\gamma, \lambda) \in \Pi^{1} } g^{1} (\gamma, \lambda ) \mu^{1}_{0}  (\gamma, \lambda ) \Big( 1 + O \big( 2^{- \delta q n } \big) \Big).
\end{equation}

Similarly, we have
\begin{equation}\label{eq15}
\sum_{ (\gamma, \lambda) \in \Pi^{2} } g^{2} (\gamma, \lambda ) \nu^{2}_{0}  (\gamma, \lambda ) = 
 \sum_{ (\gamma, \lambda) \in \Pi^{2} } g^{2} (\gamma, \lambda ) \mu^{2}_{0}  (\gamma, \lambda ) \Big( 1 + O \big( 2^{- \delta q n } \big) \Big).
 \end{equation}
Therefore, \eqref{eq16} and \eqref{eq17} follow.
\end{proof}

We now replace the starting points of the walks. Again, we start with notations.
\begin{itemize}
\item Let
\begin{itemize}
\item[$\dagger$] $x^{1}_{q} := x_{n} - \big( 2^{(1-5q) n }, 0 , 0 \big)$ and $y^{1}_{q} := x_{n} + \big( 2^{(1-5q) n }, 0 , 0 \big)$;

\item[$\dagger$] $x^{2}_{q} := x_{n-1} - \big( 2^{(1-5q) n }, 0 , 0 \big)$ and $y^{2}_{q} := x_{n-1} + \big( 2^{(1-5q) n }, 0 , 0 \big)$;

\item[$\dagger$] $x^{3}_{q} := x_{n-1} - \big( 2^{(1-5q) n -1 }, 0 , 0 \big)$ and $y^{3}_{q} := x_{n-1} + \big( 2^{(1-5q) n -1 }, 0 , 0 \big)$;

\item[$\dagger$] $x^{4}_{q} := x_{n-2} - \big( 2^{(1-5q) n -1 }, 0 , 0 \big)$ and $y^{4}_{q} := x_{n-2} + \big( 2^{(1-5q) n -1 }, 0 , 0 \big)$,
\end{itemize}
and define pairs of independent SRW's started from these points:
\begin{itemize}
\item[$\dagger$] let $S^{1}_{1}, S^{1}_{2}$ start from $x^{1}_{q}$ and $y^{1}_{q}$; let $S^{2}_{1}, S^{2}_{2}$ start from $x^{2}_{q}$ and $y^{2}_{q}$;

\item[$\dagger$] let  $S^{3}_{1}, S^{3}_{2}$ start from $x^{3}_{q}$ and $y^{3}_{q}$; let $S^{4}_{1}, S^{4}_{2}$ start from $x^{4}_{q}$ and $y^{4}_{q}$.
\end{itemize}
We sometimes write $S^{1}_{2}$,  $S^{2}_{2}$, $S^{3}_{2}$ and $S^{4}_{2}$ for  $S^{1}_{2} [0, T^{n} ]$,  $S^{2}_{2} [0, T^{n-1}]$, $S^{3}_{2} [0, T^{n-1}]$, and $S^{4}_{2} [0, T^{n-2}]$, respectively.

\item Write $T^{l}_{S^{i}_{1}} := \inf \{ k \ge 0 \ | \ S^{i}_{1} (k) \in \partial \cC_l ) \}$ for $l \ge 1$ and $i =1,2, 3, 4$, and let 
\begin{itemize}
\item[$\dagger$] $u^{1} = T^{n}_{S^{1}_{1}} \wedge T^{(1-q)n}_{S^{1}_{1}}$;  $u^{2} = T^{n-1}_{S^{2}_{1}} \wedge T^{(1-q)n}_{S^{2}_{1}}$; $u^{3} = T^{n-1}_{S^{3}_{1}} \wedge T^{(1-q)n-1}_{S^{3}_{1}}$;  $u^{4} = T^{n-2}_{S^{4}_{1}} \wedge T^{(1-q)n-1}_{S^{4}_{1}}$.
\end{itemize}
We write $\LE (S^{i}_{1} ) =\LE( S^{i}_{1}  [0, u^{i} ] )$ for $i=1,2, 3, 4$.

\item Let 
\begin{align*}
&F^{3} = \Big\{ \text{SRW started at } x_{n-1} \text{ hits } \partial \cC_{(1-q)n -1}  \text{ before hitting } \partial \cC_{n-1} \Big\}, \\
&F^{4} = \Big\{ \text{SRW started at } x_{n-2} \text{ hits } \partial  \cC_{(1-q)n -1}  \text{ before hitting } \partial \cC_{n-2} \Big\}.
\end{align*}

\item Let $G^{i} := \{ u^{i} = T^{(1-q)n}_{S^{i}_{1}} \}$ for $i=1,2$ and  $G^{i} := \{ u^{i} = T^{(1-q)n-1}_{S^{i}_{1}} \}$ for $i=3,4$.

\item Let $W^{i}$ be $S^{i}_{1}$ conditioned on $G^{i}$ for $i=1,2$.

\item Write $\LE (W^{i} ) =\LE(W^{i} [0, T^{(1-q)n } ] )$ for $i=1,2$.

\item Remind that for a path $\lambda$, we write $U^{i}_{j} := \inf \{ k \ge 0 \ | \ \lambda (k) \in \partial B^{i}_{j, q} \}$ for $i=1,2$ and $j= 1,2, \cdots$.

\item Let $V^{i} := \Big\{\LE(W^{i} ) [0, U^{i}_{1}] \cap S^{i}_{2} [1, U^{i}_{1} ] = \emptyset \Big\}$ for $i=1,2$.

%
%
%

\end{itemize}

We are now ready to state the decomposition result similar to Prop.\ \ref{par4} for the walks introduced above. Again we omit the proof for brevity. Compare this with the \eqref{henkei-5} (versus \eqref{henkei-4}).
\begin{prop} For some universal constants $c, \delta > 0$, one has
\begin{align}
&P \Big( \LE(W^{1} ) \cap S^{1}_{2} = \emptyset \ \Big| \ V^{1} \Big) = \Big( 1 + O \big( d_{x}^{-c} 2^{-\frac{q n}{2}} \big) \Big) \sum_{(\gamma, \lambda) \in \Pi^{1} } g^{1} (\gamma, \lambda ) \nu^{1}_{0}  (\gamma, \lambda ), \label{eq22} \\
&P \Big( \LE(W^{2} ) \cap S^{2}_{2} = \emptyset \ \Big| \ V^{2} \Big) = \Big( 1 + O \big( d_{x}^{-c} 2^{-\frac{q n}{2}} \big) \Big) \sum_{(\gamma, \lambda) \in \Pi^{2} } g^{2} (\gamma, \lambda ) \nu^{2}_{0}  (\gamma, \lambda ). \label{eq23}
\end{align}
\end{prop}

Now we can change the starting points. The following proposition is parallel to \eqref{CHANGE} and \eqref{CHANGE-2}.
\begin{prop} We have
\begin{equation}\label{eq26}
g_{n, x} =\frac{ P (F^{1} ) P  \Big(\LE(S^{1}_{1} ) \cap S^{1}_{2} = \emptyset,  G^{1} \Big)  P (G^{1})^{-1}  }{P(F^{2} ) P \Big(\LE(S^{2}_{1} ) \cap S^{2}_{2} = \emptyset,  G^{2} \Big) P (G^{2})^{-1}  } \cdot \Big[  1 + O \Big( d_{x}^{-c} 2^{- \delta qn} \Big) \Big],
\end{equation}
and
\begin{equation}\label{eq27}
g_{n-1, x}  =\frac{ P (F^{3} ) P  \Big(\LE(S^{3}_{1} ) \cap S^{3}_{2} = \emptyset,  G^{3} \Big)  P (G^{3})^{-1}  }{P(F^{4} ) P \Big(\LE(S^{4}_{1} ) \cap S^{4}_{2} = \emptyset,  G^{4} \Big) P (G^{4})^{-1}  } \cdot \Big[  1 + O \Big( d_{x}^{-c} 2^{- \delta qn} \Big) \Big]. 
\end{equation}
\end{prop}
\begin{proof}
 By \eqref{eq16}, \eqref{eq17}, \eqref{eq22} and \eqref{eq23}, we have 
\begin{align}
& \frac{     P \Big( \LE(Z^{1} ) \cap R^{1}_{2} = \emptyset \ \Big| \ H^{1} \Big)      }{    P \Big(\LE(Z^{2} ) \cap R^{2}_{2} = \emptyset \ \Big| H^{2} \Big)      } 
= \frac{    P \Big( \LE(W^{1} ) \cap S^{1}_{2} = \emptyset \ \Big| \ V^{1} \Big)    }{       P \Big( \LE(W^{2} ) \cap S^{2}_{2} = \emptyset \ \Big| \ V^{2} \Big)    } \cdot \Big\{  1 + O \Big( d_{x}^{-c} 2^{- \delta qn} \Big) \Big\} \notag \\
=&\frac{    P \Big( \LE(W^{1} ) \cap S^{1}_{2} = \emptyset  \Big)  P(V^{1} )  }{       P \Big( \LE(W^{2} ) \cap S^{2}_{2} = \emptyset  \Big) P (V^{2})   } \cdot \Big\{  1 + O \Big( d_{x}^{-c} 2^{- \delta qn} \Big) \Big\} =\frac{    P \Big( \LE(W^{1} ) \cap S^{1}_{2} = \emptyset  \Big)    }{       P \Big( \LE(W^{2} ) \cap S^{2}_{2} = \emptyset  \Big)    } \cdot \Big\{  1 + O \Big(  d_{x}^{-c} 2^{- \delta qn} \Big) \Big\}, \notag
\end{align}
where in the last equality we used the following fact
\begin{equation}\label{eq24}
P (V^{1} ) = \Big( 1 + O \big( d_{x}^{-1} 2^{- q n} \big) \Big) P (V^{2} )
\end{equation}
which again follows from Propositions 4.2 and 4.4 of \cite{Mas}.
Combining this with \eqref{EQ}, it follows that  
\begin{align}
g_{n, x} &= \frac{f_{n, x}}{f_{n-1, x}} = \frac{ P \Big( F^{1}, \ \LE(R^{1}_{1} ) \cap R^{1}_{2} = \emptyset \Big)}{P \Big( F^{2}, \ \LE(R^{2}_{1} ) \cap R^{2}_{2} = \emptyset \Big)} \cdot  \Big\{  1 + O \Big( d_{x}^{-c} 2^{- \delta qn} \Big) \Big\} \notag \\
&=\frac{ P (F^{1} ) P \Big( \LE(Z^{1} ) \cap R^{1}_{2} = \emptyset \ \Big| \ H^{1} \Big)  }{P(F^{2} ) P \Big(\LE(Z^{2} ) \cap R^{2}_{2} = \emptyset \ \Big| \ H^{2} \Big) } \cdot \Big\{  1 + O \Big( d_{x}^{-c} 2^{- \delta qn} \Big) \Big\} \notag \\
&=\frac{ P (F^{1} ) P  \Big(\LE(W^{1} ) \cap S^{1}_{2} = \emptyset \Big)   }{P(F^{2} ) P \Big(\LE(W^{2} ) \cap S^{2}_{2} = \emptyset \Big)  } \cdot \Big\{  1 + O \Big( d_{x}^{-c} 2^{- \delta qn} \Big) \Big\} \notag \\
&=\frac{ P (F^{1} ) P  \Big(\LE(S^{1}_{1} ) \cap S^{1}_{2} = \emptyset, \ G^{1} \Big)  P (G^{1})^{-1}  }{P(F^{2} ) P \Big(\LE(S^{2}_{1} ) \cap S^{2}_{2} = \emptyset, \ G^{2} \Big) P (G^{2})^{-1}  } \cdot \Big\{  1 + O \Big( d_{x}^{-c} 2^{- \delta qn} \Big) \Big\}. 
\end{align}
This gives \eqref{eq26}. The claim \eqref{eq27} follows similarly.
\end{proof}
Note that, an easy consequence of Proposition 1.5.10 of \cite{Lawb} is 
\begin{align}
P (F^{1} )  = P (F^{3} )   \cdot \Big\{  1 + O \Big( d_{x}^{-c} 2^{- \delta qn} \Big) \Big\}; &\quad
P (F^{2} )  = P (F^{4} )   \cdot \Big\{  1 + O \Big( d_{x}^{-c} 2^{- \delta qn} \Big) \Big\}; \label{eq29} \\
P (G^{1} )  = P (G^{3} )   \cdot \Big\{  1 + O \Big( d_{x}^{-c} 2^{- \delta qn} \Big) \Big\}; &\quad
P (G^{2} )  = P (G^{4} )   \cdot \Big\{  1 + O \Big( d_{x}^{-c} 2^{- \delta qn} \Big) \Big\}. \label{eq31}
\end{align}
Therefore, with \eqref{eq26} and \eqref{eq27} in mind, it suffices to compare 
\begin{equation}\label{FRAC12}
\frac{  P  \Big(\LE(S^{1}_{1} ) \cap S^{1}_{2} = \emptyset, \ G^{1} \Big)   }{ P \Big(\LE(S^{2}_{1} ) \cap S^{2}_{2} = \emptyset, \ G^{2} \Big)   }
\qquad \mbox{with}\qquad 
\frac{  P  \Big(\LE(S^{3}_{1} ) \cap S^{3}_{2} = \emptyset, \ G^{3} \Big)    }{ P \Big(\LE(S^{4}_{1} ) \cap S^{4}_{2} = \emptyset, \ G^{4} \Big)   }. 
\end{equation}
The following proposition is similar to Prop.\ \ref{difscale}.
\begin{prop}\label{prop:5.9}
There exist universal constants $c_{1} > 0$, $c_{2} > 0$ and $q_{2} > 0$  such that for all $n \ge 1$ and $q \in (0, q_{2} )$ 
\begin{equation}\label{Goal1-1-1}
P  \Big(\LE(S^{1}_{1} ) \cap S^{1}_{2} = \emptyset, \ G^{1} \Big) = P  \Big(\LE(S^{3}_{1} ) \cap S^{3}_{2} = \emptyset, \ G^{3} \Big)  \Big( 1 + O \big(d_{x}^{-c_{1}} 2^{- c_{2} q n} ) \Big).  
\end{equation}
Similarly,  for all $n \ge 1$ and $q \in (0, q_{2} )$
\begin{equation}\label{Goal2-1-1}
P  \Big(\LE(S^{2}_{1} ) \cap S^{2}_{2} = \emptyset, \ G^{2} \Big) = P  \Big(\LE(S^{4}_{1} ) \cap S^{4}_{2} = \emptyset, \ G^{4} \Big)  \Big( 1 + O \big(d_{x}^{-c_{1}} 2^{- c_{2} q n} ) \Big). 
\end{equation}
\end{prop}
\begin{proof}

We will show that the numerator of the first fraction of \eqref{FRAC12} is well approximated by that of the second fraction by using the same idea as in the proof of Proposition \ref{difscale}. We first couple $S^{3}_{2}$ with the Brownian motion $B^{3} (t)$ started at $y^{3}_{q}$ so that the Hausdorff distance between $S^{3}_{2} [0, T^{n-1}]$ and $B^{3} [0, T^{n-1} ]$ is less than $2^{{2n}/{3}}$ with probability at least $1 - c' \exp \{ - 2^{cn} \}$ for some universal constants $c, c' > 0$ (this is possible by Lemma 3.1 of \cite{Lawcut}). We write $B^{3} = B^{3} [0, T^{n-1} ]$ for the trace of the Brownian motion. 

Take $\epsilon, \delta$ and $\delta_{2}$ are the constants as in the proof of Theorem 5 of \cite{Koz} for that case of $G^{1} = \mathbb{Z}^{3}$ and $G^{2} = 2 \mathbb{Z}^{3}$ in the statement of the theorem. (For some technical reason, we assume $\epsilon < c_{4}$ where $c_{4}$ is a universal constant coming from Lemma 3.3 of \cite{Koz}.) Note that these three constants are universal. Taking these three universal constants, let $\rho = \frac{1}{10} \cdot \min \{ \epsilon, \delta, \delta_{2} \}$ and write 
\begin{align*}
&B^{3, 1} = \Big\{ x \in \mathbb{R}^{3} \ \Big| \ \text{ there exists } y \in 2B^{3} \text{ such that } |x -y| \le 2^{\frac{3n}{4}} + 20 \cdot 2^{(1- \rho) n } \Big\} \\
&B^{3, 2} = \Big\{ x \in \mathbb{R}^{3} \ \Big| \ \text{ there exists } y \in 2B^{3} \text{ such that } |x -y| \le 2^{\frac{3n}{4}}  \Big\} 
\end{align*}
for sets of points within a distance $2^{\frac{3n}{4}} + 20 \cdot 2^{(1- \rho) n }$ and $2^{\frac{3n}{4}}$ of $2B^{3}$ (i.e., Wiener sausages). Let 
\begin{equation*}
A= \Big\{ B^{3} \cap B \big( x^{3}_{q}, 2^{(1-15 q ) n } \big) = \emptyset \Big\}.
\end{equation*}
Write $\tilde{S}^{1}_{1}$ for the simple random walk on $2 \mathbb{Z}^3$ started at $x^{1}_{q}$. We also write $\tilde{S}^{1}_{1} = \tilde{S}^{1}_{1} [0, \tilde{u}^{1} ]$ where $\tilde{u}^{1} = T^{n}_{\tilde{S}^{1}_{1}} \wedge T^{(1-q)n}_{\tilde{S}^{1}_{1}}$. Set $\tilde{G}^{1} = \{ \tilde{u}^{1} = T^{(1-q)n}_{\tilde{S}^{1}_{1}} \}$.

As in the proof of Proposition \ref{difscale}, having conditioned $B^{3}$ on on $A$, we will compare 
$$
P_{S^{1}_{1}} \Big(\LE( S^{1}_{1} ) \cap B^{3,1} = \emptyset, \ G^{1} \Big)\mbox{ with }
P_{\tilde{S}^{1}_{1}} \Big(\LE( \tilde{S}^{1}_{1} ) \cap B^{3,2} = \emptyset, \  \tilde{G}^{1} \Big)
$$
for sufficiently small $q >0$ via Theorem 5 of \cite{Koz}. For this purpose, take $q_{1} := {\rho}/{100}$. We assume $ q \in (0, q_{1} )$.  We now apply Theorem 5 of \cite{Koz} with the parameters in the following table:

\vspace{2mm}
\begin{center}
 \begin{tabular}{|c|c|c|c|c|c|c|} 
\hline
Theorem 5 of \cite{Koz} & $G^{1}$ & $G^{2}$  & ${\cal D}$ & $s$  & $s a$ & $s {\cal E}$ \\ \hline
Here & $\mathbb{Z}^{3}$ & $2 \mathbb{Z}^{3}$ & $\mathbb{D} \setminus  \mathbb{D}_{q, n}$ & $2^{n}$ & $x^{1}_{q}$ & $2^{n} \mathbb{D} \setminus \big( B^{3,1} \cup D \big)$ \\ 
\hline
\end{tabular}
\end{center}
where $\mathbb{D}_{q, n } = 2^{-q n}  \mathbb{D}$ and 
\begin{equation*}
D = \Big\{ x \in \mathbb{R}^{3} \ \Big| \  \text{dist} \Big( x, \partial \big( 2^{n} \mathbb{D} \big) \Big) \le 2^{\frac{3n}{4}} + 20 \cdot 2^{(1- \rho) n } \Big\}.
\end{equation*}
Then it follows that there exist universal constants $C_{0} < \infty , c_{0} > 0$ such that if $B^{3}$ satisfies $A$, we have
\begin{equation}\label{KOZMA1}
P_{S^{1}_{1}} \Big(\LE( S^{1}_{1} ) \cap \big( B^{3,1} \cup D \big) = \emptyset \Big) \le P_{\tilde{S}^{1}_{1}} \Big(\LE( \tilde{S}^{1}_{1} ) \cap \big( B^{3,2} \cup D' \big) = \emptyset \Big) + C_{0} 2^{-c_{0} n },
\end{equation}
where 
\begin{equation*}
D' = \Big\{ x \in \mathbb{R}^{3} \ \Big| \  \text{dist} \Big( x, \partial \big( 2^{n} \mathbb{D} \big) \Big) \le 2^{\frac{3n}{4}}  \Big\}.
\end{equation*}
We note that we can take $C_{0} $ and $c_{0}$ as universal constants because if $q \le q_{1}$
\begin{itemize}
\item The LHS of (132) of \cite{Koz} is bounded above by $C 2^{-{ 4 \epsilon n}/{5}}$ for some universal constant $C$. We have the same upper bound for $|p^{7} - p^{8}|$ in line -8 page 133 of \cite{Koz}.

\item For the constants $K$ and $k$ in (137) of \cite{Koz}, we can take $K $ as a universal constant and can take $k = -1$, because the LHS of (137) of \cite{Koz} can be approximated by the probability that the coupled Brownian motion as in Section 3.4 of \cite{Koz} avoids the boundary of $\mathbb{D} \setminus \mathbb{D}_{q, n}$ even though its starting point is close to the boundary. Namely, since $q < q_{1}$ and $\epsilon < c_{4}$ (see Lemma 3.3 of \cite{Koz} for $c_{4}$), the LHS of (137) is bounded above by $C 2^{-{\epsilon n}/{2} }$ for some universal constant $C$.

\item By the same reason as above, we can take $C$ and $c$ of (138) of \cite{Koz} as universal constants. 

\item The other constants appeared in the comparison between $p_{i}$ and $p_{i+1}$ in the proof of Theorem 5 of \cite{Koz} can be taken as universal constants.

\end{itemize}

\vspace{2mm}

Given \eqref{KOZMA1}, we want to compare 
\begin{equation*}
P_{S^{1}_{1}} \Big(\LE( S^{1}_{1} ) \cap B^{3,1} = \emptyset, \ G^{1} \Big)
\end{equation*}
with the LHS of \eqref{KOZMA1}. Note that 
\begin{equation*}
P_{S^{1}_{1}} \Big(\LE( S^{1}_{1} ) \cap B^{3,1} = \emptyset, \ G^{1} \Big) =
P_{S^{1}_{1}} \Big(\LE( S^{1}_{1} ) \cap \big( B^{3,1} \cup \partial \cC_n \big) = \emptyset  \Big),
\end{equation*}
which is clearly bigger than the LHS of \eqref{KOZMA1}. The difference of the two probabilities is bounded above by
\begin{equation*}
P_{S^{1}_{1}} \Big( S^{1}_{1} \cap \partial \cC_n = \emptyset, \ S^{1}_{1} \cap D \neq \emptyset \Big) \le C 2^{- \rho n}.
\end{equation*}
Similarly, we have
\begin{equation*}
\Big| P_{\tilde{S}^{1}_{1}} \Big(\LE( \tilde{S}^{1}_{1} ) \cap \big( B^{3,2} \cup D' \big) = \emptyset \Big) - P_{\tilde{S}^{1}_{1}} \Big(\LE( \tilde{S}^{1}_{1} ) \cap B^{3,2} = \emptyset, \  \tilde{G}^{1} \Big) \Big| \le C 2^{- \frac{n}{4}}.
\end{equation*}

\vspace{2mm}
Consequently, we see that there exist universal constants $C_{1}$ and $c_{1}$ such that if we condition $B^{3}$ on $A$ and if $q \le q_{1}$, 
\begin{equation}\label{KOZMA2} 
P_{S^{1}_{1}} \Big(\LE( S^{1}_{1} ) \cap B^{3,1} = \emptyset, \ G^{1} \Big) \le P_{\tilde{S}^{1}_{1}} \Big(\LE( \tilde{S}^{1}_{1} ) \cap B^{3,2} = \emptyset, \  \tilde{G}^{1} \Big) + C_{1} 2^{-c_{1} n}.
\end{equation}
Given \eqref{KOZMA2}, the remaining part can be dealt with the same argument as in the proof of Proposition \ref{difscale}. We can replace the Wiener sausages $B^{3, 1}$ and $B^{3, 2}$ with $S^{1}_{2}$ and $S^{3}_{2}$. Namely, for $q \le q_{1}$
\begin{equation}\label{KOZMA3} 
P \Big(\LE( S^{1}_{1} ) \cap S^{1}_{2} = \emptyset, \ G^{1} \Big) \le P \Big(\LE( S^{3}_{1} ) \cap S^{3}_{2} = \emptyset, \  G^{3} \Big) + C_{1} 2^{-c_{1} n} + C 2^{-10 q n }.
\end{equation}
Thus, if we let $q_{2} = q_{1} \wedge \frac{c_{1}}{10}$ then we have for $q \le q_{2}$
\begin{equation*}
P \Big(\LE( S^{1}_{1} ) \cap S^{1}_{2} = \emptyset, \ G^{1} \Big) \le P \Big(\LE( S^{3}_{1} ) \cap S^{3}_{2} = \emptyset, \  G^{3} \Big) +  C 2^{-10 q n }.
\end{equation*}
Similarly we have 
\begin{equation*}
 P \Big(\LE( S^{3}_{1} ) \cap S^{3}_{2} = \emptyset, \  G^{3} \Big) \le P \Big(\LE( S^{1}_{1} ) \cap S^{1}_{2} = \emptyset, \ G^{1} \Big) +  C 2^{-10 q n }.
\end{equation*}
However, we know that 
\begin{equation*}
P \Big(\LE( S^{1}_{1} ) \cap S^{1}_{2} = \emptyset, \ G^{1} \Big) \ge c d_{x}^{1- \alpha} 2^{- q (1 + 5 \alpha ) n }.
\end{equation*} 
Since $\alpha \in [\frac{1}{3}, 1)$, we conclude with \eqref{Goal1-1-1} as desired. We also obtian \eqref{Goal2-1-1} through an easy modification.
\end{proof}

We are now ready to prove the main result of this subsection.
\begin{proof}[Proof of Prop.\ \ref{PROP}]
Combining \eqref{eq26}, \eqref{eq27}, \eqref{eq29}, \eqref{eq31}, \eqref{Goal1-1-1} and \eqref{Goal2-1-1}, it follows that there exist universal constants $c >0$, $\delta >0$ and $q_{0} > 0$ (in fact, we can talk $q_{0} = {q_{2}}/{2}$ of Prop.\ \ref{prop:5.9}) such that for all $n$ and $x \in \mathbb{D} \setminus \{ 0 \}$, if we let $q = q_{0}$,
\begin{align}
g_{n, x} &\;= \frac{ P (F^{1} ) P  \Big(\LE(S^{1}_{1} ) \cap S^{1}_{2} = \emptyset, \ G^{1} \Big)  P (G^{1})^{-1}  }{P(F^{2} ) P \Big(\LE(S^{2}_{1} ) \cap S^{2}_{2} = \emptyset, \ G^{2} \Big) P (G^{2})^{-1}  } \cdot \Big\{  1 + O \Big( d_{x}^{-c} 2^{- \delta q_{0} n} \Big) \Big\} \notag \\
&\;= \frac{ P (F^{3} ) P  \Big(\LE(S^{3}_{1} ) \cap S^{3}_{2} = \emptyset, \ G^{3} \Big)  P (G^{3})^{-1}  }{P(F^{4} ) P \Big(\LE(S^{4}_{1} ) \cap S^{4}_{2} = \emptyset, \ G^{4} \Big) P (G^{4})^{-1}  } \cdot \Big\{  1 + O \Big( d_{x}^{-c} 2^{- \delta q_{0} n} \Big) \Big\} \notag =g_{n-1, x}  \cdot \Big\{  1 + O \Big( d_{x}^{-c} 2^{- \delta q_{0} n} \Big) \Big\}. 
\end{align}
This finishes the proof.
\end{proof}

\end{document}